\documentclass{article}

\usepackage{amsmath}
\usepackage{color}

\usepackage{amsmath}
\usepackage{amscd}
\usepackage{amssymb}
\usepackage{enumerate}
\usepackage{indentfirst}
\usepackage{latexsym}
\usepackage{multicol}
\usepackage{color}

\usepackage{graphicx}

\newtheorem{definition}{Definition}
\newtheorem{lemma}{Lemma}
\newtheorem{remark}{Remark}
\newtheorem{theorem}{Theorem}
\newtheorem{proposition}{Proposition}
\newenvironment{proof}{\textit{Proof. }}{\hfill$\Box$}

\newcommand{\Ov}[1]{\overline{#1}}

\newcommand{\vr}{\varrho}
\newcommand{\vt}{\vartheta}
\newcommand{\vu}{\vc{u}}
\newcommand{\vre}{\vr_\ep}
\newcommand{\vte}{\vt_\ep}
\newcommand{\vue}{\vu_\ep}
\newcommand{\vrd}{\vr_\delta}
\newcommand{\vtd}{\vt_\delta}
\newcommand{\vud}{\vu_\delta}
\newcommand{\vc}[1]{{\bf #1}}
\newcommand{\vcg}[1]{\boldsymbol{#1}}
\newcommand{\Div}{{\rm div}\,}
\newcommand{\rot}{{\rm rot}\,}
\newcommand{\bomega}{\vcg{\omega}}
\newcommand{\dist}{{\rm dist}\,}
\newcommand{\Grad}{\nabla}
\newcommand{\tn}[1]{\mbox {\F #1}}

\newcommand{\dx}{{\rm d} {x}}
\newcommand{\dt}{\, {\rm d} {t}}
\newcommand{\dS}{{\rm d} {S}}
\newcommand{\intO}[1]{\int_{\Omega} #1 \ \dx}
\newcommand{\intdO}[1]{\int_{\partial \Omega} #1 \ {\rm d}S}

\newcommand{\AAA}{{\cal A}}
\newcommand{\BBB}{{\cal B}}
\newcommand{\LL}[2]{L_{z^{#1}\ln^{#2}(1+z)}}
\newcommand{\ep}{\varepsilon}
\font\F=msbm10 scaled 1100
\newcommand{\R}{R}

\font\FF=msbm10 scaled 800
\newcommand{\pder}[2]{\frac{\partial #1}{\partial #2}}

\begin{document}

\title{Existence of stationary weak solutions for the
heat conducting flows}

\author{Piotr B. Mucha$^1$,  Milan Pokorn\'y$^2$  and Ewelina Zatorska$^3$ }

\maketitle

\begin{center}
1. University of Warsaw \\
Institute of Applied Mathematics and Mechanics \\
Banacha 2, Warsaw, Poland \\
e-mail: {\tt p.mucha@mimuw.edu.pl}\\
2. Charles University in Prague \\
Faculty of Mathematics and Physics \\
Sokolovsk\'a 83, 186 75 Praha, 
Czech Republic \\
e-mail: {\tt pokorny@karlin.mff.cuni.cz},\\
3. Imperial College London \\
Department of Mathematics\\
London SW7 2AZ, 
United Kingdom \\
e-mail: {\tt e.zatorska@imperial.ac.uk}
\end{center}

\begin{abstract}
The steady compressible Navier--Stokes--Fourier system is considered, with either Dirichlet or Navier boundary conditions for the velocity and the heat flux on the boundary proportional to the difference of the temperature inside and outside.  In dependence on several parameters, i.e. the adiabatic constant $\gamma$ appearing in the pressure law $p(\vr,\vt) \sim \vr^\gamma + \vr \vt$  and the growth exponent in the heat conductivity, i.e. $\kappa(\vt) \sim (1+ \vt^m)$, and without any restriction on the size of the data,  the main ideas of the construction of weak and variational entropy solutions for the three-dimensional flows with temperature dependent viscosity coefficients are explained. Further, the case when it is possible to prove existence of solutions with bounded density is reviewed. The main changes in the construction of solutions for the two-dimensional flows are mentioned and finally, results for more complex systems are reviewed, where the steady compressible Navier--Stokes--Fourier equations play an important role.
\end{abstract}

\section{Introduction}

This survey paper is devoted to the study of {weak} and {variational entropy solutions} to the system of partial differential equations describing the steady flow of a heat-conducting compressible Newtonian fluid, i.e. 
we consider the  {steady compressible Navier--Stokes--Fourier system}. The fact that we want to have solutions for arbitrary large data results in necessity of dealing with weak (or variational entropy) solutions; the strong or classical solutions are not known to exist, even for arbitrarily regular data.

Note further that we must be more careful with the choice of correct boundary conditions (b.c.). Recall that we have to allow the energy exchange through the boundary as for thermally and mechanically insulated boundary the steady solutions may not exist. More precisely, considering the evolutionary system with mechanically and thermally insulated boundary and a time-independent external force, either the energy of the system grows to infinity, or the force is potential and the velocity tends to zero, the temperature to a constant, and the density solves a certain simple first order partial differential equation, see \cite{FeNo_Stab}.

The steady compressible Navier--Stokes--Fourier system attracted more attention in the last few years. Even though the existence of weak solutions to the steady compressible Navier-Stokes system (i.e. the isentropic system, or the system, where the exchange of the heat is negligible with respect to other processes) 
has been studied already in the seminal monograph \cite{Li_Book2} by the end of the last century, the first results for the system studied here go back to the end of the first decade of this century. Indeed, P.-L. Lions in his monograph considered the system of equations describing the steady flow of a heat conducting compressible fluid, however, he assumed that the density of the fluid is bounded {\it a priori}  in some $L^p$-space for $p$ sufficiently large. As we shall see later, to prove the bound of the density in a better space than $L^1$ is one of the main difficulties for our system of equations. The $L^1$-norm, i.e. the total mass of the fluid, is a quantity which must be known and hence, in a physically reasonable model, it is the only given bound for the density we may expect.

The first existence result for such formulation appeared in 2009 in the paper \cite{MuPo_CMP}. The proof was based on the technique developed for the stationary Navier--Stokes equations in the papers \cite{MuPo_Nonlinearity} and \cite{MuPo_DCDS}, 
which for the {Navier boundary conditions} for the velocity and sufficiently large adiabatic exponent $\gamma$ allowed to prove existence of solutions with bounded density and ``almost bounded'' velocity and temperature gradients,
 see Chapter 12.1.1 for more details.  Note that in this case even the internal energy balance is valid. This result was later  extended in \cite{MuPo_M3AS} to a larger interval 
for $\gamma$ and also for the homogeneous {Dirichlet boundary conditions} (replacing the internal energy balance by the total one), however the value of $\gamma$ still remained far above the largest physically reasonable value $\gamma = 5/3$, i.e. the monoatomic gas model. Note also that both above mentioned results were proved for a three-dimensional domain and for viscosities which depend neither on the temperature, nor on the density of the fluid.  The corresponding result in the two-dimensional case can be found in \cite{PePo_CMUC}.

Later on, in \cite{NoPo_JDE} and \cite{NoPo_SIMA} the authors observed that using directly the estimates from the entropy inequality one can obtain much better results, especially when additionally the viscosity coefficients depend on the temperature as $\sim (1+\vt)$. The existence of a solution for the Navier boundary conditions and any $\gamma >1$ was shown in \cite{JeNoPo_M3AS}. Note that the recently published paper \cite{Zhong_15} claims the existence of weak solutions for the homogeneous Dirichlet boundary conditions for the velocity under the same conditions which guarantee the existence of weak solutions in the case of the  Navier boundary conditions.  However, the authors of this paper are strongly convinced that the proof contains a gap in the part concerning the estimates of the density near the boundary. 

Analogous problems (for the Dirichlet boundary conditions) in two space dimensions were studied in \cite{NoPo_AppMa} and \cite{Po_JPDE} in the context of Orlicz spaces for the density. Finally, the situation when the viscosity behaves as $\sim (1+\vt^\alpha)$, $\alpha \in [0,1)$ is the subject of the forthcoming paper \cite{KrPo_Prep}. Some partial results in this direction can be found in \cite{KrNePo_ZAMP_2013}, where only the case $\gamma>\frac 32$ has been studied.

The paper is organized as follows. In the next section we introduce the model, the rheological relations as well as the thermodynamical concept that we  use. Then we introduce the notions of weak and variational entropy solutions and present the main existence results in the case of the three-dimensional domains. Next section  contains {\it a priori} estimates for our system (only for the adiabatic constant $\gamma > \frac 32$) to demonstrate why it is reasonable to consider two different definitions of the solutions. Following section   contains all the necessary mathematical tools to deal with our problem. Next we present four approximation levels for our problem and briefly explain how to prove existence for the last one and how to pass through several levels to the first one. In the subsequent section  we show {\it a priori} estimates independent of the first approximation parameter and the following section  contains the ideas of the limit passage to the original problem. Note that we restrict ourselves to the viscosity coefficients proportional to $\sim (1+\vt)$, i.e. to $\alpha =1$. 
The next part of the paper is devoted to presentation of existence results for several related systems. First we discuss some ideas for proof of the existence of more regular solutions for $\gamma >3$ in three space dimensions for constant viscosity coefficients and the Navier boundary conditions. Further we comment on the results in two space dimensions.  Finally, we briefly mention few results for more complex system as e.g. the steady flow with radiation or the steady flow of chemically reacting gaseous mixture. 

We also emphasize that results for the steady state solutions are an important step in analysis of {time periodic solutions}. Due to 
estimates constructed in \cite{MuPo_CMP} and \cite{NoPo_JDE} it was possible to prove existence of a weak time periodic solutions to system (\ref{1.1})--(\ref{1.3}) in the paper \cite{FMNP}.
The result has been generalized in \cite{AxPo} .

In the whole paper, we use standard notation for the Lebesgue space
$L^p(\Omega)$ endowed with the norm $\|\cdot\|_{p,\Omega}$ and
Sobolev spaces $W^{k,p}(\Omega)$ endowed with the norm
$\|\cdot\|_{k,p,\Omega}$. If no confusion may arise, we skip the
domain $\Omega$ in the norm. The vector-valued functions will be
printed in bold face, the tensor-valued functions with a special
font. Moreover, we will use notation $\vr \in L^p(\Omega)$,
$\vu \in L^p(\Omega;R^3)$, and $\tn{S} \in L^p(\Omega;R^{3\times
3})$. The generic constants are denoted by $C$ or $c$ and their values
may change even in the same formula or in the same line. We also
use summation convention over twice repeated indeces, from $1$ to
$N$; e.g. $u_i v_i$ means $\sum_{i=1}^N u_i v_i$, where $N=2$ or $3$.

\section{The model}

The steady flow of a compressible heat-conducting fluid in a bounded domain $\Omega \subset R^N$, $N=2$ or $3$, with sufficiently smooth boundary, can be described as follows 
\begin{equation} \label{1.1}
\Div (\vr \vu) = 0,
\end{equation}
\begin{equation} \label{1.2}
\Div (\vr \vu \otimes \vu) - \Div \tn{S} + \Grad p = \vr \vc{f},
\end{equation}
\begin{equation} \label{1.3}
\Div (\vr E \vu) = \vr \vc{f} \cdot \vu - \Div (p \vu) + \Div (\tn{S} \vu) -\Div \vc{q}.
\end{equation}

Here, $\vr \geq 0$ is the density of the fluid, $\vu$ is the velocity
field, $\tn{S}$  is the viscous part of the stress tensor, $p$ is the
pressure, $\vc{f}$ is the vector of specific external forces, $E$ is the
specific total energy, and $\vc{q}$ is the heat flux.   System
(\ref{1.1})--(\ref{1.3}) will be endowed with the boundary conditions on $\partial \Omega$
\begin{equation} \label{1.4}
\vu = \vc{0},
\end{equation}
(i.e. the homogeneous Dirichlet boundary conditions for the velocity), or
\begin{equation} \label{1.4a}
\vu\cdot \vc{n} = 0, \qquad (\tn{S} \vc{n}) \cdot \vcg{\tau} + \lambda \vu \cdot \vcg{\tau} = 0
\end{equation}
(i.e. the Navier slip boundary conditions for the velocity, where $\lambda \geq 0$ is the slip coefficient and $\vcg{\tau}$ denotes the tangent vector to the boundary), and
\begin{equation} \label{1.5}
-\vc{q} \cdot \vc{n} + L(\vt) (\vt - \Theta_0) = 0
\end{equation}
(i.e. the {Newton type boundary conditions} for the temperature; here
$\Theta_0 > 0$ is a given temperature outside $\Omega$).

We also assume that the total mass is given,
\begin{equation} \label{1.5a}
\intO {\vr}  = M >0.
\end{equation}

In what follows we specify the constitutive laws for our gas.
We will assume that the viscous part of the stress tensor obeys the {Stokes law}  for Newtonian fluids, namely
\begin{equation} \label{1.6}
\tn{S} = \tn{S}(\vt,\Grad\vu) = \mu(\vt) \Big[\nabla \vu + (\nabla \vu)^T
- \frac 2N \Div \vu \tn{I}\Big] + \xi(\vt) \Div \vu \tn{I}
\end{equation}
with $\mu(\cdot)$, $\xi(\cdot)$ continuous functions such that
\begin{equation} \label{1.7}
c_1 (1+\vt)^\alpha \leq \mu(\vt) \leq c_2(1+\vt)^\alpha, \qquad 0 \leq \xi(\vt) \leq c_2(1+\vt)^\alpha
\end{equation}
with some $0\leq \alpha \leq 1$.
Moreover, the function $\mu(\cdot)$ is additionally globally Lipschitz on $R^+_0$. 

The heat flux satisfies the {Fourier law}, i.e.
\begin{equation} \label{1.8}
\vc{q} = -\kappa(\vt) \nabla \vt,
\end{equation}
where 
\begin{equation} \label{1.9}
\displaystyle
\kappa(\cdot) \in C([0,\infty)), \qquad c_3(1+\vt^m) \leq \kappa(\vt) \leq c_4 (1+\vt^m), 
\end{equation}
with $m>0$. The coefficient $L(\vt)$ in (\ref{1.5}) satisfies
\begin{equation} \label{1.10}
L(\cdot) \in C([0,\infty)), \qquad c_5(1+\vt)^l \leq L(\vt) \leq c_6 (1+\vt)^l, \quad l\in \R.
\end{equation}

The specific total energy reads
\begin{equation} \label{1.11}
E(\vr,\vt,\vu) = \frac 12 |\vu|^2 + e(\vr,\vt),
\end{equation}
where $e(\cdot,\cdot)$ is the specific internal energy. We will
consider a gas law in the form
\begin{equation} \label{1.12}
p(\vr,\vt)= (\gamma-1) \vr e(\vr,\vt),\quad \mbox{where $\gamma
>1$}.
\end{equation}
This constitutive relation includes at least two physically relevant cases: if  $\gamma = 5/3$ it is
the generic law for the
monoatomic gases, while if  $\gamma = 4/3$ it describes the so-called relativistic gas, cf. \cite{EGH_Book}.

In agreement with the second law of thermodynamics, we postulate
the existence of a differentiable function $s(\vr,\vt)$ called the specific entropy which is (up to an additive
constant) given by the {Gibbs relation}
\begin{equation} \label{G}
\frac{1}{\vartheta} \Big(\vc{D}e(\vr,\vt) + p(\vr,\vt)\vc{D}\Big(\frac{1}{\vr}\Big)\Big) = \vc{D}s(\vr,\vt).
\end{equation}
Due to (\ref{G}) and (\ref{1.1})--(\ref{1.3}), the specific entropy
obeys the entropy equation
\begin{equation}\label{EI}
{\rm div}(\vr s\vu) +{\rm div}\Big(\frac{\vc{ q}}\vt\Big)=\frac
{\tn{S}:\nabla\vu}\vt-\frac {\vc{ q}\cdot\nabla\vt}{\vt^2}.
\end{equation}

It is easy to verify that the functions $p$ and $e$ are compatible
with the existence of entropy if and only if they satisfy the
{Maxwell relation}
\begin{equation} \label{1.19}
\pder{e(\vr,\vt)}{\vr} = \frac{1}{\vr^2} \Big(p(\vr,\vt) - \vt
\pder{p(\vr,\vt)}{\vt}\Big).
\end{equation}
Consequently, if $p \in C^1((0,\infty)^2)$, then it has necessarily
the form
\begin{equation} \label{1.14}
p(\vr,\vt) = \vt^{\frac{\gamma}{\gamma-1}} P\Big(\frac{\rho}{\vt^{\frac {1}{\gamma-1}}}\Big),
\end{equation}
where $P \in C^1(0,\infty)$.

We will assume that
\begin{equation} \label{1.18}
\begin{array}{c}
\displaystyle P(\cdot) \in C^1([0,\infty)) \cap C^2(0,\infty), \\
\displaystyle P(0) = 0, \quad P'(0) = p_0 >0, \quad P'(Z) >0, \quad Z>0,\\
\displaystyle \lim_{Z \to \infty} \frac{P(Z)}{Z^{\gamma}} = p_\infty >0, \\[10pt]
\displaystyle 0 < \frac{1}{\gamma-1} \frac{\gamma P(Z) - Z P'(Z)}{Z} \leq c_7 <\infty, \quad Z>0.
\end{array}
\end{equation}
For more details about (\ref{1.14}) and about physical motivation
for assumptions (\ref{1.18}) see e.g. \cite[Sections 1.4.2 and
3.2]{FeNo_Book}.  The consequences of these assumptions  are listed below.

Exactly the same results, modulo minor
modifications in the proofs, can be obtained with the constitutive
laws
\begin{equation}\label{1.18a}
p(\vr,\vt)=\vr^\gamma+\vr\vt,\quad e(\vr,\vt)=\frac
1{\gamma-1}\vr^{\gamma-1}+ c_v\vt,\;\mbox{with}\; c_v>0,
\end{equation}
whose physical relevance is discussed in \cite{Fe_Book}.
   
We will need several elementary properties of the functions
$p(\vr,\vt)$, $e(\vr,\vt)$ and the entropy $s(\vr,\vt)$. They 
follow more or less directly from  (\ref{1.12})--(\ref{1.18}). We will only
list them referring to \cite[Section 3.2]{FeNo_Book} for more details. Therein, the case $\gamma = \frac 53$ is considered, however, the computations for general $\gamma >1$ are exactly the same.

We have for $K$ a fixed constant
\begin{equation} \label{1.24}
\begin{array}{lclcl}
\displaystyle
c_8 \vr \vt &\leq & p(\vr,\vt) & \leq & c_9 \vr\vt, \mbox{ for } \vr \leq K \vt^{\frac{1}{\gamma-1}}, \\
c_{10} \vr^{\gamma} &\leq & p(\vr,\vt) & \leq & c_{11} \left\{\begin{array}{ll}
\vt^{\frac{\gamma}{\gamma-1}}, &  \mbox{ for } \vr \leq K \vt^{\frac{1}{\gamma-1}}, \\
\vr^\gamma, & \mbox{ for } \vr > K \vt^{\frac{1}{\gamma-1}}.
\end{array}\right.
\end{array}
\end{equation}
Further
\begin{equation} \label{1.25a}
\begin{array}{c}
\displaystyle \pder{p(\vr,\vt)}{\vr} > 0 \qquad \mbox{ in } (0,\infty)^2, \\
\displaystyle p = d \vr^\gamma +p_m(\vr,\vt) , \quad d>0, \qquad \mbox{ with } \quad \pder
{p_m(\vr,\vt)}{\vr} > 0 \qquad \mbox { in } (0,\infty)^2.
\end{array}
\end{equation}
For the specific internal energy defined by (\ref{1.12}) it follows
\begin{equation} \label{1.26}
\left.\begin{array}{c}
\displaystyle \frac{1}{\gamma-1} p_\infty \vr^{\gamma-1} \leq e(\vr,\vt) \leq c_{12} (\vr^{\gamma-1} + \vt),  \\
\displaystyle
\pder{e(\vr,\vt)}{\vr} \vr \leq c_{13} (\vr^{\gamma-1} + \vt)
\end{array}\right\} \mbox{ in } (0,\infty)^2.
\end{equation}
Moreover, for the specific entropy $s(\vr,\vt)$ defined by the Gibbs law (\ref{G}) we have
\begin{equation} \label{1.27}
\begin{array}{c}
\displaystyle \pder{s(\vr,\vt)}{\vr} = \frac{1}{\vt}\Big(-\frac{p(\vr,\vt)}{\vr^2} +
\pder{e(\vr,\vt)}{\vr}\Big) = -\frac{1}{\vr^2} \pder{p(\vr,\vt)}{\vt}, \\
\displaystyle \pder{s(\vr,\vt)}{\vartheta} = \frac{1}{\vartheta} \pder{e(\vr,\vt)}{\vartheta}
= \frac{1}{\gamma-1} \frac{\vt^{\frac{1}{\gamma-1}}}{\vr}
\Big(\gamma P\Big(\frac{\vr}{\vt^{\frac{1}{\gamma-1}}}\Big) - \frac{\vr}{\vt^{\frac{1}{\gamma-1}}}
P'\Big(\frac{\vr}{\vt^{\frac{1}{\gamma-1}}}\Big)\Big) >0.
\end{array}
\end{equation}
We also have for suitable choice of the additive constant in the definition of the specific entropy
\begin{equation} \label{1.29}
\begin{array}{lclcl}
\displaystyle |s(\vr,\vt)| & \leq & c_{14}(1+ |\ln \vr| + |\ln \vt|) \qquad & \mbox{ in } & (0,\infty)^2, \\
\displaystyle
|s(\vr,\vt)| & \leq & c_{15} (1 + |\ln \vr|)  \qquad & \mbox{ in } & (0,\infty) \times (1,\infty), \\
\displaystyle s(\vr,\vt) &\geq & c_{16} >0  \qquad & \mbox{ in } & (0,1) \times (1,\infty), \\
\displaystyle s(\vr,\vt) & \geq & c_{17}(1+ \ln \vt)  \qquad & \mbox{ in } & (0,1) \times (0,1).
\end{array}
\end{equation}

\section{Weak and variational entropy solutions. Main results}

In this section we present definitions of weak and variational entropy solutions to our problem. They differ in the following way: for the weak solution we require that our functions $(\vr,\vu,\vt)$ fulfill all equations of system (\ref{1.1})--(\ref{1.3}) in the distributional sense, while for the variational entropy solutions we do not require \eqref{1.3} to hold. Indeed, in some situations (we shall demonstrate this in the following section) we do not have sufficient regularity, hence the term $\vr |\vu|^2\vu$ from the total energy balance may not be integrable. One possible remedy is to consider the internal energy balance. We shall comment on this later; here let us only mention that the internal energy balance contains term like $\tn{S}(\vt,\Grad \vu):\Grad \vu$ which is possible to control only in $L^1(\Omega)$ and thus any limit passage in this term is difficult to perform. Therefore we shall use another possibility, namely, we replace the total energy balance by the entropy inequality. The reason why we cannot expect the entropy balance to hold is the fact that we are not able to keep equality in the limit passages in two terms and we are obliged to use the weak lower semicontinuity therein. At the first glance it looks like we generalized the definition of a solution too much. On the other hand, if we add to the entropy inequality the identity called  {\it the global total energy balance} which is the total energy balance integrated over $\Omega$ (here, the unpleasant term $\vr |\vu|^2\vu$ disappears), we end up with a system for which it is possible to show that any regular solution fulfilling three equalities (weak formulation for the continuity equation and for the balance of momentum, and the global total energy balance) together with one inequality (the entropy one) is in fact a classical solution to (\ref{1.1})--(\ref{1.3}), 
i.e. the {weak--strong compatibility} holds.

In order to simplify the situation we shall assume that our domain $\Omega$ in the case of the Navier boundary conditions is not axially symmetric. It is connected with the form of the Korn inequality valid in this case. If $\Omega$ is axially symmetric, we have to assume that $\lambda >0$ in (\ref{1.4a}); the results in this situation can be found in \cite{JeNoPo_M3AS}. We shall comment on them later.
We have also  to distinguish between the solution to the Dirichlet boundary conditions (\ref{1.4}) and the Navier boundary conditions (\ref{1.4a}). Moreover, we mostly consider only the case $N=3$. Finally, we take $\alpha =1$ in \eqref{1.7}.

We have

\begin{definition}[weak solution for the Dirichlet b.c.] \label {d 1.1}
The triple $(\vr,\vu,\vt)$ is called a weak solution to system (\ref{1.1})--(\ref{1.4}), (\ref{1.5})--(\ref{1.18}),
if $\vr \in L^{\frac {6\gamma}{5}}(\Omega)$, $\intO{\vr} = M$, $\vu \in W^{1,2}_0(\Omega;\R^3)$,
$\vt \in W^{1,r}(\Omega) \cap L^{3m}(\Omega) \cap L^{l+1}(\partial\Omega)$, $r>1$ with $\vr |\vu|^2 \in L^{\frac 65}(\Omega)$,
$\vr \vu \vt \in L^1(\Omega;\R^3)$, $\tn{S}(\vt,\Grad \vu) \vu \in L^1(\Omega;\R^3)$, $\vt^m \nabla \vt \in L^1(\Omega;\R^3)$, and
\begin{equation} \label{1.21}
\intO {\vr \vu \cdot \nabla \psi} = 0 \qquad \forall \psi \in C^1(\overline{\Omega}),
\end{equation}
\begin{equation} \label{1.22}
\begin{array}{c}
\displaystyle 
\intO {\big(-\vr (\vu\otimes \vu) : \nabla \vcg{\varphi} - p(\vr,\vt) \Div \vcg{\varphi} +
\tn{S} (\vt,\Grad \vu):\nabla \vcg{\varphi}\big)} \\
\displaystyle = \intO{\vr \vc{f} \cdot \vcg{\varphi}} \quad \forall \vcg{\varphi} \in C^1_0(\Omega;\R^3),
\end{array}
\end{equation}
\begin{equation} \label{1.23}
\begin{array}{c}
\displaystyle \intO {-\Big(\frac 12 \vr |\vu|^2 + \vr e(\vr,\vt)\Big) \vu \cdot \nabla \psi } =
\intO {\big(\vr \vc{f} \cdot \vu \psi + p(\vr,\vt) \vu \cdot \nabla \psi\big)}  \\[8pt]
\displaystyle - \intO{\big(\big(\tn{S}(\vt,\Grad \vu) \vu\big) \cdot \nabla \psi+
\kappa(\cdot,\vt) \nabla \vt \cdot \nabla \psi\big)} \\
- \intdO{
L(\vt)(\vt-\Theta_0) \psi} \qquad \forall \psi \in C^1(\overline{\Omega}).
\end{array}
\end{equation}
\end{definition}

We denote 
$$
W^{1,p}_{\vc{n}}(\Omega;\R^3) = \{\vu \in W^{1,p}(\Omega;\R^3); \vu \cdot \vc{n} = 0 \mbox{ in the sense of traces.}\}
$$
Similarly the space $C^1_{\vc{n}}(\Omega;R^3)$ contains all differentiable functions with zero normal trace at $\partial \Omega$.
Then we have 

\begin{definition}[weak solution for the Navier b.c.] \label {d 1.1a}
The triple $(\vr,\vu,\vt)$ is called a weak solution to system (\ref{1.1})--(\ref{1.3}), (\ref{1.4a})--(\ref{1.18}),
if $\vr \in L^{\frac {6\gamma}{5}}(\Omega)$, $\intO{\vr} = M$, $\vu \in W^{1,2}_{\vc{n}}(\Omega;\R^3)$,
$\vt \in W^{1,r}(\Omega) \cap L^{3m}(\Omega) \cap L^{l+1}(\partial\Omega)$, $r>1$ with $\vr |\vu|^2 \in L^{\frac 65}(\Omega)$,
$\vr \vu \vt \in L^1(\Omega;\R^3)$, $\tn{S}(\vt,\Grad \vu) \vu \in L^1(\Omega;\R^3)$, $\vt^m \nabla \vt \in L^1(\Omega;\R^3)$. Moreover, the continuity equation is satisfied in the sense as in (\ref{1.21}), and
\begin{equation} \label{1.22a}
\begin{array}{c}
\displaystyle
\intO {\big(-\vr (\vu\otimes \vu) : \nabla \vcg{\varphi} - p(\vr,\vt) \Div \vcg{\varphi} +
\tn{S} (\vt,\Grad \vu):\nabla \vcg{\varphi}\big)}  + \lambda \intdO{\vu \cdot \vcg{\varphi}} \\
\displaystyle = \intO{\vr \vc{f} \cdot \vcg{\varphi}} \quad \forall \vcg{\varphi} \in C^1_{\vc{n}}(\Omega;\R^3),
\end{array}
\end{equation}
\begin{equation} \label{1.23aa}
\begin{array}{c}
\displaystyle \intO {-\Big(\frac 12 \vr |\vu|^2 + \vr e(\vr,\vt)\Big) \vu \cdot \nabla \psi } =
\intO {\big(\vr \vc{f} \cdot \vu \psi + p(\vr,\vt) \vu \cdot \nabla \psi\big)}  \\[8pt]
\displaystyle - \intO{\big(\big(\tn{S}(\vt,\Grad \vu) \vu\big) \cdot \nabla \psi  +
\kappa(\cdot,\vt) \nabla \vt \cdot \nabla \psi\big)} \\
\displaystyle- \intdO{
L(\vt)(\vt-\Theta_0) \psi} -\lambda \intdO{|\vu|^2\psi} \,\,\, \forall \psi \in C^1(\overline{\Omega}).
\end{array}
\end{equation}
\end{definition}

\begin{definition}[variational entropy solution for the Dirichlet b.c.] \label{d 1.2}
The triple $(\vr,\vu,\vt)$ is called a variational entropy
solution to system (\ref{1.1})--(\ref{1.4}), (\ref{1.5})--(\ref{1.18}), if $\vr  \in L^\gamma
(\Omega)$, $\intO{\vr} = M$, $\vu \in W^{1,2}_0(\Omega;\R^3)$, $\vt \in
W^{1,r}(\Omega) \cap L^{3m}(\Omega) \cap L^{l+1}(\partial\Omega)$, $r>1$, with $\vr \vu \in  L^{\frac 65}(\Omega;\R^3)$, $\vr
\vt \in L^1(\Omega)$, $\vt^{-1} \tn{S}(\vt,\Grad \vu) \vu \in
L^1(\Omega)$, $L(\vt), \frac{L(\vt)}{\vt} \in
L^1(\partial\Omega)$, $\kappa(\vt)\frac{|\nabla
\vt|^2}{\vt^2} \in L^1(\Omega)$ and $
\kappa(\vt)\frac{\nabla \vt}{\vt} \in L^1(\Omega;\R^3)$. Moreover, 
equalities (\ref{1.21}) and (\ref{1.22}) are satisfied in the same
sense as in Definition \ref{d 1.1}, and we have the entropy
inequality
\begin{equation} \label{1.23a}
\begin{array}{c}
\displaystyle \intO{\Big(\frac{\tn{S}(\vt,\Grad \vu):\nabla \vu}{\vt} +
\kappa(\vt) \frac{|\nabla \vt|^2}{\vt^2}\Big)\psi} +
\intdO{ \frac{L(\vt)}{\vt}
\Theta_0 \psi} \\
\displaystyle \leq \intdO{ L(\vt)\psi}
 + \intO {\Big(  \kappa(\vt) \frac{\nabla \vt\cdot\nabla
\psi}{\vt}-\vr s(\vr,\vt)\vu\cdot \nabla \psi \Big)}
\end{array}
\end{equation}
for all non-negative $\psi \in C^1(\overline{\Omega})$, together with the global total energy balance
\begin{equation} \label{1.23ab}
\intdO{ L(\vt)(\vt-\Theta_0)}  = \intO {\vr \vc{f} \cdot \vu} .
\end{equation}
\end{definition}

Similarly as above we have 

\begin{definition}[variational entropy solution for the Navier b.c.] \label{d 1.2a}
The triple $(\vr,\vu,\vt)$ is called a variational entropy
solution to system (\ref{1.1})--(\ref{1.3}), (\ref{1.4a})--(\ref{1.18}), if $\vr  \in L^\gamma
(\Omega)$, $\intO{\vr} = M$, $\vu \in W^{1,2}_{\vc{n}}(\Omega;\R^3)$, $\vt \in
W^{1,r}(\Omega) \cap L^{3m}(\Omega) \cap L^{l+1}(\partial\Omega)$, $r>1$, with $\vr \vu \in  L^{\frac 65}(\Omega;\R^3)$, $\vr
\vt \in L^1(\Omega)$, $\vt^{-1} \tn{S}(\vt,\Grad \vu) \vu \in
L^1(\Omega)$, $L(\vt), \frac{L(\vt)}{\vt} \in
L^1(\partial\Omega)$, $\kappa(\vt)\frac{|\nabla
\vt|^2}{\vt^2} \in L^1(\Omega)$ and $
\kappa(\vt)\frac{\nabla \vt}{\vt} \in L^1(\Omega;\R^3)$. Moreover, 
equalities (\ref{1.21}) and (\ref{1.22a}) are satisfied in the same
sense as in Definition \ref{d 1.1a}, we have the entropy
inequality (\ref{1.23}) in the same sense as in Definition \ref{d 1.2}, together with the global total energy balance
\begin{equation} \label{1.23bb}
\lambda \intdO{|\vu|^2} + \intdO{ L(\vt)(\vt-\Theta_0)}  = \intO {\vr \vc{f} \cdot \vu} .
\end{equation}
\end{definition}

\begin{remark} \label{r 1.1a}
{\rm As mentioned above, any solution in the sense of Definitions \ref{d 1.2} or \ref{d 1.2a} 
which is sufficiently smooth is actually a classical solution to the corresponding problem. 
It can be shown exactly as in the case
of the evolutionary system and we refer to \cite[Chapter 2]{FeNo_Book}
for more details. Indeed, the same holds also for the weak solutions, i.e. for Definitions \ref{d 1.1} and \ref{d 1.1a}, where the proof is straightforward.}
\end{remark}

We will also need the notion of the {renormalized solution} to the continuity equation

\begin{definition}[renormalized solution to the continuity equation] \label{d 1.3}
Let $\vu \in W^{1,2}_{loc}(\R^3;\R^3)$ and $\vr \in L^{\frac 65}_{loc}(\R^3)$ solve
$$
\Div (\vr \vu) = 0 \mbox{ in } {\cal D}'(\R^3).
$$
Then the pair $(\vr,\vu)$ is called a renormalized solution to the continuity equation, if
\begin{equation} \label{1.23c}
\Div (b(\vr) \vu) + \big(\vr b'(\vr) - b(\vr)\big) \Div \vu = 0 \mbox{ in } {\cal D}'(\R^3)
\end{equation}
for all $b \in C^1([0,\infty)) \cap W^{1,\infty}(0,\infty)$ with $zb'(z) \in L^\infty(0,\infty)$.
\end{definition}
  
The main results read

\begin{theorem}[Dirichlet boundary conditions; Novotn\'y, Pokorn\'y, 2011] \label{t 1.1}
Let $\Omega \in C^2$ be a bounded domain in $\R^3$, $\vc{f} \in
L^\infty(\Omega;\R^3)$, $\Theta_0 \geq K_0 >0$ a.e. at $\partial
\Omega$, $\Theta_0 \in L^1(\partial \Omega)$. Let $\gamma >1$, $m > \max\big\{\frac 23,\frac{2}{3(\gamma-1)}\big\}$, $l=0$.
Then there exists a variational entropy solution to
(\ref{1.1})--(\ref{1.4}), (\ref{1.5})--(\ref{1.18}) in the sense of Definition \ref{d 1.2}.
Moreover, $\vr \geq 0$, $\vt >0$ a.e. in $\Omega$ and $(\vr,\vu)$ is
a renormalized solution to the continuity equation.

In addition, if $m>\max\{1, \frac {2\gamma}{3(3\gamma-4)}\}$ and $\gamma >\frac 43$, then the solution is a weak solution
in the sense of Definition \ref{d 1.1}.
\end{theorem}

\begin{theorem}[Navier boundary conditions; Jessl\'e, Novotn\'y, Pokorn\'y, 2014] \label{t 1.2}
Let $\Omega \in C^2$ be a bounded domain in $\R^3$, $\vc{f} \in
L^\infty(\Omega;\R^3)$, $\Theta_0 \geq K_0 >0$ a.e. at $\partial
\Omega$, $\Theta_0 \in L^1(\partial \Omega)$. Let $\gamma >1$, $m > \max\big\{\frac 23,\frac{2}{3(\gamma-1)}\big\}$, $l=0$.
Then there exists a variational entropy solution to (\ref{1.1})--(\ref{1.3}), (\ref{1.4a})--(\ref{1.18}) in the sense of Definition \ref{d 1.2a}.
Moreover, $\vr \geq 0$, $\vt >0$ a.e. in $\Omega$ and $(\vr,\vu)$ is
a renormalized solution to the continuity equation.

In addition, if $m>1$ and $\gamma >\frac 54$, then the solution is a weak solution
in the sense of Definition \ref{d 1.1a}.
\end{theorem}

\begin{remark}
{\rm
The same holds for the problem (\ref{1.1})--(\ref{1.11}) and (\ref{1.18a}) (i.e. with either the Dirichlet or the Navier boundary condition) with the specific entropy defined by the Gibbs
relation (\ref{G}).
}
\end{remark}

\begin{remark}
{\rm If $\Omega$ is an axially symmetric domain and $\lambda>0$ in (\ref{1.4a}), then the variational entropy solutions to problem (\ref{1.1})--(\ref{1.3}), (\ref{1.4a})--(\ref{1.18}) exist under the assumptions of Theorem \ref{t 1.2}. However, for the existence of weak solutions we need additionally $m> \frac{6\gamma}{15\gamma-16}$ for $\gamma \in (\frac 54,\frac 43]$ and $m>\frac{18-6\gamma}{9\gamma-7}$ for $\gamma \in (\frac 43,\frac 53)$. More details can be found in \cite{JeNoPo_M3AS}.}
\end{remark}

\section{A priori estimates for $\gamma > \frac 32$}  

In this section we present a priori estimates for our problem with both the homogeneous Dirichlet and the Navier boundary conditions for the velocity. Let us emphasize that these estimates will not be optimal in most of the cases.  They just illustrate that for some values of $\gamma$ and $m$ one may get estimates which indicate that the weak solution is available, however, in some situations the only hope is the variational entropy solution. The subsequent computations also indicate how one may obtain estimates for the approximate problems. Moreover, we assume that $l=0$, more precisely $L=const$. For simplicity we take in the case of the Navier boundary conditions $\lambda =0$ and assume that $\Omega$ is not axially symmetric.

We start with the entropy inequality (\ref{1.23a}) (note that for sufficiently smooth solutions it can be deduced from the total energy balance, in the case of the existence proof a certain version is available for the approximation), where we use as test function $\psi = 1$. We have
\begin{equation} \label{e1}
\intO{\Big(\kappa(\vt) \frac{|\Grad \vt|^2}{\vt^2} + \frac{1}{\vt} \tn{S}(\vt,\Grad \vu):\Grad \vu\Big)}  + \intdO{\frac{L \Theta_0}{\vt}} \leq \intdO{L}.
\end{equation}
Next we use as test function in the total energy balance $\psi=1$ and get
\begin{equation} \label{e2}
\intdO{L \vt} = \intO{\vr \vu \cdot \vc{f}} + \intdO{L\Theta_0}.
\end{equation}
Using the Korn inequality (see Lemma \ref{l 4.1} in the next section) we have from (\ref{e1})
\begin{equation} \label{e3}
\|\vu\|^2_{1,2} + \|\Grad(\vt^{m/2})\|_2^2 + \|\ln \vt\|_{1,2}^2 \leq C,
\end{equation}
while (\ref{e1}) and (\ref{e2}) together with the Sobolev embedding theorem yield
\begin{equation} \label{e4}
\|\vt\|_{3m} \leq C(1+ \|\vu\|_6 \|\vr\|_{\frac 65} \|\vc{f}\|_{\infty})  \leq C (1+\|\vr\|_{\frac 65}).
\end{equation}
It remains to estimate the density. In order to simplify the situation as much as possible at this moment, we use the estimates based on the application of 
the {Bogovskii operator} \ref{l 4.2} below. To this aim we apply as test function in (\ref{1.22}) or (\ref{1.22a})  a solution to 
$$
\begin{array}{c}
\displaystyle \Div \vcg{\varphi} = \vr^\alpha - \frac{1}{|\Omega|} \intO{\vr^\alpha} \quad \mbox{ in } \Omega \\ 
\displaystyle \vcg{\varphi} = \vc{0} \quad \mbox{ on } \partial \Omega.
\end{array}
$$
We have
\begin{equation} \label{e5}
\begin{array}{c}
\displaystyle \intO{p(\vr,\vt) \vr^\alpha} = - \intO{\vr(\vu\otimes \vu) :\Grad \vcg{\varphi}} + \intO{\tn{S}(\vr,\Grad \vu): \Grad \vcg{\varphi}} \\
\displaystyle - \intO{\vr\vc{f}\cdot \vcg{\varphi}} + \frac{1}{|\Omega|} \intO{p(\vr,\vt)} \intO{\vr^\alpha} = \sum_{i=1}^4 I_i.
\end{array}
\end{equation}  
Recalling that the density is bounded in $L^1(\Omega)$ (the prescribed total mass) and using Lemma \ref{l 4.2} below it is not difficult to check that the most restrictive terms are $I_1$ and $I_2$ leading to bounds (the details can be found in \cite{NoPo_JDE})
\begin{equation} \label{e6}
\alpha \leq \min\Big\{2\gamma-3,\frac{3m-2}{3m+2}\gamma\Big\}, \qquad \gamma > \frac 32, \, m> \frac 23.
\end{equation}
Hence under assumption (\ref{e6}) we have
\begin{equation} \label{e7}
\|\vu\|_{1,2} + \|\Grad(\vt^{m/2})\|_2 + \|\ln \vt\|_{1,2} +\|\vt\|_{3m} + \|\vr\|_{\gamma +\alpha} \leq C.
\end{equation}
Therefore we see that we have all  quantities in the weak formulation integrable (i.e., in particular, the density is bounded in $L^{2+\ep}(\Omega)$, and the term $\vr |\vu|^3$ is integrable in $L^{1+\ep}(\Omega)$) if
\begin{equation} \label{e8}
\gamma > \frac 53, \qquad m \geq 1,
\end{equation}
while all terms in the variational entropy formulation are integrable if
\begin{equation} \label{e9}
\gamma > \frac 32, \qquad m >\frac 23.
\end{equation}
Thus, under these assumptions, we may try to construct a solution to our problems. As we shall see later (cf. \cite{NoPo_JDE}), the limit passage requires one more condition, namely
\begin{equation} \label{e10}
m > \frac{2}{3(\gamma-1)},
\end{equation}
which comes into play for small $\gamma$'s. Under assumptions (\ref{e9})--(\ref{e10}) we may prove existence of variational entropy solutions while under assumptions (\ref{e8}), (\ref{e10}) we could prove existence of weak  solutions, see \cite{NoPo_JDE}. In what follows, using finer density estimates, we weaken the assumptions on $\gamma$ and $m$, i.e. we prove Theorems \ref{t 1.1} and \ref{t 1.2}.

\section{Mathematical tools}

In this section we present several well-known  results needed later in the proof of the existence of weak and variational entropy solutions. We first have

\begin{lemma}[Korn's inequality] \label{l 4.1}
Let $\vt>0$ and $\tn{S}(\vt,\nabla\vu)$ satisfy (\ref{1.6})--(\ref{1.7}) with $\alpha =1$. \newline
(i) Let  $\vu \in W^{1,2}_0(\Omega;R^3)$. Then
\begin{equation} \label{2.1}
\begin{array}{c}
\displaystyle \intO {\frac{\tn{S}(\vt,\nabla \vu) :\nabla \vu}{\vt} } \geq C \|\vu\|_{1,2}^2, \\[8pt]
\displaystyle \intO {\tn{S}(\vt,\nabla \vu) :\nabla \vu} \geq C \|\vu\|_{1,2}^2.
\end{array}
\end{equation}
(ii) Let $\Omega \in C^{0,1}$ and  $\vu \in W^{1,2}_{\vc{n}}(\Omega;R^3)$. Then
\begin{equation} \label{2.2}
\begin{array}{c}
\displaystyle \intO {\frac{\tn{S}(\vt,\nabla \vu) :\nabla \vu}{\vt} } + \int_{\partial \Omega} |\vu|^2 \, {\rm d} S \geq C \|\vu\|_{1,2}^2, \\[8pt]
\displaystyle \intO {\tn{S}(\vt,\nabla \vu) :\nabla \vu} + \int_{\partial \Omega} |\vu|^2 \, {\rm d} S \geq C \|\vu\|_{1,2}^2.
\end{array}
\end{equation}
If $\Omega$ is in addition not axially symmetric, then also \eqref{2.1} holds.
\end{lemma}

\begin{proof}
The proof of the first statement is nothing but integration by parts, see e.g. \cite{NoPo_JDE}. The proof of the second statement can be found e.g. in \cite{JeNo_JMPA} or \cite{JeNoPo_M3AS}.
\end{proof}

Further we need special solutions to the following problem:
\begin{equation} \label{2.3}
\begin{array}{c}
\Div \vcg{\varphi} = f \quad \mbox{ in } \Omega,\\
\vcg{\varphi} = \vc{0}\quad \mbox{ on } \partial \Omega.
\end{array}
\end{equation}

We have (see e.g. \cite{NoSt_Book})
\begin{lemma}[Bogovskii operator] \label{l 4.2}
Let $f \in L^p(\Omega)$, $1<p<\infty$, $\intO{f} = 0$, $\Omega \in C^{0,1}$. Then there exists a solution to \eqref{2.3} and a constant $C>0$ independent of $f$ such that
\begin{equation} \label{2.6}
\|\vcg{\varphi}\|_{1,p} \leq C \|f\|_p.
\end{equation}
Moreover, the solution operator $B$: $f \mapsto \vcg{\varphi}$ is linear.
\end{lemma}
 
Looks from new perspectives at solutions to system (\ref{2.3}) can be found in \cite{DaMu} or in \cite{PaPil}. 
 
Next we recall several technical results needed in the part dealing with the strong convergence for the density. We denote for $v$ a scalar function
\begin{equation} \label{2.4a}
({\cal R}[v])_{ij} = ((\nabla \otimes \nabla)\Delta^{-1})_{ij} v = {\cal F}^{-1}\Big[\frac{\xi_i \xi_j}{|\xi|^2}{\cal F}(v)(\xi)\Big],
\end{equation}
and for $\vc{u}$ a vector-valued function
\begin{equation} \label{2.4b}
({\cal R}[\vc{u}])_{i} = ((\nabla \otimes \nabla)\Delta^{-1})_{ij} u_j = {\cal F}^{-1}\Big[\frac{\xi_i \xi_j}{|\xi|^2}{\cal F}(u_j)(\xi)\Big],
\end{equation} 
with ${\cal F}(\cdot)$ the Fourier transform. We have (see \cite[Theorems 10.27, 10.28 and 10.19]{FeNo_Book})

\begin{lemma}[Commutators I] \label{l 4.3}
\medskip
Let $\vc{U}_\delta \rightharpoonup \vc{U}$ in
$L^p(\R^3;\R^3)$, $v_\delta \rightharpoonup v$ in
$L^q(\R^3)$, where
$$
\frac 1p + \frac 1q = \frac 1s <1.
$$
Then
$$
v_\delta {\cal R}[\vc{U}_\delta] - {\cal
R}[{v}_\delta] \vc{U}_\delta \rightharpoonup v {\cal
R}[\vc{U}] - {\cal R}[{v}] \vc{U}
$$
in $L^s(\R^3;\R^3)$.
\end{lemma}

\begin{lemma}[Commutators II] \label{l 4.4}
\medskip
Let $w \in W^{1,r}(\R^3)$, $\vc{z} \in L^p(\R^3;\R^3)$, $1<r<3$,
$1<p<\infty$, $\frac 1r + \frac 1p -\frac 13 <\frac 1s <1$. Then
for all such $s$ we have
$$
\|{\cal R}[w\vc{z}] - w{\cal R}[\vc{z}] \|_{a,s,R^3} \leq C
\|w\|_{1,r,R^3} \|\vc{z}\|_{p,R^3},
$$
where $\frac a 3 = \frac 1s + \frac 13 -\frac 1p - \frac 1r$. Here, $\|\cdot\|_{a,s, R^3}$ denotes the norm in
the Sobolev--Slobodetskii space $W^{a,s}(\R^3)$.
\end{lemma}

\begin{lemma}[Weak convergence and monotone operators] \label{l 4.5}
Let the couple of non-decreasing functions $(P,G)$ be in $C(R) \times C(R)$. 
  Assume that $\vr_n\in
L^1(\Omega)$ is a sequence such that
$$
\left.\begin{array}{c}
P(\vr_n) \rightharpoonup \overline{P(\vr)}, \\
G(\vr_n) \rightharpoonup \overline{G(\vr)}, \\
P(\vr_n)G(\vr_n) \rightharpoonup \overline{P(\vr)G(\vr)}
\end{array} \right\} \mbox{ in } L^1(\Omega).
$$
Then
$$
\overline{P(\vr)}\, \, \overline{G(\vr)} \leq
\overline{P(\vr)G(\vr)}
$$
a.e. in $\Omega$.
\end{lemma} 

We also have (see e.g. \cite[Lemma 2.8]{JeNoPo_M3AS})
\begin{lemma} \label{l 4.6}
Let $\Omega$ be bounded,  $f_n \rightharpoonup f$ in $L^1(\Omega)$,  $g_n \to g$ in $L^1(\Omega)$ and  $f_n g_n \rightharpoonup h$ in $L^1(\Omega)$. Then $h = fg$. 
\end{lemma}

We need the following version of the Schauder fixed point theorem
(for the proof see e.g. \cite[Theorem 9.2.4]{Ev_Book}).
\begin{lemma} \label{l 4.7}
Let ${\cal T}: X \to X$ be a continuous, compact mapping, $X$ be a Banach space. Let for any $t \in [0,1]$ the fixed points $t {\cal T}u = u$ be bounded. Then ${\cal T}$ possesses at least one fixed point in $X$.
\end{lemma}

\section{Approximation}

\subsection{Approximate system level 4}
Let us now introduce the approximating procedure. For simplicity we consider the Dirichlet boundary conditions. The proof for the Navier boundary conditions is basically the same. We also set
immediately $l=0$ and assume $L$ to be constant. Recall that we have $\alpha =1$ in \eqref{1.7}. The approximation for $l\neq 0$ and $\alpha <1$ can be done similarly. 
We fix $N$ a positive integer and $\varepsilon$, $\delta$ and $\eta >0$ (we pass subsequently $N \to \infty$, $\eta \to 0^+$, $\varepsilon \to 0^+$ and finally  $\delta \to 0^+$, thus the assumption $\varepsilon$ sufficiently small with respect to $\delta$ does not cause any problems) and denote by
$$
X_N = \mbox{span}\, \{\vc{w}^1, \dots,\vc{w}^N\} \subset W^{1,2}_0(\Omega;\R^3)
$$
with $\{\vc{w}^i\}_{i=1}^\infty$ an orthonormal basis in
$W^{1,2}_0(\Omega;\R^3)$. Due  to the smoothness of
$\Omega$ we may additionally assume that $\vc{w}^i \in
W^{2,q}(\Omega;\R^3)$ for all $1\leq q <\infty$ (we may take e.g.
the eigenfunctions of the Laplace operator with the
homogeneous Dirichlet boundary conditions). Similarly we may proceed for the slip boundary conditions, we only replace the eigenfunction to the Laplace operator by e.g. the eigenfunctions to the Lam\'e system with the Navier boundary conditions.  

We look for a triple $(\vr_{N,\eta,\varepsilon,\delta}, \vu_{N,\eta,\varepsilon,\delta}, \vt_{N,\eta,\varepsilon,\delta})$ (denoted briefly $(\vr,\vu,\vt)$) such that $\vr \in W^{2,q}(\Omega)$, $\vu \in X_N$ and $\vt \in W^{2,q}(\Omega)$, $1\leq q <\infty$ arbitrary, where
\begin{equation} \label{5.2}
\begin{array}{c}
\displaystyle \intO {\Big(\frac 12 \vr (\vu \cdot \nabla \vu) \cdot \vc{w}^i - \frac 12 \vr (\vu \otimes \vu) : \nabla \vc{w}^i + \tn{S}_\eta(\vt,\vu) : \nabla \vc{w}^i\Big)} \\
\displaystyle -\intO { \big(p(\vr,\vt)+\delta (\vr^\beta+\vr^2) \big)\Div \vc{w}^i}  = \intO {\vr \vc{f} \cdot \vc{w}^i}
\end{array}
\end{equation}
for all $i = 1,2,\dots,N$,
\begin{equation} \label{5.3}
\varepsilon \vr -\varepsilon \Delta \vr + \Div (\vr \vu) = \varepsilon h \qquad \mbox{ a.e. in } \Omega,
\end{equation}
and
\begin{equation} \label{5.4}
\begin{array}{c}
\displaystyle
-\Div \Big((\kappa_{\eta}(\vt) + \delta \vt^B + \delta \vt^{-1}) \frac{\varepsilon + \vt}{\vt}\nabla \vt\Big) + \Div \big(\vr e(\vr,\vt) \vu\big)  \\
\displaystyle  =\tn{S}_\eta(\vt,\vu):\nabla \vu + \delta \vt^{-1} -p(\vr,\vt) \Div \vu + \delta \varepsilon|\nabla \vr|^2(\beta \vr^{\beta-2} + 2) \quad \mbox{ a.e. in } \Omega,
\end{array}
\end{equation}
with $\beta \geq \max\{8,3\gamma, \frac{3m+2}{3m-2}\}$, $B\geq 2m+2$, $B \leq 6\beta -8$,
$$
\tn{S}_\eta(\vt,\vu)  = \frac{\mu_\eta(\vt)}{1+\eta \vt} \Big[\nabla
\vu + (\nabla \vu)^T - \frac 23 \Div \vu \tn{I}\Big] +
\frac{\xi_\eta(\vt)}{1+\eta \vt} \Div \vu \tn{I}.
$$
In the above formulas, $h = \frac{M}{|\Omega|}$, $\mu_\eta$,
$\xi_\eta$ and $\kappa_\eta$ are suitable regularizations of $\mu$,
$\xi$ and $\kappa$, respectively, that conserve (\ref{1.7}) and (\ref{1.9}) and that converge uniformly on compact
subsets of $[0,\infty)$ to $\mu$, $\xi$ and $\kappa$, respectively.
We consider system (\ref{5.2})--(\ref{5.4}) together with the
following boundary conditions on $\partial \Omega$
\begin{equation} \label{5.5}
\pder{\vr}{\vc{n}} = 0,
\end{equation}
\begin{equation} \label{5.6}
\big(\kappa_{\eta}(\vt) + \delta \vt^B + \delta \vt^{-1}\big) \frac{\varepsilon + \vt}{\vt} \pder{\vt}{\vc{n}} + \big(L + \delta \vt^{B-1}) (\vt - \Theta^\eta_0) + \varepsilon \ln\vt = 0,
\end{equation}
with $\Theta_0^\eta$ a smooth approximation of $\Theta_0$ such that
$\Theta_0^\eta$ is strictly positive at $\partial \Omega$. The
no-slip boundary condition for the approximate velocity is
included in the choice of $X_N$. We have

\begin{proposition} \label{p 5.1}
Let $\varepsilon$, $\delta$, $\eta$ and $N$ be as above, $\beta \geq \max\{8,2\gamma\}$ and $B\geq 2m+2$. Let $\varepsilon$ be sufficiently small with respect to $\delta$. Under the assumptions of Theorem \ref{t 1.1} and the assumptions made above in this section, there exists a solution to system (\ref{5.2})--(\ref{5.6}) such that $\vr \in W^{2,q}(\Omega)$ $\forall q<\infty$, $\vr \geq 0$ in $\Omega$, $\intO \vr = M$, $\vu \in X_N$, and $\vt \in W^{2,q}(\Omega)$ $\forall q <\infty$, $\vt \geq C(N) >0$.
\end{proposition}

The detailed proof of the proposition is in  \cite{NoPo_JDE}. Let us only recall the main steps here. We consider a mapping
$$
T: X_N \times W^{2,q}(\Omega) \to X_N \times W^{2,q}(\Omega)
$$
with
$$
T(\vc{v},z) = (\vu,r),
$$
where
\begin{equation} \label{5.7}
\begin{array}{c}
\displaystyle
\intO {\tn{S}_\eta(\mbox{e}^z,\vu) : \nabla \vc{w}^i } =\\
\displaystyle  \intO{\Big(\frac 12 \vr (\vc{v} \otimes \vc{v}) : \nabla \vc{w}^i - \frac 12 \vr (\vc{v} \cdot \nabla \vc{v}) \cdot \vc{w}^i +\big(p(\vr,\mbox{e}^z)+\delta (\vr^\beta +  \vr^2)\big)\Div \vc{w}^i + \vr \vc{f} \cdot \vc{w}^i\Big)}
\end{array}
\end{equation}
$\forall i = 1,2,\dots,N$,
\begin{equation} \label{5.8}
\begin{array}{c}
\displaystyle
-\Div \Big((\kappa_{\eta}(\mbox{e}^z) + \delta \mbox{e}^{zB} + \delta \mbox{e}^{-z}) (\varepsilon + \mbox{e}^z)\nabla r \Big) = - \Div \big(\vr e(\vr,\mbox{e}^z) \vc{v}\big) + \tn{S}_\eta(\mbox{e}^z,\vc{v}):\nabla \vc{v}\\
\displaystyle   + \delta \mbox{e}^{-z} -p(\vr,\mbox{e}^z) \Div \vc{v} + \delta \varepsilon|\nabla \vr|^2(\beta \vr^{\beta-2} + 2) \qquad \mbox{ a.e. in } \Omega,
\end{array}
\end{equation}
with $\vr$, a unique solution to (see Lemma \ref{l 5.3} below)
\begin{equation} \label{5.9}
\begin{array}{c}
\displaystyle \varepsilon \vr -\varepsilon \Delta \vr + \Div(\vr \vc{v}) = \varepsilon h \qquad \mbox{ in } \Omega, \\
\displaystyle \pder{\vr}{\vc{n}} = 0 \qquad \mbox{ at } \partial \Omega,
\end{array}
\end{equation}
together with the boundary conditions on $\partial \Omega$
\begin{equation} \label{5.10}
\begin{array}{c}
\displaystyle \big(\kappa_{\eta}(\mbox{e}^z) + \delta \mbox{e}^{zB} + \delta \mbox{e}^{-z}\big) (\varepsilon + \mbox{e}^z) \pder{r}{\vc{n}} + \big(L+ \delta \mbox{e}^{(B-1)z}\big)(\mbox{e}^z -\Theta_0^n) + \varepsilon r = 0.
\end{array}
\end{equation}

Note that the fixed point of $T$ (provided it exists) corresponds to $r=\ln \vt$ in (\ref{5.2})--(\ref{5.6}).
We now apply Lemma \ref{l 4.7}.

For fixed $\vc{v}\in X_N$, we can find a unique solution to the
approximate continuity equation (\ref{5.9}). More precisely, we
have  (see e.g. \cite[Proposition 4.29]{NoSt_Book})
\begin{lemma} \label{l 5.3}
Let $\varepsilon>0$, $h = \frac{M}{|\Omega|}$. Let $\vc{v} \in X_N$. Then there exists a unique solution to (\ref{5.9}) such that $\vr \in W^{2,p}(\Omega)$ for all $p<\infty$, $\intO \vr = M$ and $\vr \geq 0$ in $\Omega$. Moreover, the mapping $S: \vc{v} \mapsto \rho$ is continuous and compact from $X_N$ to $W^{2,p}(\Omega)$.
\end{lemma}

Both (\ref{5.7}) and (\ref{5.8}) with (\ref{5.10}) are linear elliptic problems. Hence their unique solvability as well as regularity of the solution is straightforward. Using also Lemma \ref{l 5.3} we get  (see \cite[Lemma 3]{MuPo_CMP} for a similar result) the following lemma.

\begin{lemma} \label{l 5.4}
Under the assumptions of Proposition \ref{p 5.1}, for $p>3$, the operator $T$ is a continuous and compact operator from $X_N \times W^{2,p}(\Omega)$ into itself.
\end{lemma}

To fulfill the assumptions of Lemma \ref{l 4.7}, we need to verify boundedness of possible fixed points to $t T(\vu,r) = (\vu,r)$, $t \in [0,1]$. As this is the most difficult part of the proof of Proposition \ref{p 5.1}, we give more details here. The full proof can be found in \cite{NoPo_JDE}.

\begin{lemma} \label{l 5.5}
Let the assumptions of Proposition \ref{p 5.1} be satisfied. Let $p>3$. Then there exists $C>0$ such that all solutions to
\begin{equation} \label{5.10a}
t T(\vu,r) = (\vu,r)
\end{equation}
fulfill
$$
\|\vu\|_{2,p} + \|r\|_{2,p} + \|\vt\|_{2,p} \leq C,
$$
where $\vt = \mbox{e}^{r}$ and $C$ is independent of $t\in [0,1]$.
\end{lemma}

Idea of the proof:
\begin{itemize}
\item[(i)] Testing (\ref{5.2}) (in the form of (\ref{5.10a})) by $\vu$ (i.e. linear combinations of $\vc{w}^{i}$) yields
\begin{equation} \label{5.101}
\intO { \tn{S}_\eta (\vt,\vu) : \nabla \vu} = t \intO {\Big(\big(p(\vr,\vt)  + \delta ( \vr^{\beta} +\vr^2)\big)\Div\vu
+ \vr \vc{f} \cdot \vu \Big)}.
\end{equation}

\item[(ii)] Integrating (\ref{5.4}) (in the form of (\ref{5.10a})) over $\Omega$, together with (\ref{5.3}), (\ref{5.6}) and (\ref{5.101}) implies
\begin{equation} \label{5.102}
\begin{array}{c}
\displaystyle
\intdO{\Big(t \big(L+ \delta \vt^{B-1}\big)
(\vt -\Theta_0^\eta) + \varepsilon \ln \vt\Big)} +
(1-t) \intO { \tn{S}_\eta (\vt,\vu) : \nabla \vu}  \\
\displaystyle + \varepsilon \delta  t  \intO { \Big(\frac{\beta}{\beta-1}\vr^\beta  +
2 \vr^2\Big)} \\
\displaystyle = t \intO {\Big(\vr \vc{f} \cdot \vu + \varepsilon \delta \frac{\beta}{\beta-1} h  \vr^{\beta-1} + 2\varepsilon \delta h \vr +  \delta \vt^{-1}\Big)}.
\end{array}
\end{equation}

\item[(iii)] Integration over $\Omega$ of the entropy version of the approximate energy balance (\ref{5.4}) (i.e. (\ref{5.4}) divided by $\vt$,
again in the form of (\ref{5.10a})), after slightly technical
computations, yields
\begin{equation} \label{5.103}
\begin{array}{c}
\displaystyle \intO{\big(\kappa_{\eta}(\vt) + \delta \vt^{B} + \delta \vt^{-1}\big) \frac{\varepsilon + \vt}{\vt}\frac{|\nabla \vt|^2}{\vt^2}} + t \intO {\Big(\frac{1}{\vt} \tn{S}_\eta(\vt,\vu):\nabla \vu +  \delta\vt^{-2}\Big)} \\
\displaystyle + \intdO{ \frac{1}{\vt}\Big(t \big(L + \delta \vt^{B-1}\big) \Theta_0^\eta - \varepsilon \ln \vt\Big)} + t\varepsilon \delta \intO{\frac{1}{\vt}|\nabla \vr|^2 \big(\beta\vr^{\beta-2} + 2\big)}
\\[8pt]
\displaystyle \leq t \intdO {\big(L+ \delta \vt^{B-1}\big)}  + t \frac{\varepsilon}{2} \frac{\beta}{\beta-1} \intO {\vr^\beta} + Ct\varepsilon.
\end{array}
\end{equation}

\item[(iv)] Combing identities from steps (ii)--(iii) we get
\begin{equation} \label{5.104}
\begin{array}{c}
\displaystyle \intO{\big(\kappa_{\eta}(\vt) + \delta \vt^{B} + \delta \vt^{-1}\big) \frac{\varepsilon + \vt}{\vt}\frac{|\nabla \vt|^2}{\vt^2}} + t \intO {\Big(\frac{1}{\vt} \tn{S}_\eta(\vt,\vu):\nabla \vu +  \delta\vt^{-2}\Big)} \\[8pt]
\displaystyle +(1-t) \intO { \tn{S}_\eta (\vt,\vu) : \nabla \vu} + \frac 12 \varepsilon \delta  t  \intO { \Big(\frac{\beta}{\beta-1}\vr^\beta + 2 \vr^2\Big)} \\
\displaystyle + t\varepsilon \delta \intO{\frac{1}{\vt} |\nabla \vr|^2(\beta\vr^{\beta-2} + 2)}
 \\
\displaystyle
+\intdO{\Big(t \big(L\vt + \delta \vt^B\big) + \varepsilon |\ln \vt|  + t \frac{\Theta_0^\eta}{\vt} L  \Big)} 
\displaystyle   \leq C t \Big(1+ \Big|\intO {\vr \vc{f} \cdot \vu} \Big|\Big).
\end{array}
\end{equation}

\item[(v)]
Estimates (\ref{5.101}) and (\ref{5.104}) lead to
\begin{equation} \label{5.105}
\|\vu\|_{1,2} + \|\vt\|_{3B} + \|\nabla \vt\|_{2} + \|\vr\|_{\beta} \leq C
\end{equation}
with $C$ independent of $t$ (and also of $\eta$ and $N$).

\item[(vi)] Properties of $X_N$ and standard regularity results for elliptic equations imply
\begin{equation} \label{5.106}
\|\vu\|_{2,p} + \|\vr\|_{2,p} \leq C(N).
\end{equation}

\item[(vii)] Standard tools as Kirchhoff transform and regularity results for elliptic problems finally yield
\begin{equation} \label{5.107}
\|r\|_{2,p} + \|\vt\|_{2,q} \leq C(N)
\end{equation}
which finishes the proof of Lemma \ref{l 5.5} as well as of Proposition \ref{p 5.1}.
\end{itemize}

\subsection{Limit passage $N\to \infty$ towards approximate system level 3}

From the proof of Proposition \ref{p 5.1} above we can deduce the following uniform estimates
\begin{equation} \label{2.11}
\begin{array}{c}
\displaystyle
\|\vu_N\|_{1,2} + \|\vr_N\|_{\beta} + \|\vt_N\|_{3B} + \|\vt_N\|_{1,2} + \|\vt_N^{-2}\|_{1} + \|\vt_N^{-1}\|_{1,\partial \Omega} \\
\displaystyle
+ \|\vt_N^{-4} |\nabla \vt_N|^2\|_{1} + \|\vr_N\|_{2,2} \leq C(\varepsilon, \delta).
\end{array}
\end{equation}


Thus, extracting suitable subsequences if necessary, we can get a triple  $(\vr,\vu,\vt)$ being a limit of $(\vr_{N_k}, \vu_{N_k}, \vt_{N_k})$ in spaces given by estimates (\ref{2.11}), and solving

\begin{equation} \label{2.12}
\begin{array}{c}
\displaystyle
\intO {\Big(\frac 12 \vr (\vu \cdot\nabla \vu)\cdot \vcg{\varphi} - \frac 12 \vr (\vu\otimes \vu) : \nabla \vcg{\varphi} + \tn{S}_\eta(\vt,\vu):\nabla  \vcg{\varphi}\Big)} \\
\displaystyle - \intO{\big(p(\vr,\vt) + \delta \vr^\beta + \delta \vr^2 \big)\Div \vcg{\varphi}}
\displaystyle = \intO {\vr \vc{f} \cdot \vcg{\varphi}} \qquad \forall \vcg{\varphi} \in W^{1,2}_0(\Omega;\R^3),
\end{array}
\end{equation}
\begin{equation} \label{2.13}
\varepsilon \vr -\varepsilon \Delta \vr + \Div (\vr \vu) = \varepsilon h \qquad \mbox{ a.e. in } \Omega,
\end{equation}
with
\begin{equation} \label{2.14}
\pder{\vr}{\vc{n}} = 0 \qquad \mbox{ a.e. at } \partial \Omega,
\end{equation}
and
\begin{equation} \label{2.15}
\begin{array}{c}
\displaystyle \intO {\Big(\big( \kappa_\eta(\vt) + \delta \vt^B + \delta \vt^{-1}\big) \frac{\varepsilon + \vt}{\vt} \nabla \vt\cdot\nabla \psi - \vr e(\vr,\vt) \vu \cdot \nabla \psi\Big)}\\
\displaystyle  + \intdO{ \Big(\big(L + \delta \vt^{B-1}\big) (\vt-\Theta_0^\eta) + \varepsilon \ln \vt\Big) \psi} 
\\
\displaystyle = \intO {\Big(\tn{S}_\eta(\vt,\vu):\nabla \vu + \delta \vt^{-1} - p(\vr,\vt) \Div \vu + \varepsilon \delta\big|\nabla \vr|^2(\beta \vr^{\beta-2} + 2 \big)\Big) \psi},
\end{array}
\end{equation}
for all $\psi \in C^1(\overline{\Omega})$.

Note that in order to get (\ref{2.15}) we need $\tn{S}_\eta(\vt_N,\vu_N):\nabla \vu_N \to \tn{S}_\eta(\vt,\vu):\nabla \vu $ in $L^1(\Omega)$. This fact does not follow from (\ref{2.11}) but we may show it realizing that we can use as a test function in (\ref{2.12}) the limit function $\vu$, together with the limit passage in (\ref{2.2}) with $\vc{w}^i$ replaced by $\vu_N$. Last but not least, we can also get the entropy inequality
\begin{equation} \label{2.16}
\begin{array}{c}
\displaystyle \intO{\Big(\vt^{-1} \tn{S}_\eta (\vt,\vu):\nabla \vu + \delta \vt^{-2}  + \big(\kappa_\eta(\vt) + \delta \vt^B + \delta \vt^{-1}\big) \frac{\varepsilon + \vt}{\vt} \frac{|\nabla \vt|^2}{\vt^2}  \Big)\psi} \\
\displaystyle \leq \intO {\Big(\big(\kappa_\eta(\vt) + \delta \vt^B + \delta \vt^{-1}\big) \frac{\varepsilon + \vt}{\vt} \frac{\nabla \vt \cdot \nabla \psi}{\vt} - \vr s(\vr,\vt) \vu \cdot \nabla \psi \Big)} \\
\displaystyle + \intdO{ \Big(\frac{L+ \delta \vt^{B-1}}{\vt} (\vt-\Theta_0^\eta) + \varepsilon \ln \vt\Big) \psi} + F_\varepsilon,
\end{array}
\end{equation}
for all $\psi \in C^1(\overline{\Omega})$, non-negative, with
$F_\varepsilon\to 0$ as $\varepsilon \to 0^+$.

\subsection{Limit passage $\eta \to 0^+$ towards approximate system level 2}

We can use again (\ref{2.11}) to pass to the limit in the approximate continuity equation, momentum equation and entropy inequality  (switching to subsequences if necessary)
\begin{equation} \label{2.17}
\begin{array}{c}
\displaystyle
\intO {\Big(\frac 12 \vr (\vu \cdot\nabla \vu)\cdot \vcg{\varphi} - \frac 12 \vr (\vu\otimes \vu) : \nabla \vcg{\varphi} + \tn{S}(\vt,\vu):\nabla  \vcg{\varphi}\Big)}\\
\displaystyle +\intO {\Big( \big(p(\vr,\vt) + \delta (\vr^\beta + \vr^2)\big) \Div \vcg{\varphi} + \vr \vc{f} \cdot \vcg{\varphi}\Big)} \qquad \forall \vcg{\varphi} \in W^{1,\frac{6B}{3B-2}}_0(\Omega;\R^3),
\end{array}
\end{equation}
\begin{equation} \label{2.18}
\varepsilon \intO {(\vr \psi + \nabla \vr \cdot \nabla \psi)} - \intO {\vr \vu \cdot \nabla \psi} = \varepsilon h \intO \psi \qquad \forall \psi \in W^{1,\frac 65}(\Omega),
\end{equation}
\begin{equation} \label{2.19}
\begin{array}{c}
\displaystyle \intO{\Big(\vt^{-1} \tn{S} (\vt,\vu):\nabla \vu + \delta \vt^{-2}  + \big(\kappa(\vt) + \delta \vt^B + \delta \vt^{-1}\big) \frac{\varepsilon + \vt}{\vt} \frac{|\nabla \vt|^2}{\vt^2}  \Big)\psi} \\
\displaystyle \leq \intO {\Big(\big(\kappa(\vt) + \delta \vt^B + \delta \vt^{-1}\big) \frac{\varepsilon + \vt}{\vt} \frac{\nabla \vt \cdot \nabla \psi}{\vt} - \vr s(\vr,\vt) \vu \cdot \nabla \psi \Big)} \\
\displaystyle + \intdO{ \Big(\frac{L+ \delta \vt^{B-1}}{\vt} (\vt-\Theta_0) + \varepsilon \ln \vt\Big) \psi}  + F_\varepsilon,
\end{array}
\end{equation}
for all $\psi \in C^1(\overline{\Omega})$, non-negative, with
$F_\varepsilon$ as above. The main difficulty in this step appears
in the limit passage in the energy balance. We are not anymore able
to guarantee the strong convergence of $\nabla \vu_\eta$ in
$L^2(\Omega;\R^{3\times 3})$ and thus we are not able to recover in
the limit the balance of the internal energy. However, we may consider
instead of it the balance of the total energy which we get summing the
approximate balance of the internal energy (\ref{2.15}) and the
approximate momentum equation (\ref{2.12}) tested by $\vu_\eta
\psi$ (i.e. the balance of the kinetic energy). Doing so we may now pass
with $\eta \to 0^+$ --- as the most difficult term
$\intO{\tn{S}_\eta(\vu_\eta,\vt_\eta):\nabla \vu_\eta \psi}$ is
replaced by $\intO{\tn{S}_\eta(\vu_\eta,\vt_\eta) \vu_\eta \cdot
\nabla \psi}$ --- and here the information from (\ref{2.11}) is
sufficient. We get
\begin{equation} \label{2.20}
\begin{array}{c}
\displaystyle \intO {\Big(\big(-\frac 12 \vr |\vu|^2 - \vr e(\vr,\vt)\big)\vu \cdot \nabla \psi + \big(\kappa(\vt) + \delta\vt^B + \delta \vt^{-1}\big) \frac{\varepsilon + \vt}{\vt} \nabla \vt\cdot\nabla \psi\Big)} \\
\displaystyle + \intdO{ \Big(\big(L+ \delta \vt^{B-1}\big) (\vt-\Theta_0) + \varepsilon \ln \vartheta\Big)\psi}  =
 \intO{\vr \vc{f} \cdot \vu \psi} \\
\displaystyle + \intO {\Big(\big(-\tn{S}(\vt,\vu)\vu +p(\vr,\vt)\vu + \delta( \vr^\beta + \vr^2)\vu\big) \cdot \nabla \psi + \delta \vt^{-1}\psi \Big) }   \\
\displaystyle + \delta \intO {\frac{1}{\beta-1}\Big( \varepsilon \beta h \vr^{\beta-1} \psi + \vr^\beta \vu  \cdot \nabla \psi - \varepsilon \beta \vr ^\beta \psi\Big)  } \\
\displaystyle + \delta \intO {\big( 2\varepsilon  h  \vr \psi + \vr^2 \vu  \cdot \nabla \psi - 2\varepsilon  \vr ^2 \psi\Big)  } \qquad \forall \psi \in C^1(\overline{\Omega}).
\end{array}
\end{equation}
Note that due to bounds (\ref{2.11}) the temperature is
positive a.e. in $\Omega$ and a.e. on $\partial \Omega$.

\subsection{Limit passage $\varepsilon \to 0^+$ towards approximate system level 1}

From the entropy inequality (\ref{2.19}) and the total energy
balance (\ref{2.20}), together with a version of Korn's inequality
(see Lemma \ref{l 4.1}), we can deduce the following estimates
independent of $\varepsilon$:
\begin{multline} \label{2.21}
\|\vue\|_{1,2}^2 +  \|\vte\|_{3B}^{B} + \|\vte\|_{1,2}^2 + \|\nabla (\vte^{-\frac 12})\|_2^2 + \|\vte^{-2}\|_1 + \|\vte\|_{B,\partial \Omega}^{B} + \|\vte^{-1}\|_{1,\partial \Omega}\\
 \leq C(1+ \|\vre\|_{\frac 65}^2),
\end{multline}
with $C=C(\delta)$, but independent of $\varepsilon$. The estimates
above do not contain any bound on the density.  To deduce it,  we can apply the 
Bogovskii-type estimates, meaning we employ  as test function in
(\ref{2.17}) a vector field $\vcg{\Phi}$, a solution to
\begin{equation} \label{2.22}
\begin{array}{c}
\displaystyle \Div \vcg{\Phi} = \vr_\varepsilon ^{(s-1)\beta} - \frac{1}{|\Omega|}\intO{\vr_\varepsilon^{(s-1)\beta}} \qquad \mbox{ in }\Omega, \\
\displaystyle \vcg{\Phi} = \vc{0} \qquad \mbox{ on } \partial \Omega
\end{array}
\end{equation}
with
$$
\|\vcg{\Phi}\|_{1,\frac{s}{s-1}}^{\frac{s}{s-1}} \leq C \|\vr_\varepsilon\|_{s\beta}^{s\beta}, \qquad 1<s<\infty,
$$
see Lemma \ref{l 4.2}. 
After straightforward calculations we  get
\begin{equation} \label{2.23}
\|\vr_\varepsilon\|_{\frac 53\beta} \leq C
\end{equation}
and thus, using also the approximate continuity equation with the test
function $\psi = \vre$, we obtain
\begin{equation} \label{2.24}
\|\vu_\varepsilon\|_{1,2} + \|\vt_\varepsilon\|_{3B} + \|\vt_\varepsilon\|_{1,2} + \|\vt_\varepsilon^{-\frac 12}\|_{1,2} + \|\ln \vt_{\varepsilon}\|_{1,2} + \|\vt_\varepsilon^{-1}\|_{1,\partial \Omega} + \|\vr_\varepsilon\|_{\frac 53 \beta} + \sqrt{\varepsilon} \|\nabla \vr_\varepsilon\|_2 \leq C
\end{equation}
with $C$ independent of $\varepsilon$. Note that we still miss an
information providing the compactness of the sequence approximating the density. Passing to the
limit $\varepsilon \to 0^+$ we get (switching to subsequences, if
necessary)
\begin{equation} \label{2.25}
\intO {\vr \vu \cdot \nabla \psi} = 0 \qquad \forall \psi \in W^{1,\frac{30\beta}{25\beta-18}}(\Omega),
\end{equation}
\begin{equation} \label{2.26}
\intO {\Big(-\vr (\vu\otimes \vu) :\nabla \vcg{\varphi} + \tn{S}(\vt,\vu):\nabla \vcg{\varphi} - \overline{\big(p(\vr,\vt) + \delta \vr^\beta + \delta\vr^2\big)} \Div \vcg{\varphi}\Big)} = \intO {\vr \vc{f} \cdot \vcg{\varphi}}
\end{equation}
for all $\vcg{\varphi} \in W^{1,\frac{5}{2}}_0(\Omega;\R^3)$. Here
and in the sequel, $\overline{g(\vr,\vu,\vt)}$ denotes the weak
limit of a sequence
$g(\vr_\varepsilon,\vu_\varepsilon,\vt_\varepsilon)$. Further, we
obtain
\begin{equation} \label{2.27}
\begin{array}{c}
\displaystyle \intO {\Big(\big(-\frac 12 \vr |\vu|^2 - \overline{\vr e(\vr,\vt)}\big)\vu \cdot \nabla \psi + \big(\kappa(\vt) + \delta\vt^B + \delta \vt^{-1}\big) \nabla \vt\cdot \nabla \psi\Big)} \\
\displaystyle + \intdO{ \big(L+ \delta \vt^{B-1}\big) (\vt-\Theta_0) \psi} =
 \intO{\vr \vc{f} \cdot \vu \psi} \\
\displaystyle + \intO {\Big(\big(-\tn{S}(\vt,\vu)\vu +\overline{\big(p(\vr,\vt) + \delta \vr^\beta + \delta \vr^2\big)}\, \vu \big)\cdot \nabla \psi + \delta\vt^{-1}\psi \Big) }   \\
\displaystyle + \delta \intO {\Big(\frac{1}{\beta-1}  \overline{\vr^\beta} + \overline{\vr^2}\Big) \vu  \cdot \nabla \psi}  \qquad \forall \psi \in C^1(\overline{\Omega}),
\end{array}
\end{equation}
and the entropy inequality
\begin{equation} \label{2.28}
\begin{array}{c}
\displaystyle \intO{\Big(\vt^{-1} \tn{S} (\vt,\vu):\nabla \vu + \delta \vt^{-2}  + \big(\kappa(\vt) + \delta \vt^B + \delta \vt^{-1}\big) \frac{|\nabla \vt|^2}{\vt^2}  \Big)\psi} \\
\displaystyle \leq \intO {\Big(\big(\kappa(\vt) + \delta \vt^B + \delta \vt^{-1}\big) \frac{\nabla \vt \cdot \nabla \psi}{\vt} - \overline{ \vr s(\vr,\vt)} \vu \cdot \nabla \psi \Big)} \\
\displaystyle + \intdO{ \frac{L + \delta \vt^{B-1}}{\vt} (\vt-\Theta_0) \psi } 
\end{array}
\end{equation}
for all non-negative $\psi \in C^1(\overline{\Omega})$.

In order to show the strong convergence of density (which is
sufficient to remove the bars in (\ref{2.26})--(\ref{2.28})) we can
combine technique introduced in \cite{Li_Book2}  with some
of techniques from \cite[Chapter 3]{FeNo_Book}.

Using, roughly speaking, as test function in (\ref{2.17}) $\vcg{\varphi} :=\nabla \Delta^{-1} \vre$ and in (\ref{2.26}) $\vcg{\varphi} :=\nabla \Delta^{-1} \vr$, passing to the limit $\varepsilon \to 0^+$, together with several deep results from the harmonic analysis, we end up with
\begin{equation} \label{2.29}
\begin{array}{c}
\displaystyle
\overline{\big(p(\vr,\vt) + \delta \vr^\beta + \delta \vr^2\big) \vr} - \Big(\frac 43 \mu(\vt) + \xi(\vt)\Big) \overline{\vr \Div \vu} \\
\displaystyle =  \overline{\big(p(\vr,\vt) + \delta \vr^\beta + \delta \vr^2\big)} \, \vr - \Big(\frac 43 \mu(\vt) + \xi(\vt)\Big) \vr \Div \vu
\end{array}
\end{equation}
a.e. in $\Omega$. This fact, together with the theory of renormalized solutions to continuity equation and standard properties of weakly convergent sequences, leads to
\begin{equation} \label{2.30}
\overline{\big(p(\vr,\vt)+ \delta \vr^\beta + \delta \vr^2\big)\vr} = \overline{\big(p(\vr,\vt)+ \delta \vr^\beta + \delta \vr^2\big)}\, \vr
\qquad \mbox{ a.e. in } \Omega,
\end{equation}
in particular
$$
\overline{\vr^{\beta+1}} = \overline{\vr^\beta} \vr,
$$
which implies the strong convergence of the density.
The reasoning above is somewhat similar (and simpler) than that one needed for the passage $\delta \to 0^+$.
The latter is described in more details below.

As a conclusion, (\ref{2.25})--(\ref{2.28}) can be rewritten as
\begin{equation} \label{2.31}
\intO {\vr \vu \cdot \nabla \psi} = 0
\end{equation}
 for all $\psi \in W^{1,\frac{30\beta}{25\beta-18}}(\Omega)$,
\begin{equation} \label{2.32}
\intO {\Big(-\vr (\vu\otimes \vu) :\nabla \vcg{\varphi} + \tn{S}(\vt,\vu):\nabla \vcg{\varphi} - \big(p(\vr,\vt) + \delta \vr^\beta + \delta\vr^2\big) \Div \vcg{\varphi}\Big)} = \intO {\vr \vc{f} \cdot \vcg{\varphi}}
\end{equation}
for all $\vcg{\varphi} \in W^{1,\frac 52}_0(\Omega;\R^3)$,
\begin{equation} \label{2.33}
\begin{array}{c}
\displaystyle \intO {\Big(\big(-\frac 12 \vr |\vu|^2 - \vr e(\vr,\vt)\big)\vu \cdot \nabla \psi + \big(\kappa(\vt) + \delta\vt^B + \delta \vt^{-1}\big) \nabla \vt\cdot \nabla \psi\Big)} \\
\displaystyle + \intdO{ \big(L+ \delta \vt^{B-1}\big) (\vt-\Theta_0) \psi}  =
 \intO{\vr \vc{f} \cdot \vu \psi } \\
\displaystyle + \intO {\Big(\big(-\tn{S}(\vt,\vu)\vu +\big(p(\vr,\vt) + \delta \vr^\beta + \delta \vr^2\big)\vu \big) \cdot \nabla \psi + \delta \vt^{-1} \psi \Big) }   \\
\displaystyle + \delta \intO {\Big(\frac{1}{\beta-1}  \vr^\beta + \vr^2\Big) \vu  \cdot \nabla \psi}
\end{array}
\end{equation}
for all $\psi \in C^1(\overline{\Omega})$, and
\begin{equation} \label{2.34}
\begin{array}{c}
\displaystyle \intO{\Big(\vt^{-1} \tn{S} (\vt,\vu):\nabla \vu + \delta \vt^{-2}  + \big(\kappa(\vt) + \delta \vt^B + \delta \vt^{-1}\big) \frac{|\nabla \vt|^2}{\vt^2}  \Big)\psi} \\
\displaystyle \leq \intO {\Big(\big(\kappa(\vt) + \delta \vt^B + \delta \vt^{-1}\big) \frac{\nabla \vt : \nabla \psi}{\vt} - \vr s(\vr,\vt) \vu \cdot \nabla \psi \Big)} \\
\displaystyle + \intdO{ \frac{L+ \delta \vt^{B-1}}{\vt} (\vt-\Theta_0) \psi}
\end{array}
\end{equation}
for all non-negative $\psi \in C^1(\overline{\Omega})$.

More details concerning all estimates and limit passages performed above are contained in \cite{NoPo_JDE}.

\section{Estimates independent of $\delta$: Dirichlet boundary conditions}

We now present basic estimates independent of $\delta$ for the solutions to system (\ref{2.31})--(\ref{2.34}). The first part (up to few details) is the same as in the section devoted to the a priori estimates, however, the estimates of the density are different.

\subsection{Estimates based on entropy inequality}

We first aim at showing the following estimates with constants independent of $\delta$: 
\begin{equation} \label{3.1}
\|\vu_\delta\|_{1,2} \leq C,
\end{equation}
\vspace{-.9cm}
\begin{equation} \label {3.2}
\|\vtd\|_{3m} \leq C \Big(1+ \Big|\intO{\vrd \vud \cdot \vc{f}}\Big|\Big).
\end{equation}

We proceed as in the case of the formal a priori estimates. We use as test functions in the approximate entropy inequality (\ref{2.34}) and in the total energy balance (\ref{2.33}) $\psi \equiv 1$, which leads to  
\begin{equation} \label{3.3}
\begin{array}{c}
\displaystyle \intO{\big(\kappa(\vtd) + \delta \vtd^{B} + \delta \vtd^{-1}\big)\frac{|\nabla \vtd|^2}{\vtd^2}} +  \intO {\Big(\frac{1}{\vtd} \tn{S}(\vtd,\vud):\nabla \vud +  \delta\vtd^{-2} \Big)} \\[8pt]
\displaystyle + \int_{\partial \Omega} \frac{L + \delta \vtd^{B-1}}{\vtd}  \Theta_0 \, \dS \leq  \int_{\partial \Omega} \big(L + \delta \vtd^{B-1}\big) \, \dS,
\end{array}
\end{equation}
and
\begin{equation} \label{3.4}
\int_{\partial \Omega} \big(L \vtd + \delta \vtd^B\big) \, \dS = \intO {\vrd \vud \cdot \vc{f}} + \int_{\partial \Omega} \big(L + \delta \vtd^{B-1}\big) \Theta_0  \, \dS + \delta \intO{\vtd^{-1}}.
\end{equation}
We can get rid of the $\delta$-dependent terms on the right-hand sides (r.h.s.) (more details are given in \cite{NoPo_JDE}) and hence deduce
for $\beta$ and $B$ sufficiently large
\begin{equation} \label{3.9}
\begin{array}{c}
\displaystyle
\|\vud\|_{1,2} + \|\nabla \vtd^{\frac m2}\|_2 + \|\nabla \ln \vtd\|_2 + \|\vtd^{-1}\|_{1,\partial \Omega} \\
\displaystyle + \delta \big(\|\nabla \vtd^{\frac B2}\|_2^2 + \|\nabla \vtd^{-\frac 12}\|_2^2 + \|\vtd\|_{3B}^{B-2} + \|\vtd^{-2}\|_1\big) \leq C.
\end{array}
\end{equation}
Estimate (\ref{3.9}) with (\ref{3.4}) leads to (\ref{3.1})--(\ref{3.2}).

\subsection {Estimates of the pressure}

As shown in the formal a priori estimates, the method based on the pressure estimates by means of the Bogovskii operator has a natural limitation, namely $\gamma > \frac 32$. To avoid this, we apply a different idea based on local pressure estimates. The method was developed in the context of the compressible Navier--Stokes system in the following three papers \cite{PlSo_JMFM}, \cite{BrNo_CMUC} and \cite{FrStWe_JMPA}.
 Note that unlike the heat-conducting case, the method gives existence of weak solutions to the compressible Navier--Stokes system only for $\gamma > \frac 43$. This problem has been removed in the recent paper \cite{PlWe_2015}, using a slightly different technique, which,
 however, leads in the heat-conducting case to more severe restrictions than the original method.

We denote for $b\geq 1$
\begin{equation} \label{3.11}
\AAA = \intO {\vrd^b |\vud|^2}.
\end{equation}
Using H\"older's inequality, one may easily deduce for any $b\geq 1$
\begin{equation} \label{3.11a}
\begin{array}{c}
\displaystyle \|\vud\|_{1,2} \leq C, \\
\displaystyle \|\vtd\|_{3m} \leq C\big(\AAA^{\frac {1}{6b-4}} + 1\big).
\end{array}
\end{equation}

We also need the following estimate based on the application of the Bogovskii operator from Lemma \ref{l 4.2}
\begin{lemma} \label{l 3.3}
We have for $1<s\leq \frac{3b}{b+2}$, $s\leq \frac{6m}{2+3m}$, $m>\frac 23$ and $b\geq 1$
\begin{equation} \label{3.12}
\intO{\vrd^{s\gamma}}
+ \intO {\vrd^{(s-1)\gamma} p(\vrd,\vtd)} + \intO {\big(\vrd |\vud|^2\big)^s} + \delta \intO {\vrd ^{\beta + (s-1)\gamma}} \leq C\big(1+ \AAA^{\frac {4s-3}{3b-2}}\big).
\end{equation}
\end{lemma}
\begin{proof}
We use as test function in (\ref{2.32}) solution to (\ref{2.3}) with 
$$
f= \vr^{(s-1)\gamma} - \frac{1}{|\Omega|} \intO{\vr^{(s-1)\gamma}},
$$
we get
$$
\begin{array}{c}
\displaystyle \intO {\vrd^{(s-1)\gamma} p(\vrd,\vtd)} + \delta \intO {\vrd ^{(s-1)\gamma} \big(\vrd^\beta + \vrd^2)} = \frac{1}{|\Omega|} \intO {p(\vrd,\vtd)} \intO {\vrd^{(s-1)\gamma}} \\[9pt]
\displaystyle + \frac{\delta}{|\Omega|} \intO{(\vrd^\beta + \vrd^2)}  \intO {\vrd^{(s-1)\gamma}} - \intO{\vrd(\vud\otimes \vud):\nabla \vcg{\varphi}} + \intO{\tn{S}(\vud,\vtd): \nabla \vcg{\varphi}}  \\
\displaystyle - \intO{\vrd \vc{f} \cdot \vcg{\varphi}} = I_1 + I_2 + I_3 + I_4 + I_5.
\end{array}
$$
The most restrictive terms are $I_3$ leading to the first restriction on $s$ and $I_4$ which gives the second restriction as well as $m>\frac 23$. More details can be found in the paper \cite{NoPo_SIMA}.
\end{proof}

We are now coming to the most important (and also the most difficult) part of the estimates. We aim at proving that for some $\alpha >0$
$$
\sup_{x_0 \in \Ov{\Omega}} \intO{\frac{p(\vrd,\vtd)}{|x-x_0|^\alpha}} \leq C
$$
with $C$ independent of $\delta$. The goal is to get $\alpha$ as large as possible under some conditions on $m$ and $\gamma$ which still allow the values for these quantities as small as possible. Indeed, this might be sometimes contradictory. Here, the situation for the Dirichlet and the Navier boundary conditions differs, as in the latter we have larger class of possible test functions (only the normal component vanishes on $\partial \Omega$).

We distinguish three possible situations. In the first one, the point $x_0$ is ``far'' from the boundary, in the second one $x_0 \in \partial \Omega$ and in the last one $x_0$ is ``close'' to $\partial \Omega$, but does not belong to it. The first case is independent of the boundary conditions and we have  
\begin{lemma} \label{l 3.4}
Let $x_0 \in \Omega$, $R_0 < \frac 13 \mbox{dist}\, (x_0,\partial\Omega)$. Then
\begin{equation} \label{3.14}
\begin{array}{c}
\displaystyle
\int_{B_{R_0}(x_0)} \frac{p(\vrd,\vtd) + \delta(\vrd^\beta + \vrd^2)}{|x-x_0|^\alpha}\, \dx \\
\displaystyle \leq C\big(1+ \|p(\vrd,\vtd)\|_1 + \|\vud\|_{1,2}(1+\|\vtd\|_{3m}) + \|\vrd|\vud|^2\|_1 + \delta \|\vrd\|_\beta^\beta\big),
\end{array}
\end{equation}
provided
\begin{equation} \label{3.15}
\alpha < \min\Big\{\frac{3m-2}{2m},1\Big\}.
\end{equation}
\end{lemma}
\begin{proof}
We use as test function in (\ref{2.32})
$$
\varphi_i(x) = \frac{(x-x_0)_i}{|x-x_0|^\alpha} \tau^2
$$
with $\tau \equiv 1$ in $B_{R_0}(x_0)$, $R_0$ as above, $\tau \equiv 0$ outside $B_{2R_0}(x_0)$, $|\nabla \tau| \leq \frac{C}{R_0}$. The important observation is that
$$
\begin{array}{c}
\displaystyle \Div \vcg{\varphi} = \frac{3-\alpha}{|x-x_0|^\alpha} \tau^2 + g_1(x), \\
\displaystyle \partial_i \varphi_j = \Big(\frac{\delta_{ij}}{|x-x_0|^\alpha} - \alpha \frac{(x-x_0)_i (x-x_0)_j}{|x-x_0|^{\alpha+2}}\Big) \tau^2 + g_2(x)
\end{array}
$$
with $g_1$, $g_2$ in $L^\infty(\Omega)$. Then
\begin{equation} \label{3.16}
\begin{array}{c}
\displaystyle \intO {\frac{p(\vrd,\vtd)+\delta (\vrd^\beta + \vrd^2)}{|x-x_0|^\alpha} (3-\alpha) \tau^2} + \intO {\Big(\frac{\vrd|\vud|^2}{|x-x_0|^\alpha} - \alpha \vrd \frac{(\vud \cdot (x-x_0))^2}{|x-x_0|^{\alpha+2}} \Big) \tau^2} \\[8pt]
\displaystyle = - \intO {\big(p(\vrd,\vud)+ \delta (\vrd^\beta +  \vrd^2)\big) \frac{(x-x_0)\cdot \nabla \tau^2}{|x-x_0|^\alpha}} + \intO {\tn{S}(\vud,\vtd):\nabla \Big(\frac{x-x_0}{|x-x_0|^\alpha}\Big)\tau^2}  \\
\displaystyle + \intO {\tn{S}(\vud,\vtd) : \frac{x-x_0}{|x-x_0|^\alpha} \nabla \tau^2}- \intO {\vrd \vc{f} \cdot \frac{x-x_0}{|x-x_0|^\alpha} \tau^2} \\
\displaystyle - \intO {\vrd (\vud \otimes \vud) : \frac{x-x_0}{|x-x_0|^\alpha} \nabla \tau^2}.
\end{array}
\end{equation}
We easily see that
$$
\nabla \frac{x-x_0}{|x-x_0|^\alpha} \sim \frac{1}{|x-x_0|^\alpha},
$$
hence
$$
\Big| \intO {\tn{S}(\vud,\vtd) : \nabla \frac{x-x_0}{|x-x_0|^\alpha} \tau^2}\Big| \leq C (1+\|\vtd\|_{3m}) \|\nabla \vud\|_2
$$
provided
$$
\frac 1q = 1 - \frac 12 -\frac 1{3m} > \frac{\alpha}{3},
$$
leading to the restriction $\alpha < \frac{3m-2}{2m}$ for $m>\frac 23$.
The second integral on the l.h.s. of
(\ref{3.16}) is non-negative; it even gives a certain information about $\vrd|\vud|^2$, however, we are not able to recover it (in the case of the Dirichlet b.c.) for $x_0$ near or on the boundary. As $\alpha\le 1$, the other terms on
the r.h.s. of (\ref{3.16}) are evidently bounded. 
\end{proof}

Next, we consider the situation when $x_0 \in \partial \Omega$. In this case we may use the following test function
\begin{equation} \label{3.17}
\vcg{\varphi}(x) = d(x) \nabla d(x) (d(x) + |x-x_0|^a)^{-\alpha}
\end{equation}
with $a= \frac{2}{2-\alpha}$, $x_0 \in \partial \Omega$ and $d(x)$ denoting the distance of $x$ from the boundary. As $\Omega \in C^2$, the distance function $d \in C^2(\overline{\Omega})$. Moreover, $\nabla d(x) = \frac{x-\xi(x)}{d(x)}$, where $\xi(x) \in \partial \Omega$ is the closest point to $x$, cf. \cite[Exercise 1.15]{Zi_1989}. Thus  there exist $c_1$, $c_2$ positive such that:
\begin{itemize} \item[(i)] $d(x) \in C^2(\overline{\Omega})$, $d(x) >0$ in $\Omega$, $d(x) =0$ at $\partial \Omega$
\item [(ii)] $|\nabla d(x)| \geq c_1 >0$, $x\in \overline{\Omega}$ with $\mbox{dist}\, (x,\partial \Omega) \leq c_2$
\item [(iii)] $d(x) \geq c_1 >0$, $x \in \overline{\Omega}$ with $\mbox{dist}\, (x,\partial \Omega)\geq c_2$
\end{itemize}

The main properties of $\vcg{\varphi}$ are (see \cite[Lemma 3.5]{NoPo_SIMA})
\begin{lemma} \label{l 3.5}
The function $\vcg{\varphi}$, defined by (\ref{3.17}), belongs to $W^{1,q}_0(\Omega; \R^3)$ for $1\leq q < \frac{3-\alpha}{\alpha}$. Moreover,
\begin{equation} \label{3.18}
\begin{array}{c}
\displaystyle \partial_j \varphi_i(x) = \frac{d(x) \partial^2_{ij}d(x)}{(d(x)+|x-x_0|^a)^{\alpha}} + \frac{(1-\alpha) d(x) + |x-x_0|^a}{2(d(x)+|x-x_0|^a)^{1+\alpha}} \partial_i d(x) \partial_j d(x) \\[8pt]
\displaystyle + \frac{(1-\alpha)d(x) + |x-x_0|^a}{2(d(x)+ |x-x_0|^a)^{1+\alpha}} (\partial_i d(x) - \mu^i(x))(\partial_j d(x) - \mu^j(x)) \\[8pt]
\displaystyle + \frac{\alpha d(x) [\partial_j d(x) \partial_i (|x-x_0|^a)- \partial_i d(x) \partial_j (|x-x_0|^a)]}{2(d(x)+|x-x_0|^a)^{1+\alpha}} \\[8pt]
\displaystyle - \frac{\alpha^2 d^2(x) \partial_i (|x-x_0|^a) \partial_j (|x-x_0|^a)}{2(d(x)+|x-x_0|^a)^{1+\alpha} \big((1-\alpha)d(x) + |x-x_0|^a\big)},
\end{array}
\end{equation}
where
$$
\mu^i(x) = \alpha d(x) \big((1-\alpha)d(x) + |x-x_0|^a\big)^{-1} \partial_i (|x-x_0|^a),
$$
$i=1,2,3$.
\end{lemma}

We now use $\vcg{\varphi}$ from (\ref{3.17}) as a test function
in (\ref{2.32}). It yields
\begin{lemma} \label{l 3.6}
Under assumptions above, we have for $\alpha < \frac{9m-6}{9m-2}$ and $x_0 \in \partial \Omega$
\begin{equation} \label{3.20}
\begin{array}{c}
\displaystyle
\int_{B_{R_0}(x_0)\cap \Omega} \frac{p(\vrd,\vtd) + \delta (\vrd^\beta + \vrd^2)}{|x-x_0|^\alpha} \,{\rm d} x \\
\displaystyle \leq C\big(1+ \|p(\vrd,\vtd)\|_1 + (1+\|\vtd\|_{3m})\|\vud\|_{1,2} + \|\vrd |\vud|^2\|_1 +\delta \|\vrd\|_{\beta}^\beta\big).
\end{array}
\end{equation}
\end{lemma}
\begin{proof}
We have
\begin{equation} \label{3.21}
\begin{array}{c}
\displaystyle
\intO{\big(p(\vrd,\vtd) + \delta (\vrd^\beta + \vrd^2)\big) \Div \vcg{\varphi}} + \intO{\vrd (\vud \otimes \vud):\nabla \vcg{\varphi}} \\
\displaystyle
= \intO {\tn{S} (\vud,\vtd):\nabla \vcg{\varphi}} - \intO {\vrd \vc{f} \cdot \vcg{\varphi}}.
\end{array}
\end{equation}
From (\ref{3.18}) we see that (recall that $a = \frac{2}{2-\alpha}$)
$$
\begin{array}{c}
\displaystyle\Div \vcg{\varphi} = \frac{d(x) \Delta d(x)}{(d(x)+|x-x_0|^a)^\alpha} + \frac{(1-\alpha)d(x) + |x-x_0|^a}{2(d(x)+|x-x_0|^a)^{\alpha+1}}|\nabla d(x)|^2 \\[8pt]
\displaystyle + \frac{(1-\alpha)d(x) + |x-x_0|^a}{2(d(x)+|x-x_0|^a)^{\alpha+1}} |\nabla d(x) - \mu(x)|^2 \\[8pt]
\displaystyle - \frac{\alpha^2 d^2(x) |\nabla|x-x_0|^a|^2}{2(d(x)+|x-x_0|^a)^{\alpha+1} \big((1-\alpha)d(x) + |x-x_0|^a\big)}
\displaystyle \geq \frac{C_1 |\nabla d(x)|^2}{(d(x)+ |x-x_0|^a)^\alpha} - C_2.
\end{array}
$$
Thus, as $d(x) + |x-x_0|^a\leq C |x-x_0|$, $d(x)$ is continuously differentiable near the boundary and $d(x_0) = 0$, together with $a>1$
$$
\begin{array}{c}
\displaystyle 
\intO {(p(\vrd,\vtd) +\delta(\vrd^\beta + \vrd^2))\Div\vcg{\varphi}} \\
\displaystyle \geq C_1 \int_{B_{R_0}(x_0)\cap \Omega} \frac{p(\vrd,\vtd) +\delta (\vrd^\beta + \vrd^2)}{|x-x_0|^\alpha} \, \mbox{d} x - C_2 \intO {\big(p(\vrd,\vtd) +\delta (\vrd^\beta +\vrd^2)\big)}.
\end{array}
$$
Next
$$
\intO {\vrd (\vud \otimes \vud):\nabla \vcg{\varphi}} \geq -C \intO {\vrd |\vud|^2 },
$$
as the skew symmetric part of $\nabla \vcg{\varphi}$ has zero contribution. Note that the positive part of $\vrd (\vud \otimes \vud):\nabla \vcg{\varphi}$ does not provide any useful information. The last term on the r.h.s. of \eqref{3.21} can be estimated by $\|\vc{f}\|_\infty \|\vrd\|_1 \|\vcg{\varphi}\|_\infty$,  and
$$
\Big| \intO {\tn{S}(\vud,\vtd):\nabla \vcg{\varphi}}\Big| \leq C \|\nabla \vud\|_2 (1+ \|\vtd\|_{3m})
$$
provided
$
\frac 1q = \frac 12 - \frac{1}{3m}
$
is such that $q< \frac{3-\alpha}{\alpha}$, i.e. $\alpha < \frac{9m-6}{9m-2}$.
\end{proof}

Finally, we deal with the case when $x_0$ is close to the boundary, but does not belong to it. Here we must combine the test functions from the previous two situations. Note that this case was not carefully commented in the original papers.

Assume that $x_0 \in \Omega$ is such that $\dist\{x_0,\partial \Omega\}= 5\varepsilon$ for some $0<\varepsilon \ll 1$. Our aim is to get estimates as above with constants independent of $\varepsilon \to 0^+$. First we consider the test function as in the case when $x_0 \in \partial \Omega$, i.e.
$$
\vcg{\varphi}^1(x) = d(x) \nabla d(x) (d(x) + |x-x_0|^a)^{-\alpha}.
$$

We have as above  
$$
\intO {\vrd (\vud \otimes \vud):\nabla \vcg{\varphi}^1} \geq C_1 \intO{\frac{\vrd(\vud\cdot \nabla d)^2}{(d(x) + |x-x_0|^a)^\alpha}}-C_2 \intO {\vrd |\vud|^2},
$$
and,
$$
\begin{array}{c}
\displaystyle
\intO {\big(p(\vrd,\vtd) + \delta(\vrd^\beta + \vrd^2)\big) \Div\vcg{\varphi}^1} \\
\displaystyle \geq C_1 \intO {\frac{p(\vrd,\vtd) + \delta(\vrd^\beta + \vrd^2)}{(d(x)+ |x-x_0|^a)^\alpha}} - C_2 \intO {\big(p(\vrd,\vtd) +\delta(\vrd^\beta + \vrd^2)\big)}.
\end{array}
$$
It is easy to see that $\|\vcg{\varphi}^1\|_{1,q} \leq C$ independently of $\varepsilon$ for $q< \frac{3-\alpha}{\alpha}$. 
However, we have 
$$
\frac{1}{(d(x)+ |x-x_0|^a)^\alpha} \geq \frac{C}{|x-x_0|^\alpha}
$$
only for $x \in \Omega \setminus B_\varepsilon(x_0)$. Therefore we need an additional estimate in the ball $B_\varepsilon(x_0)$.

To this aim we use a test function, which is similar to the test function when $x_0$ is far from the boundary. Here, however, we have additional difficulty connected with the fact that the test function must vanish on $\partial \Omega$. We consider
\begin{equation} \label{lovely_test}
\vcg{\varphi}^2(x) = 
\left\{
\begin{array}{rl}
\displaystyle \frac{x-x_0}{|x-x_0|^\alpha} \Big(1-\frac{1}{2^{\frac \alpha 2}}\Big)^2, & \quad |x-x_0| \leq \varepsilon, \\
\displaystyle (x-x_0) \Big( \frac{1}{|x-x_0|^{\frac{\alpha}{2}}}  - \frac{1}{(|x-x_0|+ \varepsilon)^{\frac{\alpha}{2}}}\Big)^2, & \quad |x-x_0| > \varepsilon, d(x) >\varepsilon, \\
\displaystyle (x-x_0) \Big( \frac{1}{|x-x_0|^{\frac{\alpha}{2}}}  - \frac{1}{(|x-x_0|+ d(x))^{\frac{\alpha}{2}}}\Big)^2, & \quad |x-x_0| > \varepsilon, d(x) \leq \varepsilon. \\
\end{array}\right.
\end{equation}
It is easy to verify that $\vcg{\varphi}^2 \in W^{1,q}_0(\Omega;\R^3)$ with the norm bounded independently of $\varepsilon$ for all $1\leq q< \frac{3}{\alpha}$. Moreover, due to properties mentioned above, we can verify that
$$
\begin{array}{c}
\displaystyle \intO {\vrd (\vud \otimes \vud):\nabla \vcg{\varphi}^2} \geq K_1 \int_{B_\varepsilon(x_0)} \frac{\vrd|\vud|^2}{|x-x_0|^\alpha}\, \dx \\
\displaystyle +K_2 \int_{\{x; d(x)\leq \varepsilon\}} \frac{\vrd (\vud\cdot \nabla d)^2}{(d(x) +|x-x_0|^a)^\alpha}\, \dx -K_3 \intO {\vrd |\vud|^2},
\end{array}
$$
and,
$$
\begin{array}{c}
\displaystyle
\intO {\big(p(\vrd,\vtd) + \delta(\vrd^\beta + \vrd^2)\big) \Div\vcg{\varphi}^2} \\
\displaystyle \geq K_1 \int_{B_\varepsilon(x_0)} \frac{p(\vrd,\vtd) + \delta(\vrd^\beta + \vrd^2)}{|x-x_0|^\alpha} \, \dx - K_3 \intO{\big(p(\vrd,\vtd) +\delta(\vrd^\beta + \vrd^2)\big)}.
\end{array}
$$
Hence, taking as a test function in \eqref{2.32} 
$$
\vcg{\varphi} = K \vcg{\varphi}^1 + \vcg{\varphi}^2
$$
for a sufficiently large $K>0$, we end up with
\begin{equation} \label{3.23}
\begin{array}{c}
\displaystyle 
\sup_{x_0 \in \Ov{\Omega}}\intO {\frac{p(\vrd,\vtd) + \delta(\vrd^\beta + \vrd^2)}{|x-x_0|^\alpha}} \\
\displaystyle \leq C\Big(1+ \|p(\vrd,\vtd)\|_1 + \|\vud\|_{1,2}(1+\|\vtd\|_{3m}) + \|\vrd |\vud|^2\|_1 + \delta\|\vrd\|_\beta^\beta \Big),
\end{array}
\end{equation}
provided $\alpha < \frac{9m-6}{9m-2}$.

We can deduce from \eqref{3.23}

\begin{lemma} \label{l 3.7}
Let $1\leq b <\gamma$, $\alpha < \frac{9m-6}{9m-2}$ and $\alpha > \frac{3b-2\gamma}{b}$. Then
\begin{equation} \label{3.22}
\AAA = \intO{\vrd^b |\vud|^2}  \leq C \|\vud\|_{1,2}^2 \big(1+ \|p(\vrd,\vtd)\|_1 +  \|\vud\|_{1,2}(1+\|\vtd\|_{3m}) + \|\vrd |\vud|^2\|_{1}\big)^{\frac{b}{\gamma}}.
\end{equation}
\end{lemma}
\begin{proof}
Take $b<\gamma$ and $\nu = \frac{\gamma-\alpha b}{\gamma-b} <3$ (i.e. $\alpha > \frac{3b-2\gamma}{b}$). As $\vrd^\gamma \leq C p(\vrd,\vtd)$, we have
\begin{equation} \label{3.24}
\begin{array}{c}
\displaystyle \intO {\frac{\vrd^b}{|x-x_0|}} = \intO {\Big(\frac {\vrd^\gamma}{|x-x_0|^\alpha}\Big)^{\frac b\gamma} \Big(\frac{1}{|x-x_0|^\nu}\Big)^{1-\frac b \gamma}} \\
\displaystyle \leq C \Big(1+ \|p(\vrd,\vtd)\|_1 + \|\vud\|_{1,2} (1+ \|\vtd\|_{3m}) + \|\vrd |\vud|^2\|_1 \Big)^{\frac b\gamma}.
\end{array}
\end{equation}

Let $h$ be the unique solution to
\begin{equation} \label{3.25}
\begin{array}{c}
\displaystyle -\Delta h = \vrd^b >0 \qquad \mbox{ in } \Omega,\\
\displaystyle h = 0 \qquad \mbox{ at } \partial \Omega.
\end{array}
\end{equation}
Then
$$
h(x) = \int_{\Omega} G(x,y) \vrd^b(y)\, \mbox{d} y
$$
with $G(\cdot,\cdot)$ the Green function to problem (\ref{3.25}). As $|G(x,y)| \leq \frac{C}{|x-y|}$ for all $x,y \in \Omega$, $x\neq y$, we get from (\ref{3.23})--(\ref{3.24}) that $h \in L^\infty(\Omega)$ with
\begin{equation} \label{3.26}
\|h\|_\infty \leq C \sup_{x_0 \in \Ov{\Omega}} \intO {\frac{\vrd^b(x)}{|x-x_0|}}.
\end{equation}
Therefore
$$
\AAA = \intO {-\Delta h(x) |\vud(x)|^2} = 2 \intO {\nabla \vud: (\vud \nabla h)},
$$
and
\begin{equation} \label{3.27}
\AAA \leq C \|\nabla\vud\|_2 \Big(\intO{|\vud|^2 |\nabla h|^2}\Big)^{\frac 12}.
\end{equation}
Now
$$
\begin{array}{c}
\displaystyle D = \intO{|\vud|^2 |\nabla h|^2} = - \intO {h \nabla h \cdot \nabla |\vud|^2} - \intO {|\vud|^2 h \Delta h} \\
\displaystyle \leq C \|h\|_\infty (\AAA + \|\nabla \vud\|_2 D^{\frac 12}).
\end{array}
$$
Thus
$$
D \leq C (\AAA \|h\|_\infty + \|\nabla \vud\|_2^2 \|h\|_{\infty}^2 ).
$$
Returning to (\ref{3.27}), Young's and Friedrichs' inequalities imply
$$
\AAA \leq C \|\nabla \vud\|_2^2 \|h\|_\infty
$$
which finishes the proof.
\end{proof}

Now, combining  Lemmas \ref{l 3.3} and \ref{l 3.7}
$$
\AAA \leq C \Big(1+ \AAA^{\frac{4s-3}{3b-2}\frac 1s} + \AAA^{\frac 1{6b-4}} + \AAA^{\frac 1{6b-4} (\frac{(s-1)\gamma}{(s-1)\gamma+1} + \frac{2(4s-3)}{(s-1)\gamma+1})}\Big)^{\frac b\gamma},
$$
i.e.
\begin{equation} \label{3.29}
\AAA \leq C \Big(1+ \AAA^{\frac{4s-3}{3b-2} \frac 1s} + \AAA^{\frac 1 {6b-4} (1+ \frac{8s-7}{(s-1)\gamma +1})} \Big)^{\frac b\gamma}.
\end{equation}
Therefore, assuming
\begin{equation} \label{3.30}
\frac{4s-3}{s} \frac{1}{3b-2} \frac{b}{\gamma} <1,   \qquad \frac 1 {6b-4} \Big(1+ \frac{8s-7}{(s-1)\gamma +1}\Big) \frac{b}{\gamma} <1
\end{equation}
we get
$$
\AAA \leq C.
$$
Checking all conditions throughout the computations above we conclude
\begin{lemma} \label{l 3.8}
Let $\gamma>1$, $m>\frac 23$ and $m > \frac{2}{9} \frac \gamma{\gamma-1}$. Then there exists $s>1$ such that $\vrd$ is bounded in $L^{s\gamma}(\Omega)$ and $p(\vrd,\vtd)$, $\vrd |\vud|$ and $\vrd |\vud|^2$ are bounded in $L^s(\Omega)$. Moreover, if $\gamma > \frac 43$, and
\begin{equation} \label{4.5}
\begin{array}{c}
\displaystyle m>1 \qquad \mbox{ for } \qquad \gamma > \frac{12}{7}, \\[8pt]
\displaystyle m> \frac{2\gamma}{3(3\gamma-4)} \qquad \mbox{ for } \qquad \gamma \in \Big(\frac 43,\frac {12}{7}\Big],
\end{array}
\end{equation}
we can take $s> \frac 65$.
\end{lemma}
\begin{proof}
The details can be found in \cite{NoPo_SIMA}.
\end{proof}
 
\section{Estimates independent of $\delta$: Navier boundary conditions}

Note that we get in the case when $\Omega$ is not axially symmetric estimates \eqref{3.1}, \eqref{3.2} and \eqref{3.9} exactly as for the Dirichlet boundary conditions. Therefore, we only deal here with the estimates of the pressure, where we closely follow the papers \cite{JeNo_JMPA} and \cite{JeNoPo_M3AS}.

\subsection{Estimates of the pressure}

We define now for $1\leq a\leq \gamma$ and $0<b<1$
\begin{equation} \label{II3.11}
\BBB = \intO{\big(\vrd^a |\vu|^2 + \vrd^b |\vu|^{2b+2}\big)},
\end{equation}
where $1\leq a\leq \gamma$ and $0<b<1$. Employing H\"older's inequality we easily have
\begin{lemma} \label{l II3.2}
Under the assumptions on $a$ and $b$, there exists $C$ independent of $\delta$, such that
\begin{equation} \label{II3.12}
\|\vrd \vud\|_1 \leq C \BBB^{\frac{a-b}{2(ab+a-2b)}}.
\end{equation}
\end{lemma}

\begin{lemma} \label{l II3.3}
Under the assumptions on $a$ and $b$, and for $1<s< \frac{1}{2-a}$ (if $a<2$), $0<(s-1)\frac{a}{a-1} <b<1$, there exists $C$, independent of $\delta$, such that
\begin{equation} \label{II3.13}
\|\vrd |\vud|^2\|_s \leq C \BBB^{\frac{a-b/s}{ab+a-2b}}.
\end{equation}
\end{lemma}

Next, using also the Bogovskii-type estimate we get as in Lemma \ref{l 3.3} 

\begin{lemma} \label{l II3.4}
Let $1\leq a\leq \gamma$, $0<b<1$, $1<s<\frac 1{2-a}$ (if $a<2$), $0< (s-1)\frac{a}{a-1} <b<1$, $s\leq \frac{6m}{3m+2}$, $m> \frac 23$. Then there exists a constant $C$ independent of $\delta$ such that
$$
\intO{\big(\vrd^{s\gamma} + \vrd^{(s-1)\gamma} p(\vrd, \vtd) + (\vrd |\vud|^2)^s + \delta \vrd^{\beta + (s-1)\gamma}\big)} \leq C(1+ \BBB^{\frac{sa-b}{ab+a-2b}}).
$$
\end{lemma}

As in the previous case, we need to estimate the pressure. We proceed similarly as for the Dirichlet boundary conditions, however, since we have more freedom, i.e. we have a larger class of the test functions for the momentum equation, we can get better results here. First of all, in the case when $x_0$ is far from the boundary, we have again Lemma \ref{l 3.4}. However, it is now possible to use the information from the second term on the l.h.s. as it is possible to recover it for $x_0$ close and on the boundary. The difference appears near the boundary. The situation is more complex here. For the sake of simplicity, we restrict ourselves to the case of the flat boundary $\partial \Omega$. The general case can be treated using the standard change of variables which flattens the curved boundary. This is the reason why we require the boundary to be $C^2$. More details are given in \cite{JeNo_JMPA} or in \cite{JeNoPo_M3AS}. 

Let us hence assume that we deal with the part of boundary of $\Omega$ which is flat and is described by $x_3 = 0$, i.e. $z(x')=0$, $x' \in {\cal O}\subset \R^2$ with the normal vector $\vc{n} = (0,0,-1)$ and  $\vcg{\tau}_1 =(1,0,0)$, $\vcg{\tau}_2=(0,1,0)$ the tangent vectors. Consider  first that $x_0$ lies on the boundary of $\Omega$, i.e. $(x_0)_3 = 0$. Then it is possible to use as the test function in the approximate momentum equation
$$
\vc{w}(x) = \vc{v}(x-x_0),
$$   
where
$$
\vc{v}(x) = 
\frac{1}{|x|^\alpha}(x_1,x_2,x_3) = (x\cdot \vcg{\tau}_1)\vcg{\tau}_1 + (x\cdot \vcg{\tau}_2)\vcg{\tau}_2 + ((0,0,x_3-z(x'))\cdot \vc{n})\vc{n}, \quad x_3\geq 0.
$$
Note that if $(x_0)_3=0$ we get precisely what we need, i.e. estimate \eqref{II3.31} (but with $\sup_{x_0 \in \partial \Omega}$ instead of $\sup_{x_0 \in \Ov{\Omega}}$. 
 
However, if $x_0$ is close to the boundary but not on the boundary, i.e. $(x_0)_3>0$, but small, we loose  control of some terms for $0<x_3<(x_0)_3$. In this case, as for the Dirichlet boundary conditions, we must modify the test functions; recall that we require that only the normal component (i.e. in our case the third component) of the test function vanishes on $\partial \Omega$. We first consider
$$
\vc{v}^1(x) = \left\{
\begin{array}{ll}
\displaystyle \frac{1}{|x-x_0|^\alpha}\big((x-x_0)_1,(x-x_0)_2,(x-x_0)_3\big) , & x_3 \geq (x_0)_3/2, \\[8pt]
\displaystyle \frac{1}{|x-x_0|^\alpha}\Big((x-x_0)_1,(x-x_0)_2,4(x-x_0)_3 \frac{x_3^2}{|(x-x_0)_3|^2}\Big),  & 0<x_3 <  (x_0)_3/2.
\end{array}
\right.
$$
Nonetheless, using $\vc{v}^1$ as test function we would still miss control of some terms from the convective term, more precisely of those, which contain at least one velocity component $u_3$, however, only close to the boundary, i.e. for $x_3 < (x_0)_3/2$. Hence we further consider
$$
\vc{v}^2(x) = \left\{
\begin{array}{ll}
\displaystyle \frac{(0,0,x_3)}{(x_3+ |x-x_0| |\ln |x-x_0||^{-1})^\alpha} , & |x-x_0| \leq 1/K, \\[8pt]
\displaystyle \frac{(0,0,x_3)}{(x_3+ 1/K |\ln K|^{-1})^\alpha} , & |x-x_0| > 1/K
\end{array}
\right.
$$
for $K$ sufficiently large (but fixed, independently of the distance of $x_0$ from $\partial \Omega$). Note that both functions belong to $W^{1,q}_{\vc{n}}(\Omega;\R^3)$ and their norms are bounded uniformly (with respect to the distance of $x_0$ from $\partial \Omega$) provided $1\leq q<\frac 3\alpha$. Thus we finally use as the test function in the approximate momentum balance
\begin{equation} \label{119a}
\vcg{\varphi} = \vc{v}^1(x) + K_1 \vc{v}^2(x)
\end{equation}
with $K_1$ suitably chosen (large). Note that the choice of $K$ and $K_1$ is done in such a way that the unpleasant terms from both functions are controlled by those from the other one which provide us a positive information. This is possible due to the fact that the unpleasant terms from $\vc{v}^2$ are multiplied by $|\ln|x-x_0||^{-1} \leq |\ln K|^{-1}\ll 1$. 

Similarly as in the case of Dirichlet boundary conditions we can therefore verify that

\begin{multline}\label{II3.31}
\sup_{x_0 \in \Ov{\Omega}} \intO{ \frac{p(\vrd,\vtd) +\delta(\vrd^\beta + \vrd^2)+ (1-\alpha)\vrd |\vud|^2}{|x-x_0|^\alpha}} \\ 
\leq C (1+ \delta \|\vrd\|_\beta^\beta + \|p(\vrd,\vtd)\|_1 + (1+ \|\vtd\|_{3m})\|\vud\|_{1,2} + \|\vrd|\vud|^2 \|_1),
\end{multline}
provided $0<\alpha < \max\{1, \frac{3m-2}{2m}\}$, $m>\frac 23$,
and, moreover, the test function \eqref{119a} belongs to $W^{1,p}(\Omega;\R^3)$ for $1\leq p < \frac 3\alpha$ with the norm bounded independently of the distance of $x_0$ from $\partial \Omega$. 

Applying \eqref{II3.31} we have to distinguish two cases. First, for $m\geq 2$ we have $\frac{3m-2}{2m} \geq 1$, hence the only restriction on $\alpha$ is actually $\alpha <1$. In the other case, if $m \in (\frac 23, 2)$, we have the restriction $\alpha < \frac{3m-2}{2m}$.
Therefore, if $m \geq 2$, we get (passing with $\alpha \to 1^-$, by Fatou's lemma)
$$
\intO{\frac{p(\vrd,\vtd)}{|x-y|}} \leq C \big(1+ \delta \|\vrd\|_\beta^\beta + \|p(\vrd,\vtd)\|_1 + (1+ \|\vtd\|_{3m})\|\vud\|_{1,2} + \|\vrd|\vud|^2 \|_1\big).
$$
Next, for $0<b<1$
$$
\frac{\vrd^b |\vud|^{2b}}{|x-y|} \leq \Big(\frac{\vrd |\vud|^2}{|x-y|^\alpha}\Big)^b \frac{1}{|x-y|^{1-b\alpha}},
$$
thus
\begin{equation} \label{II3.28}
\intO{\frac{\vrd^b |\vud|^{2b}}{|x-y|} } \leq \Big(\intO{ \frac{\vrd |\vud|^2}{|x-y|^\alpha}}\Big)^b \Big(\intO{\frac{1}{|x-y|^{\frac{1-b\alpha}{1-b}}}}\Big)^{1-b}.
\end{equation}
Hence we get (note that we may take $a=\gamma$ in \eqref{II3.11}, see below)
\begin{lemma} \label{l II3.8}
Let $b \in ((s-1)\frac{\gamma}{\gamma-1},1)$, $1<s<\frac{2}{2-\gamma}$, $m\geq 2$, $s \leq \frac{6m}{3m+2}$. Then there exists $C$ independent of $\delta$ such that for any $y \in \overline{\Omega}$
\begin{equation} \label{II3.28a}
\begin{array}{c}
\displaystyle 
\intO{\frac{p(\vrd,\vtd) + (\vrd |\vud|^2)^b}{|x-y|}} \\
\displaystyle \leq C \big(1+ \delta \|\vrd\|_\beta^\beta + \|p(\vrd,\vtd)\|_1 + (1+ \|\vtd\|_{3m})\|\vud\|_{1,2} + \|\vrd|\vud|^2 \|_1\big).
\end{array}
\end{equation}
\end{lemma}

If $m<2$, we proceed rather as for Dirichlet boundary conditions. For $1\leq a<\gamma$ we compute
\begin{equation} \label{II3.29}
\begin{array}{c}
\displaystyle \intO{\frac{\vrd^a}{|x-y|}} = \intO{\Big(\frac{\vrd^\gamma}{|x-y|^\alpha}\Big)^{a/\gamma} \Big(\frac{1}{|x-y|^{1-\frac{a\alpha}{\gamma}}}\Big)}  \\
\displaystyle \leq \Big(\intO{\frac{\vrd^\gamma}{|x-y|^\alpha}}\Big)^{a/\gamma} \Big(\intO{\frac{1}{|x-y|^{\frac{\gamma-a\alpha}{\gamma-a}}}}\Big)^{\frac{\gamma-a}{\gamma}}.
\end{array}
\end{equation}
Hence for $\frac{\gamma-\alpha a}{\gamma-a} <3$, i.e. $\alpha > \frac{3a-2\gamma}{a}$, we have
\begin{equation} \label{II3.30}
\intO{\frac{\vrd^a}{|x-y|}} \leq C \big(1+ \delta \|\vrd\|_\beta^\beta + \|p(\vrd,\vtd)\|_1 + (1+ \|\vtd\|_{3m})\|\vud\|_{1,2} + \|\vrd|\vud|^2 \|_1\big)^{\frac a\gamma}.
\end{equation}
Further, proceeding as in \eqref{II3.28}
\begin{equation} \label{II3.31a}
\intO{\frac{\vrd^b |\vud|^{2b}}{|x-y|} } \leq C \big(1+ \delta \|\vrd\|_\beta^\beta + \|p(\vrd,\vtd)\|_1 + (1+ \|\vtd\|_{3m})\|\vud\|_{1,2} + \|\vrd|\vud|^2 \|_1\big)^b,
\end{equation}
however, now for $\frac{1-b\alpha}{1-b}<3$, i.e. (if $b>\frac 23$)
\begin{equation} \label{II3.32}
\alpha > \frac{3b-2}{b}. 
\end{equation}
Altogether we have
\begin{lemma} \label{l II3.9}
Let $b \in ((s-1)\frac{\gamma}{\gamma-1},1)$, $1<s<\frac{2}{2-\gamma}$, $\alpha > \max\{\frac{3a-2\gamma}{a}, \frac{3b-2}{b}\}$, $m\in (\frac 23, 2)$. Then there exists $C$ independent of $\delta$ such that 
\begin{equation} \label{II3.33}
\begin{array}{c}
\displaystyle \sup_{x_0 \in \Ov{\Omega}}\intO{\frac{\vrd^a + (\vrd |\vud|^2)^b}{|x-x_0|}} \\
\displaystyle \leq C \big(1+ \delta \|\vrd\|_\beta^\beta + \|p(\vrd,\vtd)\|_1 + (1+ \|\vtd\|_{3m})\|\vud\|_{1,2} + \|\vrd|\vud|^2 \|_1\big)^{\frac a \gamma} \\
\displaystyle+ C \big(1+ \delta \|\vrd\|_\beta^\beta + \|p(\vrd,\vtd)\|_1 + (1+ \|\vtd\|_{3m})\|\vud\|_{1,2} + \|\vrd|\vud|^2 \|_1\big)^{b}.
\end{array}
\end{equation}
\end{lemma}

We now consider the problem
\begin{equation} \label{II3.34}
\begin{array}{c}
\displaystyle
-\Delta h = \vrd^a + \vrd^b |\vud|^{2b} - \frac{1}{|\Omega|} \intO{(\vrd^a + \vrd^b |\vud|^{2b})}, \\
\displaystyle
\pder{h}{\vc{n}}|_{\partial \Omega} = 0.
\end{array}
\end{equation}
The unique strong solution admits the following representation
\begin{equation} \label{II3.35}
h(x) = \int_\Omega G(x,y) (\vrd^a + \vrd^b|\vud|^{2b}) \, \mbox{d} y - \frac{1}{|\Omega} \int_\Omega G(x,y) \, \mbox{d} y \intO{(\vrd^a + \vrd^b |\vud|^{2b})};
\end{equation}
since $G(x,y)\leq C |x-y|^{-1}$, we get due to Lemmas \ref{l II3.8}, \ref{l II3.9} together with Lemmas \ref{l II3.2} and \ref{l II3.3}
\begin{itemize} \item $m\geq 2$
\begin{equation} \label{II3.35a}
\|h\|_\infty \leq C(1+ \BBB^{\frac{\gamma -b/s}{b\gamma + \gamma -2b}}),
\end{equation}
provided
\begin{equation} \label{II3.35b} 
1<s< \frac{1}{2-\gamma}, \quad 0<(s-1)\frac{\gamma}{\gamma-1} <b<1, \quad s\leq \frac{6m}{3m+2}
\end{equation}
\item $\frac 23 <m<2$
\begin{equation} \label{II3.35c}
\|h\|_\infty \leq C(1+ \BBB^{\frac{a -b/s}{ab + a -2b} \frac a \gamma} + \BBB^{\frac{a -b/s}{ab + a -2b} b}),
\end{equation}
provided
\begin{equation} \label{II3.35d} 
\begin{array}{c}
\displaystyle 1<s< \frac{1}{2-a}, \quad 0<(s-1)\frac{a}{a-1} <b<1, \quad s\leq \frac{6m}{3m+2}, \\[8pt]
\displaystyle \alpha > \frac{3a-2\gamma}{a}, \quad \alpha > \frac{3b-2}{b}, \quad \alpha < \frac{3m-2}{2m}. 
\end{array}
\end{equation}
\end{itemize}
Now, from \eqref{II3.11} and \eqref{II3.34}, we have
\begin{equation} \label{II3.36}
\begin{array}{c}
\displaystyle \BBB = \intO{-\Delta h \vud^2} + \frac{1}{|\Omega|} \intO{\vud^2}\intO{(\vrd^a + \vrd^b |\vud|^{2b})} \\[8pt]
\displaystyle = \intO{\nabla h \cdot \nabla |\vud|^2 } \frac{1}{|\Omega|} \intO{\vud^2}\intO{(\vrd^a + \vrd^b |\vud|^{2b})} \\[8pt]
\leq 2\|\nabla \vud\|_2 D^{\frac 12} + C(\varepsilon) \|\vud\|_{1,2}^2 (1+ \BBB^{\Gamma+\varepsilon})
\end{array}
\end{equation}
for any $\varepsilon >0$, and 
$$ \Gamma =\left\{
\begin{array}{ll}
\displaystyle \frac{\gamma-b}{b\gamma +\gamma -2b} \quad &\mbox{ if } m \geq 2\\
\displaystyle \max\Big\{\frac{a-b}{ab+a-2b} \frac{a}{\gamma}, \frac{a-b}{ab+a-2b} b\Big\} \quad &\mbox{ if } \frac 23 <m<2.
\end{array}\right.
$$
Above, we denoted
\begin{equation} \label{II3.37}
D = \intO{|\nabla h \otimes \vud|^2}.
\end{equation}
Employing once more integration by parts
\begin{multline*}
\displaystyle D = -\intO{h \Delta h |\vud|^2} - \intO{h \nabla h \cdot \nabla \vud \cdot \vud} \\
\leq \|h\|_\infty (\AAA + C(\varepsilon)\|\vud\|_{1,2}^2 \BBB^{\Gamma +\varepsilon} +\|\nabla \vud\|_2 D^{\frac 12}),
\end{multline*}
i.e.
\begin{equation} \label{II3.38}
D \leq \|h\|_\infty \BBB + C(\varepsilon)\|\vud\|_{1,2}^2 \BBB^{\Gamma +\varepsilon}+ \frac 12 \|\nabla \vud||_2^2 \|h\|_\infty^2.
\end{equation}
Returning back to \eqref{II3.36}, \eqref{II3.38} and \eqref{II3.37} yield
\begin{equation} \label{II3.39}
\BBB \leq C \|\nabla \vud\|_2^2 \|h\|_\infty + C(\varepsilon)\|\vud\|_{1,2}^2 \BBB^{\Gamma +\varepsilon}.
\end{equation}

Hence, due to \eqref{II3.35a}, 
\begin{equation} \label{II3.40}
\begin{array}{c}
\BBB \leq C \big(1+ \BBB^{\frac{\gamma-b/s}{b\gamma + \gamma-2b}}\big) \quad \mbox{ if } m\geq 2, \\
\BBB \leq C\big(1+ \BBB^{\frac {a-b/s}{ab+a-2b}\frac{a}{\gamma}} + \BBB^{\frac {a-b/s}{ab+a-2b} b}\big) \quad \mbox {if } \frac 23 <m<2.
\end{array}
\end{equation}

Checking carefully all restrictions above, we proved that for $\gamma >1$ and $m>\frac{2}{4\gamma-3}$, $m>\frac 23$, 
there exists $s>1$ such that
\begin{equation} \label{II3.41}
\begin{array}{lcr}
\sup_{\delta>0} \|\vrd\|_{\gamma s} &<& \infty, \\
\sup_{\delta>0} \|\vrd\vud \|_{s} &<& \infty, \\
\sup_{\delta>0} \|\vrd|\vud|^2\|_{s} &<& \infty, \\
\sup_{\delta>0} \|\vud\|_{1,2} &<& \infty, \\
\sup_{\delta>0} \|\vtd\|_{3m} &<& \infty, \\
\sup_{\delta>0} \|\vtd^{m/2}\|_{1,2} &<& \infty, \\
\sup_{\delta>0} \delta \|\vrd^{\beta + (s-1)\gamma}\|_{1} &<& \infty.
\end{array}
\end{equation}
Moreover, we can take $s>\frac 65$ provided 
$\gamma >\frac 54$, $m>\max \{1,\frac{2\gamma +10}{17\gamma -15}\}$, for further details see \cite{JeNoPo_M3AS}.

\section{Limit passage $\delta \to 0^+$ towards the original system}

We present the last step of the proof of our Theorem \ref{t 1.2}. The proof of Theorem \ref{t 1.1} can be done similarly, we shall only comment on the conditions for $m$ and $\gamma$ at the end.
 
\subsection{Limit passage due to uniform bounds}

Estimates \eqref{II3.41} yield existence of a subsequence of $(\vrd, \vud, \vtd)$ (denoted again in the same way) such that
\begin{equation} \label{II4.1}
\begin{array}{c}
\vud \rightharpoonup \vu \quad \mbox{ in } W^{1,2}(\Omega;\R^3), \\
\vud \to \vu \quad \mbox{ in } L^q(\Omega;\R^3), \ q<6 \\
\vud \to \vu \quad \mbox{ in } L^r(\partial \Omega;\R^3), \ r<4 \\
\vrd \rightharpoonup \vr \quad \mbox{ in } L^{s\gamma}(\Omega), \\
\delta \vrd^\beta \to 0 \quad \mbox{ in } L^q(\Omega), \ q< 1+ (s-1)\frac{\gamma}{\beta} \\
\vtd \rightharpoonup \vt \quad \mbox{ in } W^{1,p}(\Omega), \ p = \min \{2,\frac{3m}{m+1}\}, \\
\vtd \to \vt \quad \mbox{ in } L^q(\Omega),\  q<3m, \\
\vtd \to \vt \quad \mbox{ in } L^r(\partial \Omega), \ r<2m, \\
p(\vrd,\vtd) \to \overline{p(\vr,\vt)} \quad \mbox{ in } L^r(\Omega), \mbox{ for some } r>1, \\
e(\vrd,\vtd) \to \overline{e(\vr,\vt)} \quad \mbox{ in } L^r(\Omega), \mbox{ for some } r>1, \\
s(\vrd,\vtd) \to \overline{s(\vr,\vt)} \quad \mbox{ in } L^r(\Omega), \mbox{ for some } r>1. \\
\end{array}
\end{equation}

Recalling also Lemma \ref{l 4.6}, we can pass to the limit in the weak formulation of the continuity equation, momentum equation, entropy inequality and global total energy balance to get

\begin{equation} \label{II4.2}
\intO {\vr \vu \cdot \nabla \psi} = 0
\end{equation}
for all $\psi \in C^1(\overline{\Omega})$,
\begin{multline} \label{II4.3}
\intO {\Big(-\vr (\vu\otimes \vu):\nabla \vcg{\varphi} + \tn{S}(\vt,\nabla \vu):\nabla \vcg{\varphi} - \overline{p(\vr,\vt)} \Div \vcg{\varphi}\Big)} \\
+ \lambda \intdO{ \vu \cdot \vcg{\varphi}}  = \intO {\vr \vc{f} \cdot \vcg{\varphi}}
\end{multline}
for all $\vcg{\varphi} \in C^1_{\vc{n}}(\overline{\Omega}; \R^3)$,
\begin{equation} \label{II4.4}
\begin{array}{c}
\displaystyle \intO{\Big(\vt^{-1} \tn{S} (\vt,\nabla \vu):\nabla \vu  + \kappa(\vt) \frac{|\nabla \vt|^2}{\vt^2}  \Big)\psi} \\[10pt]
\displaystyle \leq \intO {\Big(\kappa(\vt) \frac{\nabla \vt \cdot \nabla \psi}{\vt} - \overline{\vr s(\vr,\vt)} \vu \cdot \nabla \psi \Big)} + \intdO{ \frac{L}{\vt} (\vt-\Theta_0) \psi},
\end{array}
\end{equation}
for all non-negative $\psi \in C^1(\overline{\Omega})$, 
\begin{equation} \label{II4.5}
\intdO{ (L(\vt-\Theta_0)+\lambda |\vu|^2 )}  = \intO{\vr \vc{f} \cdot \vu}.
\end{equation}

However, to pass to the limit in the total energy balance, we need that $\vrd |\vud|^2 \rightharpoonup \vr |\vu|^2$ in some $L^q(\Omega)$, $q>\frac 65$, $\vtd \to \vt$ in some $L^r(\Omega)$, $r>3$. This is true for $s>\frac 65$ and $m>1$; note that in this case we  also have $\delta \|\vrd\|_{\frac 65\beta}^\beta \to 0$.

Hence, assuming $\gamma > \frac 54$, $m> \max\{1,\frac{2\gamma + 10}{17\gamma -15}\}$  we also get the total energy balance
\begin{equation} \label{II4.6}
\begin{array}{c}
\displaystyle \intO {\Big(\big(-\frac 12 \vr |\vu|^2 - \overline{\vr e(\vr,\vt)}\big)\vu \cdot \nabla \psi + \kappa(\vt)   \nabla \vt\cdot \nabla \psi\Big)} \\
\displaystyle + \intdO{ \big(L (\vt-\Theta_0)  + \lambda |\vu|^2\big)\psi} =
 \intO{\vr \vc{f} \cdot \vu \psi}
 \\[8pt]
\displaystyle + \intO {\big(-\tn{S}(\vt,\nabla \vu)\vu +\overline{p(\vr,\vt)\vu} \big) \cdot \nabla \psi}
\end{array}
\end{equation}
for all $\psi \in C^1(\overline{\Omega})$.

To finish the proof, we need to verify that $\vrd \to \vr$ in some $L^r(\Omega)$, $r\geq 1$. 

\subsection{Effective viscous flux}

We use in the approximative momentum equation as a test function
\begin{equation} \label{II4.7}
\vcg{\varphi} = \zeta \nabla \Delta^{-1} (1_\Omega T_k(\vrd)), \quad k\in \tn{N}, 
\end{equation} 
with $\zeta \in C^\infty_c(\Omega)$,
$$
T_k(z) = k T\Big(\frac{z}{k}\Big), \qquad
T(z) = \left\{ \begin{array}{l}
z \mbox{ for } 0\leq z\leq 1, \\
\mbox{concave on } (0,\infty), \\
2 \mbox{ for } z\geq 3.
\end{array}
\right.
$$

In its limit version \eqref{II4.3} we use
\begin{equation} \label{II4.8}
\vcg{\varphi} = \zeta \nabla \Delta^{-1} (1_\Omega \overline{T_k(\vr)}), \quad k\in \tn{N}, 
\end{equation}
where $\overline{T_k(\vr)}$ is the weak limit of $T_k(\vrd)$ as $\delta \to 0^+$ (the corresponding chosen subsequence). After technical, but standard computation  (cf. \cite{NoPo_SIMA} or \cite{FeNo_Book} for the evolutionary case,
 see also Chapter 10.2.1) we get
\begin{equation} \label{II4.9}
\begin{array}{c}
\displaystyle
\lim_{\delta \to 0^+} \intO {\zeta \Big(p(\vrd,\vtd) T_k(\vrd) - \tn{S}(\vtd,\nabla \vud):{\cal R}[1_\Omega T_k(\vrd)]\Big) } \\
\displaystyle = \intO {\zeta \Big(\overline{p(\vr,\vt)} \ \overline{T_k(\vr)} -\tn{S}(\vt,\nabla\vu):{\cal R}[1_\Omega \overline{T_k(\vr)}]\Big)} \\
\displaystyle + \intO {\zeta \vu\cdot \big(\vr \vu \cdot{\cal R}(\overline{T_k(\vr)})\big)} -\lim_{\delta \to 0^+}\intO{\zeta \vud\cdot \big(\vrd \vud \cdot{\cal R}(T_k(\vrd))\big)},
\end{array}
\end{equation}
see \eqref{2.4a} and \eqref{2.4b} for the definition of ${\cal R}$. We now apply Lemma \ref{l 4.3} with
$$
\begin{array}{c}
\displaystyle
v_\delta = T_k(\vrd) \rightharpoonup \overline{T_k(\vr)} \qquad \mbox{ in } L^{q}(\R^3), \, q<\infty \mbox{ arbitrary } \\
\displaystyle \vc{U}_\delta = \vrd \vud \rightharpoonup \vr
\vu \qquad \mbox{ in } L^{p}(\R^3;\R^3), \, \mbox{ for certain } p>1,
\end{array}
$$
hence
$$
\vrd \vud \cdot{\cal R}(T_k(\vrd)) \rightharpoonup \vr \vu \cdot{\cal R}(T_k(\vr))
$$
in $L^p(\Omega;\R^3)$ for any $p\in [1,s)$. On the other hand,
$$
\vrd (\vud \otimes \vud) :{\cal R}(T_k(\vrd)) = \vud\cdot\big({\cal R}(T_k(\vrd))\vrd \vud\big)
$$
is uniformly bounded in $L^p(\Omega)$ with $p \in [1,s]$. Therefore, by virtue of Lemma \ref{l 4.6}, together with the facts that
$$
\begin{array}{c}
\displaystyle
\vrd \vud \rightharpoonup \vr \vu \quad \mbox{ in } L^s(\Omega;\R^3), \\
\displaystyle 
\vrd \vud \otimes \vud \rightharpoonup \vr \vu \otimes \vu\quad \mbox{ in } L^s(\Omega;\R^{3\times 3})
\end{array}
$$
we get
\begin{equation} \label{II4.10}
\begin{array}{c}
\displaystyle
\lim_{\delta \to 0^+} \intO {\zeta \Big(p(\vrd,\vtd) T_k(\vrd) - \tn{S}(\vtd,\nabla\vud):{\cal R}[1_\Omega T_k(\vrd)]\Big) } \\[8pt]
\displaystyle = \intO {\zeta \Big(\overline{p(\vr,\vt)} \ \overline{T_k(\vr)} -\tn{S}(\vt,\nabla\vu):{\cal R}[1_\Omega \overline{T_k(\vr)}]\Big)}.
\end{array}
\end{equation}
Next, writing
\begin{equation} \label{II4.11}
\begin{array}{c}
\displaystyle \lim_{\delta \to 0^+} \intO {\zeta \tn{S}(\vtd,\nabla\vud):{\cal R}[1_\Omega T_k(\vrd)] }  = \lim_{\delta \to 0^+} \intO {\zeta \Big(\frac 43 \mu(\vtd)+ \xi(\vtd)\Big) \Div \vud T_k(\vrd)} \\
\displaystyle + \lim_{\delta \to 0^+} \intO {T_k(\vtd) \Big({\cal R}\Big[\zeta \mu(\vtd) \big(\nabla \vud + (\nabla \vud)^T\big)\Big] 
\displaystyle -\zeta \mu(\vtd){\cal R}:\big[\nabla \vud+ (\nabla \vud)^T\big]\Big)},
\end{array}
\end{equation} 
and using in the second term Lemma \ref{l 4.4} we end up with the {effective viscous flux} identity
\begin{equation} \label{II4.12}
\begin{array}{c}
\displaystyle
\overline{p(\vr,\vt) T_k(\vr)} - \Big(\frac 43 \mu(\vt) + \xi(\vt)\Big) \overline{T_k(\vr) \Div \vu} \\
\displaystyle =  \overline{p(\vr,\vt)} \,\, \overline{T_k(\vr)} - \Big(\frac 43 \mu(\vt) + \xi(\vt)\Big) \overline{T_k(\vr)} \Div \vu.
\end{array}
\end{equation}

\subsection{Oscillation defect measure and renormalized continuity  equation}

Our aim is to show that the renormalized continuity equation is fulfilled. However, it can be shown directly only for $\vrd$ bounded in $L^2(\Omega)$ which for $\gamma$ close to $1$ is generally not true. Following the idea originally due to E. Feireisl, we introduce the {\it oscillation defect measure}
\begin{equation} \label{II4.13}
\mbox{{\bf osc}}_{q} [\vrd\to\vr](Q) = \sup_{k>1} \Big(\limsup_{\delta \to 0^+} \int_Q |T_k(\vrd)-T_k(\vr)|^q \dx\Big).
\end{equation}
We have (see \cite[Lemma 3.8]{FeNo_Book})
\begin{lemma} \label{l II4.1}
Let $\Omega \subset \R^3$ be open and let
$$
\begin{array}{cl}
\vrd \rightharpoonup \vr &\mbox{ in } L^1(\Omega), \\
\vud \rightharpoonup \vu &\mbox{ in } L^r(\Omega;\R^3), \\
\nabla \vud \rightharpoonup \nabla \vu &\mbox{ in } L^r(\Omega;\R^{3\times 3}), \quad r>1.
\end{array}
$$
Let
\begin{equation} \label{II6.8}
\mbox{{\bf osc}}_{q} [\vrd\to\vr](\Omega) < \infty
\end{equation}
for $\frac 1q < 1-\frac 1r$, where $(\vrd,\vud)$ solve the
renormalized continuity equation. Then the limit
functions also solve (\ref{1.23c}) for all $b \in C^1([0,\infty)) \cap
W^{1,\infty}(0,\infty)$, $zb'\in L^\infty((0,\infty)$.
\end{lemma}

We have 
\begin{lemma} \label{l II4.2}
Let $(\vrd,\vud,\vtd)$ be as above and let $m>\max\{\frac{2}{3(\gamma-1)},\frac 23\}$. Then there exists $q>2$ such that \eqref{II6.8} holds true.
Moreover,
\begin{equation} \label{II4.16}
\begin{array}{c}
\displaystyle
 \limsup_{\delta \to 0^+} \intO{ \frac{1}{\frac 43 \mu(\vt) + \xi(\vt)}|T_k(\vrd) -T_k(\vr)|^{\gamma+1}} \\
 \displaystyle \leq \frac 1d \intO {\frac{1}{\frac 43 \mu(\vt) + \xi(\vt)} \Big(\overline{p(\vr,\vt)T_k(\vr)}  - \overline{p(\vr,\vt)} \, \, \overline{T_k(\vr)}\Big)}.
 \end{array}
\end{equation}
\end{lemma}

\begin{proof}
Recall that
$$
p(\vr,\vt) = d \vr^\gamma + p_m(\vr,\vt), \qquad \pder{p_m(\vr,\vt)}{\vr} \geq 0,
$$
see (\ref{1.25a}).
We get using Lipschitz continuity of $T_k$ and
trivial inequality $(a-b)^\gamma \leq a^\gamma -b^\gamma$, $a\geq
b\geq 0$,
$$
\begin{array}{c}
\displaystyle
d \limsup_{\delta \to 0^+} \intO{ |T_k(\vrd) -T_k(\vr)|^{\gamma+1}} \leq d\limsup_{\delta\to 0^+} \intO {(\vr^\gamma - \vrd^\gamma)(T_k(\vr)-T_k(\vrd))} \\
\displaystyle = d \intO { \big(\overline{\vr^\gamma T_k(\vr)} - \overline{\vr^\gamma} \overline{T_k(\vr)}\big)} + d \intO {(\vr^\gamma - \overline{\vr^\gamma})\big(T_k(\vr) - \overline{T_k(\vr)}\big)}.
\end{array}
$$
Using convexity of $\vr \mapsto \vr^\gamma$, concavity of $T_k$, strong convergence of the temperature, (\ref{1.25a}) and Lemma \ref{l 4.5}, we get
\begin{equation} \label{III4.13}
d \limsup_{\delta \to 0^+} \intO{ |T_k(\vrd) -T_k(\vr)|^{\gamma+1}} \leq \intO {\Big(\overline{p(\vr,\vt)T_k(\vr)}  - \overline{p(\vr,\vt)} \, \, \overline{T_k(\vr)}\Big)}.
\end{equation}
The same argument also yields
\begin{equation} \label{III4.15a}
d \limsup_{\delta \to 0^+} \intO{ \frac{1}{1+\vt}|T_k(\vrd) -T_k(\vr)|^{\gamma+1}} \leq \intO {\frac{1}{1+\vt} \Big(\overline{p(\vr,\vt)T_k(\vr)}  - \overline{p(\vr,\vt)} \, \, \overline{T_k(\vr)}\Big)}.
\end{equation}

Let $G_k(t,x,z) = d |T_k(z) - T_k(\vr(t,x)|^{\gamma+1}$. Thus
$$
\overline{G_k(\cdot,\cdot,\vr)} \leq \overline{p(\vr,\vt)T_k(\vr)}  - \overline{p(\vr,\vt)} \,\, \overline{T_k(\vr)}
$$
and using (\ref{II4.12}),
$$
\overline{G_k(\cdot,\cdot,\vr)} \leq \Big(\frac 43 \mu(\vt) + \xi(\vt)\Big) \Big(\overline{T_k(\vr) \Div \vu} - \overline{T_k(\vr)} \Div \vu\Big)
$$
for all $k\geq 1$. Then
\begin{equation} \label{7.1}
\begin{array}{c}
\displaystyle
\intO {(1+\vt)^{-1} \overline{G_k(t,x,\vr)}} \leq C \sup_{\delta>0}\|\Div\vud\|_{2} \limsup_{\delta\to 0^+} \|T_k(\vrd) - T_k(\vr)\|_2 \\
\displaystyle \leq C \limsup_{\delta\to 0^+} \|T_k(\vrd) - T_k(\vr)\|_2.
\end{array}
\end{equation}
On the other hand, for $q>2$
$$
\begin{array}{c}
\displaystyle \intO {|T_k(\vrd)-T_k(\vr)|^q }\leq \intO{|T_k(\vrd)-T_k(\vr)|^q  (1+\vt)^{-\frac{q}{\gamma+1}} (1+\vt)^{\frac{q}{\gamma+1}}} \\
\displaystyle \leq C\Big[ \intO {(1+\vt)^{-1}
|T_k(\vrd)-T_k(\vr)|^{\gamma+1}} \Big]^{\frac q{\gamma+1}}\;\Big[
\intO{(1+\vt)^{\frac{q}{\gamma+1-q}}}\Big]^{\frac{\gamma+1-q}{\gamma+1}},
\end{array}
$$
which, using (\ref{7.1}) and an obvious interpolation, yields  the desired
result.
\end{proof}

As $(\vrd,\vud)$ and $(\vr,\vu)$ verify due to Lemmas \ref{l II4.1} and \ref{l II4.2} the renormalized continuity equation, we  have the identities
$$
\intO{T_k(\vr) \Div \vu} = 0,
$$
and
$$
\intO{T_k(\vrd) \Div \vud} = 0, \qquad \mbox{ i.e. } \intO{\overline{T_k(\vr) \Div \vu}} = 0.
$$
Hence, employing the effective viscous flux identity (\ref{II4.12}),
\begin{equation} \label{II4.11b}
\begin{array}{c}
\displaystyle
\displaystyle \intO {\frac{1}{\frac 43 \mu(\vt) + \xi(\vt)}\Big(\overline{p(\vr,\vt) T_k(\vr)} -
\overline{p(\vr,\vt)}\, \, \overline{
T_k(\vr)}\Big)} = \intO{\big(T_k(\vr) - \overline{T_k(\vr)}\big) \Div \vu}.
\end{array}
\end{equation}
We easily have $\lim_{k\to \infty} \|T_k(\vr)-\vr\|_1 = \lim_{k\to \infty} \|\overline{T_k(\vr)} -\vr\|_1 = 0$. Thus,  (\ref{II4.16}) and (\ref{II4.11b}) yield
$$
\lim_{k \to \infty} \intO {\frac{1}{\frac 43\mu(\vt) + \xi(\vt)}\Big(\overline{p(\vr,\vt)T_k(\vr)}  - \overline{p(\vr,\vt)}\ \overline{T_k(\vr)}\Big)} = 0.
$$
Using once more (\ref{II4.16}) we get
$$
\lim_{k \to \infty} \limsup_{\delta \to 0^+} \intO{ \frac{1}{\frac 43 \mu(\vt) + \xi(\vt)}|T_k(\vrd) -T_k(\vr)|^{\gamma+1}} = 0,
$$
which implies
$$
\lim_{k \to \infty} \limsup_{\delta \to 0^+} \intO{|T_k(\vrd)
-T_k(\vr)|^q} = 0
$$
with $q$ as in Lemma \ref{l II4.2}. 

As
$$
\|\vrd-\vr\|_1 \leq \|\vrd-T_k(\vrd)\|_1 + \|T_k(\vrd) -T_k(\vr)\|_1 + \|T_k(\vr) -\vr\|_1,
$$
we have
$$
\vrd \to \vr \qquad \mbox{ in } L^1(\Omega);
$$
whence
$$
\vrd \to \vr \quad \mbox{ in } L^p(\Omega) \quad   \forall 1\leq p < s\gamma.
$$
To finish the proof of Theorem \ref{t 1.2}, note that  the condition $m>\frac{2}{3(\gamma-1)}$ is the most restrictive one. For Theorem \ref{t 1.1} we also easily check that $\frac{2}{3(\gamma-1)}> \frac{2}{9(\gamma-1)}$ and that for the weak solutions, both $m>1$ and $m> \frac{2\gamma}{3(3\gamma-4)}$ must be taken into account.

\section{Weak solutions with bounded density and internal energy balance}

In this section, we consider a modification of the problem studied above. The approach is based on the paper \cite{MuPo_CMP}. We consider our system (\ref{1.1})--(\ref{1.3}), however, we replace the total energy balance \eqref{1.3} by the internal energy balance
\begin{equation} \label{11.1.3}
\Div\big(\vr e(\vr,\vt) \vu\big) -\Div \big(\kappa(\vt)\nabla \vt\big)=
\tn{S}(\nabla \vu):\nabla \vu - p (\vr,\vt)\Div \vu,
\end{equation}
together with the Navier boundary conditions (the use of them is essential here). Note that for sufficiently smooth solution, the total and the internal energy balances are equivalent (one is just a consequence of the other one, using the balance of momentum and the continuity equation), however, for only weak solution, this might not be the case. We further assume that the viscosity coefficients are constants, i.e. $\alpha =0$ in \eqref{1.7} (this is also essential in this approach), hence we prefer to write the viscous part of the stress tensor in the form
$\tn{S}(\nabla \vu) = 2\mu \tn{D}(\vu) + \nu (\Div \vu) \tn{I}$. We assume $2\mu + 3\nu >0$ and consider the pressure law \eqref{1.18a} and (it is also essential in this approach) the function $L(\cdot)$ from \eqref{1.5} as
$$
L(\vt) = a(1+\vt)^l
$$
for some $l \in \R_0^+$. Finally, we also prescribe the total mass of the fluid. We have the following result
\begin{theorem}[Mucha, Pokorn\'y, 2009] \label{t11.1}
Let $\Omega \in C^2$ be a bounded domain in $\R^3$ which is not axially symmetric if $\lambda=0$.
Let $\vc{f} \in L^\infty(\Omega;\R^3)$ and 
 $$
  \gamma>3, \qquad m=l+1>\frac{3\gamma-1}{3\gamma-7}.
 $$
Then there exists a weak solution to our problem such that 
$$
\vr \in L^\infty(\Omega), \qquad \vu \in W^{1,q}(\Omega;\R^3),
 \mbox { \  \ } \vt \in W^{1,q}(\Omega)  \mbox{\ \ for all }1\leq q<\infty, 
$$
and $\vr \geq 0$,  $\vt>0$ a.e. in $\Omega$.
\end{theorem}

The proof is based on the approximation procedure presented for the first time in the context of the two-dimensional steady compressible Navier--Stokes equations in \cite{MuPo_Nonlinearity}, see also \cite{Zat1}. We define 
\begin{equation} \label{def-K}
K(t)=\left\{
\begin{array}{lcr}
1 & \mbox{for} & t<k-1 \\
\in [0,1] & \mbox{for} &\quad k-1\leq t\leq k\\
0 & \mbox{for} & t>k;
\end{array}
\right.
\end{equation}
moreover, we assume that $K'(t)<0$ for $t\in (k-1,k)$, where $k\in \R^+$ will be sufficiently large.  Take $\varepsilon >0$ and $K(\cdot)$ as above. Our approximate problem reads
 \begin{equation}\label{ap-ns}
\left.\begin{array}{r}
\displaystyle\varepsilon \vr +\Div(K(\vr)\vr \vu)-\varepsilon \Delta \vr =\varepsilon
h K(\vr) \\[9pt]
\displaystyle\frac{1}{2}\Div(K(\vr)\vr \vu \otimes \vu)+\frac{1}{2}K(\vr)\vr \vu \cdot
\nabla \vu- \Div \tn{S}(\nabla \vu)+\nabla  P(\vr,\vt)=\vr K(\vr)\vc{f} \\[9pt]
\displaystyle -\Div\Big(\kappa(\vt)\frac{\varepsilon +\vt}{\vt}\nabla \vt \Big)
+\Div\Big(\vu\int_0^\rho K(t) \dt\Big)\vt +\Div\Big(K(\vr)\vr \vu\Big) \vt
\\[5pt]
\displaystyle
+K(\vr)\vr \vu \cdot \nabla \vt-\vt K(\vr) \vu\cdot \nabla \vr  = \tn{S}(\nabla \vu):\nabla \vu
\end{array}\right\}
\mbox{ in } \Omega,
\end{equation}
where 
\begin{equation}\label{ap-P}
P(\vr,\vt)=\int_0^\varrho \gamma t^{\gamma-1} K(t) \dt +\vt \int_0^\rho K(t) \dt = P_b(\vr) + \vt \int_0^\vr K(t) \dt, 
\end{equation}
and $h=\frac{M}{|\Omega|}$.

This system is completed by the boundary conditions on $\partial \Omega$
\begin{equation}\label{bc}
\begin{array}{c}
\displaystyle 
(1+\vt^m)(\varepsilon +\vt)\frac{1}{\vt}\pder{\vt}{\vc{n}}
+L(\vt)(\vt-\Theta_0)+\varepsilon \ln \vt=0,  \\[8pt]
\displaystyle 
\vu \cdot \vc{n} = 0, \qquad \vcg{\tau}_k \cdot (\tn{S}(\nabla \vu)\vc{n}) + \lambda \vu \cdot \vcg{\tau}_k = 0, \qquad k=1,2, \\[8pt]
\displaystyle \pder{\vr}{\vc{n}} =0.
\end{array}
\end{equation}
Recall that the reason for terms of the form $\ln \vt$ to appear in the approximate problem is that we in fact solve the approximate problem for the ``entropy'' $s = \ln \vt$ instead of the temperature itself. It provides straightaway  that the temperature of the approximate problem is positive a.e. in $\Omega$ and we keep this information throughout the limit passages. Moreover, the form of the function $K(\cdot)$ ensures  that the approximate density is bounded between $0$ and $k$; this can be easily seen if we integrate equation \eqref{ap-ns}$_1$ over the set, where $\vr<0$, and over the set, where $\vr >k$. The main problem in the limit passage $\varepsilon \to 0^+$ is to verify that $\vr \leq k-1$. This ensures that for the limit problem (which is our original problem) we have $K(\vr) \equiv 1$.

For the approximate problems, we can prove the following
\begin{proposition} \label{pt11.2}
Let the assumptions of Theorem \ref{t11.1} be satisfied. Moreover, let $\varepsilon >0$ and $k>0$. Then there exists a strong solution
$(\vr,\vu)$ to (\ref{ap-ns})--(\ref{bc})  such that
$$
\vr \in W^{2,p}(\Omega), \;\; \vu \in W^{2,p}(\Omega;\R^3) \mbox{ \ \ 
and \ \ } \ln \vt \in W^{2,p}(\Omega) \mbox{ \ \ for all \ \ }
1\leq p <\infty.
$$
Moreover 
$0 \leq \vr \leq k$ in $\Omega$, 
$\int_\Omega \vr \, \dx \leq M$ and
\begin{equation}\label{b-m}
\|\vu\|_{1,3m}+\sqrt{\varepsilon}\|\nabla \vr\|_{2} + \|\nabla \vt\|_{r} + \|\vt\|_{3m} \leq C(k),
\end{equation}
where $\vt >0$, $r=\min\{2,\frac{3m}{m+1}\}$ and the r.h.s. of (\ref{b-m}) is independent of  $\varepsilon$.
\end{proposition} 

The proof is based on suitable linearization and application of a version of the Schauder fixed-point theorem, similarly as above for the temperature dependent viscosities. However, the a priori estimates are obtained differently, see \cite[Theorem 2]{MuPo_CMP} for more details. 

We may therefore pass with $\varepsilon \to 0^+$. Estimates (\ref{b-m}) from Proposition \ref{pt11.2} guarantee us existence of a subsequence $\varepsilon\to 0^+$ such that
\begin{equation}\label{11.e2}
\begin{array}{c}
\vu_\varepsilon \rightharpoonup \vu \mbox{ \ \ \ in } W^{1,{3m}}(\Omega;\R^3),\qquad\qquad
\vu_\varepsilon \to \vu \mbox{ \ \ \ in } L^\infty(\Omega;\R^3),\\
\vr_\varepsilon \rightharpoonup^* \vr \mbox{ \ \ \ in } L^\infty(\Omega),\qquad\qquad
P_b(\vr_\varepsilon) \rightharpoonup^* \overline{P_b(\vr)} \mbox{ \ \ \ in } L^\infty(\Omega),\\
K(\vr_\varepsilon)\vr_\varepsilon \rightharpoonup^* \overline{K(\vr)\vr}\mbox{ \ \ \ in } L^\infty(\Omega),\qquad\qquad
K(\vr_\varepsilon)\rightharpoonup^* \overline{K(\vr)}\mbox{ \ \ \ in } L^\infty(\Omega),\\[8pt]
\displaystyle \int_0^{\vr_\varepsilon} K(t)\dt \rightharpoonup^* \overline{\int_0^\vr K(t)\dt}\mbox{ \ \ \ in } L^\infty(\Omega),\\
\vt_\varepsilon \rightharpoonup \vt \mbox{ \ \ \ in } W^{1,r}(\Omega) \mbox{ with }
r=\min\{2,\frac{3m}{m+1}\},\\
\vt_\varepsilon \to \vt \mbox{ \ \ \ \ in }L^q(\Omega) \mbox{\ \ \ for } q< 3m.
\end{array}
\end{equation}
Passing to the limit in the weak formulation of our problem (\ref{ap-ns}) we get (all equations are fulfilled in the weak sense)
\begin{equation}\label{11.4.1}
\Div (\overline{K(\vr)\vr} \vu) =0,
\end{equation} 
\begin{equation}\label{11.4.2}
\overline{K(\vr)\vr} \vu \cdot \nabla \vu - \Div\Big(2\mu \tn{D}(\vu)+\xi (\Div \vu) \tn{I}-
\overline{P_b(\vr)}\tn{I}-\vt \big(\overline{\int_0^\vr K(t)\dt}\big)\tn{I} \Big)=\overline{K(\vr)\vr}\vc{f},
\end{equation}
\begin{equation} \label{11.4.3}
-\Div(\kappa(\vt)\nabla \vt) + \vt \big(\overline{\Div \vu \int_0^\vr K(t) \dt }\big) + \Div (\overline{K(\vr)\vr }\vt \vu) = 
2\mu\overline{|\tn{D}(\vu)|^2} + \nu \overline{(\Div \vu)^2}
\end{equation}
together with the corresponding boundary conditions for the velocity and the temperature. Recall that (\ref{11.4.1})--(\ref{11.4.3}) is satisfied in the weak sense.

In what follows we  study the dependence of the a priori bounds on $k$. 

\begin{lemma} \label{l 11.4.1} 
Under the assumptions of Theorem \ref{t11.1} and Proposition \ref{pt11.2}, we have 
\begin{equation}\label{ee1}
\|\vr_\varepsilon\|_{\infty} \leq k \mbox{ \ \ and \ \  }
\|\vu_\varepsilon\|_{1,3m}\leq C(1+k^{\frac \gamma 3 \frac{3m-2}{m}}).
\end{equation}
\end{lemma}

\begin{proof}
The estimate for the density follows directly from the properties of the function $K(\cdot)$ and the estimate of the velocity follows from the momentum equation, using estimates of the temperature and the $L^1$-estimate of the density, which are independent of $k$. More details can be found in \cite{MuPo_CMP}.
\end{proof}

A crucial role in the proof of the strong convergence of the density is played by a quantity called the effective viscous flux. To define
it in this context, we use the {Helmholtz decomposition} of the velocity 
\begin{equation}\label{11.e4}
\vu=\nabla \phi  +\rot  \vc{A},
\end{equation}
where the divergence-less part of the velocity is given as a solution to the following 
elliptic problem
\begin{equation}\label{11.e5}
\begin{array}{c}
\rot\rot\vc{A}=\rot\vu=\bomega \mbox{ \ \ \ \ in } \Omega,\\
\Div\rot \vc{A}=0 \mbox{ \ \ \ \ in } \Omega,\\
\rot \vc{A}\cdot \vc{n} =0 \mbox{ \ \ \ \ on } \partial \Omega.
\end{array}
\end{equation}
The potential part of the velocity is given  by
the solution to 
\begin{equation}\label{11.e6}
\begin{array}{c}
\Delta \phi=\Div \vu \mbox{ \ \ in } \Omega,\\
\frac{\partial \phi}{\partial \vc{n}}=0 \mbox{ \ \ on } \partial \Omega,
\end{array}
\qquad\qquad  \int_\Omega \phi \ \dx =0.
\end{equation}
The classical theory for elliptic equations, see e.g. the papers \cite{So_LOMI} and \cite{MuPoJMFM},
gives us for $1<q<\infty$
\begin{equation*}\label{11.e7}
\begin{array}{rcl}
\|\nabla \rot\vc{A}\|_{q}\leq C\|\bomega\|_{q},
& \qquad &
\|\nabla^2 \rot \vc{A}\|_{q} \leq C\|\bomega\|_{1,q}, \\[8pt]
\|\nabla^2 \phi\|_{q}\leq C\|\Div \vu\|_{q}, &\qquad & 
\|\nabla^3 \phi\|_{q}\leq C\|\Div \vu\|_{1,q}.
\end{array}
\end{equation*}

We now use that the velocity satisfies the Navier boundary conditions. We have
\begin{equation}\label{11.e8}
\begin{array}{c}
\displaystyle -\mu\Delta \bomega_\varepsilon= \rot\big(K(\vr_\varepsilon)\vr_\varepsilon \vc{f} - K(\vre)\vre
\vue \cdot \nabla \vue \\[8pt]
\displaystyle -
\frac{1}{2} \varepsilon h K(\vre)\vue + \frac{1}{2} \varepsilon \vre \vue\big)
-\rot\Big(\frac{1}{2} \varepsilon \Delta \vre \vue\Big):=\vc{H}_1+\vc{H}_2 \mbox{ \ \ in } \Omega, \\[12pt]
\bomega_\varepsilon\cdot \vcg{\tau}_1=-(2\chi_2 -\lambda/\mu)\vue \cdot \vcg{\tau}_2 \mbox{ \ \ on } \partial \Omega,\\[8pt]
\bomega_\varepsilon\cdot \vcg{\tau}_2=(2\chi_1 -\lambda/\mu)\vue \cdot \vcg{\tau}_1 \mbox{ \ \ on } \partial \Omega,\\[8pt]
\Div \bomega_\varepsilon =0 \mbox{ \ \ on }\partial \Omega,
\end{array}
\end{equation}
where $\chi_k$ are the curvatures associated with the directions $\vcg{\tau}_k$. For the proof of relations (\ref{11.e8})$_{2,3}$ see e.g.
\cite{Mu_JDE} or \cite{MuRa}.

We may write
\begin{equation}\label{11.e9}
\bomega_\varepsilon=\bomega_\varepsilon^0+\bomega_\varepsilon^1+\bomega_\varepsilon^2,
\end{equation}
where 
\begin{equation}\label{11.e10}
\begin{array}{cccr}
-\mu \Delta \bomega_\varepsilon^0=0, &-\mu \Delta \bomega_\varepsilon^1=\vc{H}_1, &
-\mu \Delta \bomega_\varepsilon^2=\vc{H}_2 & \mbox{ in } \Omega,\\
\bomega_\varepsilon^0\cdot \vcg{\tau}_1=-(2\chi_2 -\lambda/\mu)\vue \cdot \vcg{\tau}_2, &
\bomega_\varepsilon^1\cdot \vcg{\tau}_1=0, & \bomega_\varepsilon^2\cdot \vcg{\tau}_1=0 & \mbox{ on } \partial \Omega,\\
\bomega_\varepsilon^0\cdot \vcg{\tau}_2=(2\chi_1 -\lambda/\mu)\vue \cdot \vcg{\tau}_1, &
\bomega_\varepsilon^1\cdot \vcg{\tau}_2=0, & \bomega_\varepsilon^2\cdot \vcg{\tau}_2=0 & \mbox{ on } \partial \Omega,\\
\Div \bomega_\varepsilon^0=0, & \Div \bomega_\varepsilon^1=0, &\Div \bomega_\varepsilon^2=0 & \mbox{ on } \partial \Omega.
\end{array}
\end{equation}
Using elliptic estimates and Lemma \ref{l 11.4.1} we can prove

\begin{lemma} \label{l 11.4.3} For the vorticity $\bomega_\varepsilon$ written in the form (\ref{11.e9}) we have:
\begin{equation}\label{11.e11}
\begin{array}{c}
\|\bomega_\varepsilon^2\|_{r}\leq C(k)\varepsilon^{1/2} \mbox{ \ \ \ for } 1\leq r \leq 2,
\\[8pt]
\|\bomega_\varepsilon^0\|_{1,q}+\|\bomega_\varepsilon^1\|_{1,q}\leq C(1+k^{1+\gamma (\frac 43 -\frac 2q)}) \mbox{ \ \ \ for }
2\leq q \leq 3m.
\end{array}
\end{equation}
\end{lemma}

We now introduce  the effective viscous flux which is in fact the potential part of the momentum equation.
Using the Helmholtz decomposition in the approximate momentum equation we have 
\begin{equation*}\label{11.e24}
\begin{array}{c}
\displaystyle \nabla (-(2 \mu+\nu) \Delta \phi_\varepsilon +P(\vre,\vte))=
\mu \Delta \rot \vc{A}_\varepsilon + K(\vre)\vre \vc{f} 
\\[12pt]
\displaystyle - K(\vre)\vre
\vue \cdot \nabla \vue 
-\frac 12 \varepsilon h K(\vre)\vue +\frac 12 \varepsilon \vre \vue
-\frac 12 \varepsilon\Delta \vre \vue.
\end{array}
\end{equation*}
We define
\begin{equation} \label{evf_eps}
G_\varepsilon = -(2 \mu+\nu) \Delta \phi_\varepsilon +P(\vre,\vte) = -(2 \mu+\nu) \Div \vue 
+P(\vre,\vte)
\end{equation}
and its limit version
\begin{equation} \label{evf}
G =  -(2 \mu+\nu) \Div \vu 
+\overline{P(\vr,\vt)}.
\end{equation}
Note that we are able to control  integrals $\int_\Omega G_\varepsilon \dx = \int_\Omega P(\vre,\vte) \dx$ and 
$\int_\Omega G \dx= \int_\Omega  \overline{P(\vr,\vt)} \dx$, where $\overline{P(\vr,\vt)} = \overline{P_b(\vr)} + \vt \big(\overline{\int_0^\vr
K(t)\dt}\big)$.

Using the results presented above we may show (see \cite{MuPo_CMP} for more details)

\begin{lemma} \label{l 11.4.4} 
We have, up to a subsequence $\varepsilon \to 0^{+}$:
\begin{equation}\label{11.e26}
G_\varepsilon \to G \mbox{ strongly in } L^2(\Omega) 
\end{equation}
and
\begin{equation} \label{11.e26a}
\|G\|_{\infty} \leq C(\eta)(1+k^{1+\frac 23 \gamma +\eta}) \mbox{ \ \ \ for any $\eta>0$.}
\end{equation}
\end{lemma}

The following two results form the core of the method. First, we show that  $\vr \leq (k-3)$ a.e. in $\Omega$, i.e. $K(\vre)\to 1$, and next we get that $\vre\to \vr$ strongly in any $L^q(\Omega)$. More precisely,

\begin{lemma}\label{lt11.3}
{\it  There exists a sufficiently large number $k_0>0$ such that for $k> k_0$
\begin{equation}\label{s9}
\frac{k-3}{k}(k-3)^\gamma -\|G\|_{\infty} \geq 1
\end{equation}
and for a subsequence $\varepsilon \to 0^{+}$ it holds
\begin{equation}\label{s2}
\lim_{\varepsilon \to 0^+} |\{ x \in \Omega: \vre(x) > k-3 \} | =0.
\end{equation}
In particular it follows: $\overline{K(\vr)\vr}=\vr$ a.e. in $\Omega$.}
\end{lemma}

\begin{lemma} \label{lt11}
 We have
\begin{equation}\label{11.s14}
\int_\Omega \overline{P(\vr,\vt)\vr}\ \dx \leq \int_\Omega G \vr \ \dx 
\mbox{ \ \ and \ \ } \int_\Omega \overline{P(\vr,\vt)}\vr \ \dx=\int_\Omega G \vr \ \dx;
\end{equation}
consequently, $\overline{P(\vr,\vt)\vr}=\overline{P(\vr,\vt)}\vr$ and
up to a subsequence $\varepsilon \to 0^+$
\begin{equation}\label{11.s15}
 \vre \to \vr \mbox{ \ \ strongly in } L^q(\Omega) 
\mbox{ \ \ for any \ \ } q<\infty.
\end{equation}

\end{lemma}

Recall that from Lemma \ref{lt11.3} and due to the strong convergence of the temperature it follows
\begin{equation*}\label{11.h1}
P(\vre, \vte) \to p(\vr,\vt)\mbox{ \ \  strongly in } L^2(\Omega),
\end{equation*}
hence (\ref{11.e26}) implies
\begin{equation}\label{11.h2}
\Div \vue \to \Div \vu \mbox{  \ \ \ \ strongly in } L^2(\Omega).
\end{equation}
Additionally, we have already proved that
\begin{equation}\label{11.h3}
\rot \vue \to \rot \vu \mbox{ \ \ \ \   strongly in } L^2(\Omega;\R^3),
\end{equation}
since we observed that the vorticity can be written as sum of two  parts, one bounded in $W^{1,q}(\Omega;\R^3)$, i.e. $\vcg{\omega}_\varepsilon^0 + \vcg{\omega}_\varepsilon^1$, and the other one going 
strongly to zero 
in $L^2(\Omega;\R^3)$, i.e. $\vcg{\omega}_\varepsilon^2$. 
Whence
\begin{equation}\label{11.h5}
\tn{S}(\nabla \vue):\nabla \vue \to \tn{S}(\vu) :\nabla \vu \mbox{ \ \ \ strongly  at least in } L^1(\Omega)
\end{equation}
which finishes the proof of Theorem \ref{t11.1}.

\begin{remark} \label{r xx}Similar results as above in the case of the two-dimensional flow can be found in the paper \cite{PePo_CMUC}. The existence of weak solutions with similar properties as in Theorem \ref{t11.1} was proved there for $\gamma >2$ and $m=l+1> \frac{\gamma-1}{\gamma-2}$.
\end{remark}

\section{Weak solutions in two space dimensions}

We consider our system of equations (\ref{1.1})--(\ref{1.3}) with the boundary conditions (\ref{1.4})--(\ref{1.5}) and the total mass \eqref{1.5a} in a bounded domain $\Omega \subset \R^2$. We assume the viscous part of the stress tensor in the form \eqref{1.6} $(N=2)$ with \eqref{1.7} for $\alpha=1$ and the heat flux in the form \eqref{1.8} with \eqref{1.9}. Moreover, we take $L=const$ in \eqref{1.5}. We assume the pressure law in the form \eqref{1.18a} or for $\gamma=1$ we take 
\begin{equation} \label{8.1.12}
p = p(\vr,\vt) = \vr \vt + \frac{\vr^2}{\vr + 1} \ln^\alpha(1+\vr)
\end{equation}
with $\alpha >0$, the corresponding specific internal energy is
\begin{equation} \label{8.1.16}
e = e(\vr,\vt) = \frac{\ln^{\alpha+1}(1+\vr)}{\alpha+1} + c_v \vt, \qquad c_v=const>0,
\end{equation} 
and the specific entropy is
\begin{equation} \label{8.1.17}
s(\vr, \vt) = \ln \frac{\vt^{c_v}}{\vr} + s_0.
\end{equation}
We consider weak solutions to the problem above defined similarly as in Definition \ref{d 1.1} with the corresponding modifications for the pressure law \eqref{8.1.12}. 
This problem was studied in \cite{NoPo_AppMa} for both \eqref{1.18a} and \eqref{8.1.12} and the result for the latter was improved in \cite{Po_JPDE}. More precisely, we have the following results

\begin{theorem}[Novotn\'y, Pokorn\,y, 20011; Pokorn\'y, 2011] \label{t 8.1}
Let $\Omega \in C^2$ be a bounded domain in $\R^2$, $\vc{f} \in
L^\infty(\Omega;\R^2)$, $\Theta_0 \geq K_0 >0$ a.e. on $\partial
\Omega$, $\Theta_0 \in L^1(\partial \Omega)$, $L>0$. \newline
(i) Let $\gamma >1$, $m > 0$.
Then there exists a weak  solution 
to our problem with the pressure law \eqref{1.18a}. \newline
(ii) Let $\alpha >1$ and $\alpha \geq \max \{1,\frac 1m\}$, $m>0$. Then there exists a weak solution to our problem with the pressure law \eqref{8.1.12}. \newline
Moreover, $(\vr,\vu)$ extended by zero outside of $\Omega$, is a renormalized solution 
to the continuity equation.
\end{theorem}

As the proof for $\gamma>1$ is easy, we only refer to \cite{NoPo_AppMa} and consider the pressure law \eqref{8.1.12}. We need  to work here with a class of Orlicz spaces. We therefore recall some of their properties, referring  to \cite{KuJoFu_Book} or \cite{Ma_Book} for further details.

Let $\Phi$ be the Young function.  We denote by $E_\Phi(\Omega)$ the set of all measurable functions $u$ such that
$$
\intO{\Phi(|u(x)|)} <+\infty,
$$
and by $L_\Phi(\Omega)$ the set of all measurable functions $u$ such that the Luxemburg norm
$$
\|u\|_\Phi = \inf\Big\{k>0; \int_\Omega \Phi\Big(\frac 1k |u(x)|\Big) dx \leq 1\Big\} <+\infty.
$$

We say that $\Phi$ satisfies the $\Delta_2$-condition if there exist $k>0$ and $c\geq 0$ such that
$$
\Phi(2t) \leq k \Phi(t) \qquad \forall t \geq c.
$$
If $c=0$, we speak about the global $\Delta_2$-condition. Note that we have for all $u \in E_\Phi(\Omega)$
\begin{equation*} 
\|u\|_\Phi \leq \intO{\Phi(|u(x)|)} + 1,
\end{equation*}
while $E_\Phi(\Omega) = L_\Phi(\Omega)$ only if $\Phi$ fulfills the $\Delta_2$-condition.
For $\alpha \geq 0$ and $\beta \geq 1$ we denote by $\LL{\beta}{\alpha}(\Omega)$ the Orlicz spaces generated by $\Phi(z) = z^\beta \ln^\alpha(1+ z)$. In the range mentioned above $z^\beta \ln^\alpha(1+z)$  fulfills the global $\Delta_2$-condition. 
Recall that the complementary function to $z \ln^\alpha(1+z)$ behaves as ${\rm e}^{z^{1/\alpha}}$; however, this function does not satisfy the $\Delta_2$-condition.  We denote by $E_{{\rm e}(1/\alpha)}(\Omega)$  and $L_{{\rm e}(1/\alpha)}(\Omega)$ the corresponding sets of functions. 

It is well-known that $W^{1,2}(\Omega) \hookrightarrow L_{{\rm e}^{z^2}-1}(\Omega)$, and thus
\begin{equation} \label{8.2.2}
\|u\|_{{\rm e}(2)} \leq C (\|u\|_{1,2} +1).
\end{equation}
Further, the generalized H\"older inequality yields
\begin{equation} \label {8.2.3}
\|uv\|_1 \leq \|u\|_{z \ln^\alpha(1+z)} \|v\|_{{\rm e}(1/\alpha)}
\end{equation}
as well as
\begin{equation} \label{8.2.4}
\|uv\|_{z \ln^\alpha (1+z)} \leq C \|u\|_{z^p \ln^\alpha(1+z)} \|v\|_{z^{p'} \ln^\alpha (1+z)},
\end{equation}
for any $\alpha > 0$ and $\frac 1p + \frac{1}{p'} =1$, $1<p,p'<\infty$. The definition of the Luxemburg norm immediately yields for $\beta \geq 1$, $\alpha \geq 0$
\begin{equation} \label{8.2.5}
\|u\|_{z^\beta \ln^\alpha(1+z)} \leq \Big(1+ \intO{|u(x)|^\beta \ln^\alpha(1+ |u(x)|)} \Big)^{\frac 1\beta}
\end{equation}
as well as for $\delta >0$
$$
\||u|^\delta\|_{\Phi(z)} = \|u\|_{\Phi(z^\delta)}^\delta;
$$
hence, especially for $\delta >0$
\begin{equation} \label{8.2.6}
\||u|^\delta\|_{{\rm e}(\alpha)} \leq C \big(\|u\|^\delta_{{\rm e}(\delta \alpha)} + 1\big),
\end{equation}
and for $\delta \geq 1$
\begin{equation} \label{8.2.7}
\||u|^\delta\|_{z \ln^\alpha(1+z)} \leq C(\delta) \big(\|u\|^\delta_{z^\delta \ln^\alpha(1+z)} + 1\big).
\end{equation}   

Finally, let us consider the Bogovskii operator, i.e. the solution operator to \eqref{2.3}
for $f$ with zero mean value. For  Orlicz spaces such that the Young function $\Phi$ satisfies the global $\Delta_2$-condition and for certain $\gamma \in (0,1)$ the function $\Phi^\gamma $ is quasiconvex, we get a similar result as for the $L^p$-spaces, see \cite{Vo_AcPa}. Hence, especially for $\alpha \geq 0$ and $\beta > 1$ we have (provided $\intO{f} =0$) the existence of a solution to \eqref{2.3} such that
\begin{equation} \label{8.2.10}
\||\nabla \vcg{\varphi}|\|_{z^\beta \ln^\alpha(1+z)} \leq C \|f\|_{z^\beta \ln^\alpha(1+z)}.
\end{equation}

To construct a weak solution to our problem,
we use the same approximation scheme as in the three space dimensions, i.e. we have (\ref{2.31})--(\ref{2.34}) for any $\delta >0$. As in three space dimensions, we get  
\begin{multline} \label{8.3.5}
\|\vud\|_{1,2} + \|\nabla \vtd^{\frac m2}\|_2 + \|\nabla \ln \vtd\|_2 + \|\vtd^{-1}\|_{1,\partial \Omega} \\
+ \delta \big(\|\nabla \vtd^{\frac B2}\|_2^2 + \|\nabla \vtd^{-\frac 12}\|_2^2 + \|\vtd\|_{r}^{B-2} + \|\vtd^{-2}\|_1\big) \leq C,
\end{multline}
$r<\infty$, arbitrary, and 
\begin{equation} \label{8.3.6}
\|\vtd^{\frac m2}\|^{\frac 2m}_{1,2} \leq C\Big(1+ \|\vtd\|_{1,\partial \Omega} + \|\nabla \vtd^\frac m2\|_2^{\frac 2m}\Big) \leq C\Big(1+
\Big|\intO{\vrd \vc{f} \cdot \vud}\Big|\Big)
\end{equation}
with $C$ independent of $\delta$.

As it is the case in three space dimensions, it is more difficult to prove the estimates for the density. We use Lemma \ref{l 4.2} with $f = \vrd^s-\frac{1}{|\Omega|}\intO{\vrd^s}$ for some $0<s<1$ and use the corresponding $\vcg{\varphi}$ as test function in \eqref{2.32}. It reads
\begin{equation} \label{8.3.8}
\begin{array}{c}
\displaystyle \intO {\Big(\frac{\vrd^{2+s}}{1+\vrd} \ln^\alpha(1+\vrd)  + \vrd^{1+s}  \vtd\Big)} + \delta \intO {\big(\vrd^{\beta+s} + \vrd^{2+s}\big)} \\
=  \displaystyle \frac{1}{|\Omega|}\intO{\vrd^s} \intO {\Big(\frac{\vrd^{2}}{\vrd+1} \ln^\alpha(1+\vrd) + \vrd \vtd + \delta(\vrd^\beta + \vrd^2)\Big)}   \\
\displaystyle -  \intO{\vrd \vc{f} \cdot \vcg{\varphi}}  +\intO {\tn{S} (\vtd,\vud):\nabla \vcg{\varphi}}-\intO {\vrd (\vud\otimes\vud) :\nabla \vcg{\varphi}}  \\
\displaystyle  = J_1 + J_2 + J_3 + J_4.
\end{array}
\end{equation}
The estimates of $J_1$ and $J_2$ are easy; hence we concentrate ourselves only on the remaining two terms. We have due to (\ref{8.2.2}), (\ref{8.2.3}) and (\ref{8.2.5}) for $\alpha \geq 0$
$$
\begin{array}{c}
\displaystyle
|J_3| \leq C \intO {(1+\vtd) |\nabla \vud| |\nabla \vcg{\varphi}|} \leq C \|\nabla \vud\|_2 \|\nabla \vcg{\varphi}\|_{\frac{1+s^2}{s}} \big(1+ \|\vtd\|_{\frac{2(1+s^2)}{(1-s)^2}}\big)\\
\displaystyle \leq C \|\vrd\|^{s}_{1+s^2} \Big(1+ \Big|\intO{\vrd \vc{f}\cdot \vud}\Big|\Big) \leq C + \frac 14 \intO{\frac{\vrd^{2+s}}{1+\vrd} \ln^\alpha(1+\vrd)}
\end{array}
$$
and the last term can be shifted to the left-hand side (l.h.s.). Note that we needed here $s<1$.
Finally, using (\ref{8.2.3})--(\ref{8.2.10}), for $\alpha >1$,
$$
\begin{array}{c}
\displaystyle
|J_4| \leq \intO{\vrd |\vud|^2 |\nabla \vcg{\varphi}|} \leq C \||\vud|^2\|_{{\rm e}(1)} \| \vrd |\nabla \vcg{\varphi}|\|_{z \ln(1+z)} \\[8pt]
\displaystyle \leq C \big(\||\vud|\|^2_{{\rm e}(2)}+1\big)   \|\vrd\|_{z^{1+s} \ln(1+z)} \| |\nabla \vcg{\varphi}|\|_{z^{\frac{1+s}{s}} \ln(1+z)}  \leq C(1+\|\vrd\|_{z^{1+s} \ln(1+z)}^{1+s}) \\[8pt]
\displaystyle \leq C \Big(1+ \intO{\vrd^{1+s}\ln(1+\vrd)}\Big) \leq C + \frac 14 \intO{\frac{\vrd^{2+s}}{1+\vrd} \ln^\alpha(1+\vrd)}
\end{array}
$$
(hint: consider separately $\vrd \leq K$ and $\vrd \geq K$ for $K$ sufficiently large). Thus we have shown the estimate
\begin{equation} \label{8.3.9}
\intO {\vrd^{1+s} \ln^\alpha (1+\vrd)} \leq C(s)
\end{equation}
with $C(s) \to +\infty$ for $s \to 1^{-}$.

\begin{remark}
{\rm In \cite{NoPo_AppMa}, the authors used the same test function with $s=1$. This leads to an $L^2$-estimate of the density (and hence the limit $(\vr,\vu)$ is immediately a renormalized solution to the continuity equation). However, this method also requires additional restriction on $\alpha$ and $m$. Note that here, we are able to get the estimates for any $m>0$ and  $\alpha >1$; nevertheless, a certain restriction on $\alpha$ in terms of $m$ appears later, when proving the strong convergence of the density.
}
\end{remark}

We can now pass to the limit in the weak formulation of the approximate system (note that we still do not know whether the density converges strongly) as in three space dimensions. The main task is to get strong convergence of the density which is based on the effective viscous flux identity and validity of the renormalized continuity equation, which is connected with the boundedness of the oscillation defect measure. As the proof is similar to the three-dimensional solutions, we will only mention steps which are different here.

First of all, we may get the effective viscous flux identity in the form
 \begin{equation} \label{8.4.7}
\begin{array}{c}
\displaystyle
\overline{p(\vr,\vt) T_k(\vr)} - \big( \mu(\vt) + \xi(\vt)\big) \overline{T_k(\vr) \Div \vu} 
\displaystyle =  \overline{p(\vr,\vt)} \, \overline{T_k(\vr)} - \big(\mu(\vt) + \xi(\vt)\big)
\overline{T_k(\vr)} \Div \vu.
\end{array}
\end{equation}   

Next, we introduce the oscillation defect measure defined in a more general context of the Orlicz spaces 
\begin{equation} \label{8.4.9a}
{\bf osc}_{\Phi}[T_k(\vrd)-T_k(\vr)] = \sup_{k \in {\mbox {\FF N}}} \limsup_{\delta \to 0^+} \|T_k(\vrd)-T_k(\vr)\|_{\Phi}.
\end{equation}
In what follows, we show that there exists $\sigma >0$ such that
\begin{equation} \label{8.4.9}
{\bf osc}_{z^2 \ln^\sigma (1+z)}[T_k(\vrd)-T_k(\vr)] < +\infty;
\end{equation}
further we verify  that this fact implies the renormalized continuity equation to be satisfied. Note that to show the latter we cannot use the approach from the book \cite{FeNo_Book} (or \cite{NoPo_JDE}) as there, it is required that $\Phi = z^{2+\sigma}$ for $\sigma>0$ which we are not able to verify here.

\begin{lemma} \label{l 8.7}
Under the assumptions of Theorem \ref{t 8.1} (particularly, for $\alpha >1$ and $\alpha \geq\frac 1m$) we have (\ref{8.4.9}).
\end{lemma}

\begin{proof}
As 
\begin{equation} \label{8.4.9b}
g(t)= \frac{t^2}{t+1} \ln^\alpha(t+1)
\end{equation}
is for $\alpha >1$ convex on $\R^+_0$, we get for $z>y\geq 0$
\begin{equation} \label{8.4.10}
\begin{array}{c}
\displaystyle
\frac{z^2}{1+z} \ln^\alpha (1+z) - \frac{y^2}{1+y} \ln^\alpha (1+y) = \int_y^z g'(t) \,{\rm d}t \geq \int_y^z g'(t-y)\, {\rm d}t  \\
\displaystyle =\frac{(z-y)^2}{1+z-y} \ln^\alpha (1+z-y) \geq \frac 12 (z-y) \ln^\alpha(1+z-y) - \ln^\alpha 2 \, \vc{1}_{\{z-y \leq 1\}}.
\end{array}
\end{equation}

Moreover,
$$
\begin{array}{c}
\displaystyle
\limsup_{\delta \to 0^+} \intO{\Big[p(\vrd,\vtd) T_k(\vrd) - \overline{p(\vr,\vt)} \, \overline{T_k(\vr)}\Big]} \\
\displaystyle
=\limsup_{\delta \to 0^+} \Big[ \intO{(g(\vrd)-g(\vr))\big(T_k(\vrd)-T_k(\vr)\big)} + \intO{\big(\vrd \vtd T_k(\vrd) - \vr \vt \overline{T_k(\vr)}\big)} \Big] \\
\displaystyle + \intO{\big(\overline{g(\vr)}-g(\vr)\big) \big(T_k(\vr) - \overline{T_k(\vr)}\big)}.
\end{array}
$$
As $z \mapsto T_k(z)$ is concave and $z \mapsto g(z)$ is convex, we have (using also Lemma \ref{l 4.5} in the second integral and the fact that $g(\cdot)$ is increasing on $\R^+_0$)
\begin{equation} \label{8.4.10aa}
\begin{array}{c}
\displaystyle
\limsup_{\delta \to 0^+} \intO{\frac{|T_k(\vrd)-T_k(\vr)|}{1+ |T_k(\vrd)-T_k(\vr)|} \ln^\alpha (1+ |T_k(\vrd)-T_k(\vr)|)} 
\\
\displaystyle
\leq \limsup_{\delta \to 0^+} \intO{(g(\vrd)-g(\vr))\big(T_k(\vrd)-T_k(\vr)\big)} \\
\displaystyle \leq
\limsup_{\delta \to 0^+} \intO{\Big[p(\vrd,\vtd) T_k(\vrd) - \overline{p(\vr,\vt)} \, \overline{T_k(\vr)}\Big]},
\end{array}
\end{equation}
and also
\begin{equation} \label{8.4.11a}
\begin{array}{c}
\displaystyle
\limsup_{\delta \to 0^+} \intO{\frac{1}{\mu(\vt) + \xi(\vt)} \frac{|T_k(\vrd)-T_k(\vr)|}{1+ |T_k(\vrd)-T_k(\vr)|} \ln^\alpha (1+ |T_k(\vrd)-T_k(\vr)|)} 
\\
\displaystyle \leq
\limsup_{\delta \to 0^+} \intO{\frac{1}{\mu(\vt) + \xi(\vt)}\Big[p(\vrd,\vtd) T_k(\vrd) - \overline{p(\vr,\vt)} \, \overline{T_k(\vr)}\Big]}.
\end{array}
\end{equation}
Due to (\ref{8.4.10}) and Lipschitz continuity of $T_k(\cdot)$ with Lipschitz constant 1, together with (\ref{8.2.5}), we arrive at
\begin{equation} \label{8.4.10b}
\begin{array}{c}
\displaystyle \limsup_{\delta \to 0^+} \|T_k(\vrd)-T_k(\vr)\|_{z^2 \ln^\alpha (z+1)}^2 \\
\displaystyle \leq 1 + \limsup_{\delta \to 0^+} \intO{|T_k(\vrd) -T_k(\vr)|^2 \ln^\alpha (1+ |T_k(\vrd) - T_k(\vr)|)} \\
\displaystyle \leq C\Big(1+  \limsup_{\delta \to 0^+} \intO{\Big[p(\vrd,\vtd) T_k(\vrd) - \overline{p(\vr,\vt)} \, \overline{T_k(\vr)}\Big]}\Big). 
\end{array}
\end{equation}
Denote now $G_k^\delta(x,z) = |T_k(z) -T_k(\vr(x))|^2 \ln^\alpha (1+ |T_k(z) - T_k(\vr(x))|)$. Hence, (\ref{8.4.7}) implies
$$
\overline{G_k(\cdot,\vr)} \leq C \Big((\mu(\vt) + \xi(\vt)) \big(\overline{T_k(\vr) \Div \vu} - \overline{T_k(\vr)}\Div \vu\big) +1\Big)
$$
for all $k \geq 1$. Then
\begin{equation} \label{8.4.12a}
\begin{array}{c}
\displaystyle
\intO {(1+\vt)^{-1} \overline{G_k(\cdot,\vr)}} \leq C \Big(\sup_{\delta >0} \|\Div \vud\|_{2} \limsup_{\delta \to 0^+} \|T_k(\vrd)-T_k(\vr)\|_2 +1\Big)\\ 
\displaystyle \leq C \Big(\limsup_{\delta \to 0^+} \|T_k(\vrd) -T_k(\vr)\|_2+1\Big).
\end{array}
\end{equation}
On the other hand, take $\sigma >0$, $s>1$ and compute
$$
\begin{array}{c}
\displaystyle
\intO{|T_k(\vrd)-T_k(\vr)|^2 \ln^\sigma(1+ |T_k(\vrd)-T_k(\vr)|)} \\
\displaystyle \leq \intO{|T_k(\vrd)-T_k(\vr)|^2 \ln^\sigma(1+ |T_k(\vrd)-T_k(\vr)|) (1+\vt)^{-s} (1+\vt)^s} \\
\displaystyle \leq C \|1+\vt^s\|_{e(\frac ms)} \||T_k(\vrd)-T_k(\vr)|^2 \ln^\sigma(1+ |T_k(\vrd)-T_k(\vr)|)(1+\vt)^{-s}\|_{z \ln^{\frac sm}(1+z)} \\
\displaystyle \leq C \Big(1+ \intO{|T_k(\vrd)-T_k(\vr)|^2 \ln^\alpha(1+ |T_k(\vrd)-T_k(\vr)|) (1+\vt)^{-1}}\Big)
\end{array}
$$
provided $\alpha > \frac 1m$ and $\sigma>0$, $s-1>0$ are sufficiently small with respect to $\alpha - \frac 1m$. We have shown that for a certain $\sigma >0$ it holds, due to (\ref{8.4.12a}),
$$
\begin{array}{c}
\displaystyle \limsup_{\delta \to 0^+} \intO{|T_k(\vrd)-T_k(\vr)|^2 \ln^\sigma(1+ |T_k(\vrd)-T_k(\vr)|)} \\
\displaystyle \leq C\Big(1+ \limsup_{\delta \to 0^+} \Big(\intO{|T_k(\vrd)-T_k(\vr)|^2}\Big)^{\frac 12} \Big),
\end{array}
$$
and thus
\begin{equation} \label{8.4.10a}
\limsup_{\delta \to 0^+} \intO{|T_k(\vrd)-T_k(\vr)|^2 \ln^\sigma(1+ |T_k(\vrd)-T_k(\vr)|)} \leq C <+\infty
\end{equation}
with $C$ independent of $k$.
\end{proof}

Next we have to show that (\ref{8.4.9}) is sufficient to guarantee that $(\vr,\vu)$ verifies the renormalized continuity equation. We have
\begin{lemma} \label{l 8.7a}
Under the assumptions of Theorem \ref{t 8.1}, the pair $(\vr,\vu)$ is a renormalized solution to the continuity equation.
\end{lemma}

\begin{proof}
As the proof is similar (even slightly easier) to the evolutionary case, we give only the main steps here; for details see also \cite[Lemma 4.5]{Er_M2AS}. 

First we mollify the limit form of the renormalized continuity equation
$$
\Div(\overline{T_k(\vr)} \vu) + \overline{(T_k(\vr) - T_k'(\vr)\vr) \Div \vu} = 0 \qquad \mbox{ in } {\cal D}'(\R^2)
$$
to get
\begin{equation} \label{8.5.1}
\Div\big(S_m[\overline{T_k(\vr)}] \vu\big) + S_m[\overline{(T_k(\vr) - T_k'(\vr)\vr) \Div \vu}] = r_m \qquad \mbox{ in } {\cal D}'(\R^2)
\end{equation}
with $S_m[\cdot]$ the standard mollifier and $r_m \to 0$ in $L^2(\Omega;\R)$ as $m\to \infty$ for any fixed $k\in \tn{N}$. Let $b \in C^1(\R)$ satisfy $b'(z) \equiv 0$ for all $z \in \R$ sufficiently large, say $z\geq M$. 
Next multiply (\ref{8.5.1}) by $b'\big(S_m[\overline{T_k(\vr)}]\big)$ and letting $m \to \infty$ we deduce
\begin{equation} \label{8.5.2}
\begin{array}{c}
\displaystyle
\Div\big(b(\overline{T_k(\vr)})\vu\big) + \big(b'(\overline{T_k(\vr)}) \overline{T_k(\vr)} - b(\overline{T_k(\vr)})\big)\Div \vu \\
\displaystyle= b'(\overline{T_k(\vr)}) \big[\overline{(T_k(\vr) - T_k'(\vr)\vr) \Div \vu}\big] \qquad \mbox{ in } {\cal D}'(\R^2).
\end{array}
\end{equation}

Now, exactly as in \cite[Lemma 4.5]{Er_M2AS} we may pass with $k \to \infty$, employing (\ref{8.4.9}) to get the renormalized form of the continuity equation for any $b$ as above. Note that we basically need to control $T_k(\vrd)- T_k(\vr)$ in a better space than just $L^2(\Omega)$; the logarithmic factor is enough. By suitable approximation we finally get (\ref{1.23c}) for any $b$ as in Definition \ref{d 1.3}.
\end{proof}

The last step, i.e. that  the validity of the renormalized continuity equation, the effective viscous flux identity, and estimates above imply the strong convergence of the density can be shown similarly as in three space dimensions, thus we skip it. More details can be found in \cite{Po_JPDE}.

\section{Further results}

In the last section we briefly mention some results, where the steady compressible Navier--Stokes equations are incorporated in some more general systems and 
the methods explained in the first part of this paper are used to get existence of a solution. 

\subsection{Steady flow of a compressible radiative gas}

The modelling of a radiative gas is a complex problem. We are not going into details of its modelling, more information can be found e.g. in \cite{KrNePo_ZAMP_2013} and references therein.
We consider the following system of equations in a bounded $\Omega \subset \R^3$
\begin{equation} \label{9.1.1}
\Div (\vr \vu) = 0,
\end{equation}
\begin{equation} \label{9.1.2}
\Div (\vr \vu \otimes \vu) - \Div \tn{S} + \Grad p = \vr \vc{f}-\vc{s}_F,
\end{equation}
\begin{equation} \label{9.1.3}
\Div (\vr E \vu) = \vr \vc{f} \cdot \vu - \Div (p \vu) + \Div (\tn{S} \vu) -\Div \vc{q} -s_E,
\end{equation}
\begin{equation} \label{9.1.4}
\lambda I + \vcg{\omega} \cdot \nabla_x I = S,
\end{equation}
where the last equation describes the transport of radiative intensity denoted by $I$. The r.h.s. $S$ is a given function of $I$, $\vcg{\omega}$ and $\vu$, see \cite{KrNePo_ZAMP_2013} for more details. The quantity $\vc{s}_F$ denotes the radiative flux and $s_E$ is the radiative energy.  The  viscous part of the stress tensor is taken in the form \eqref{1.6} with the viscosity coefficients as in \eqref{1.7} for $0\leq \alpha \leq 1$.
The pressure is considered in the form \eqref{1.12} and the heat flux fulfills \eqref{1.8} and \eqref{1.9}, $L$ is a bounded function (i.e. $l=0$ in \eqref{1.10}). The system is completed by the homogeneous Dirichlet boundary conditions for the velocity \eqref{1.4} and the Newton boundary condition for the heat flux \eqref{1.5}. We finally prescribe the total mass of the fluid \eqref{1.5a}.

The main result reads as follows
\begin{theorem}[Kreml, Ne\v{c}asov\'a, Pokorn\'y, 2013]\label{t 9.1}
Let $\Omega \in C^2$ be a bounded domain in $\R^3$, $\vc{f} \in L^\infty(\Omega;\R^3)$, $\Theta_0 \geq K_0 > 0$ a.e. at $\partial\Omega$, $\Theta_0 \in L^1(\partial\Omega)$, $M>0$. Moreover, let
\begin{equation}\label{9.1.6}
\begin{array}{l}
\alpha \in (0,1] \\
\displaystyle \gamma > \max\Big\{\frac{3}{2},1 + \frac{1-\alpha}{6\alpha} + \frac{1}{2}\sqrt{\frac{4(1-\alpha)}{3\alpha} + \frac{(1-\alpha)^2}{9\alpha^2}}\Big\} \\[10pt] 
\displaystyle m > \max\Big\{1-\alpha, \frac{1+\alpha}{3}, \frac{\gamma(1-\alpha)}{2\gamma-3}, \frac{\gamma(1-\alpha)^2}{3(\gamma-1)^2\alpha-\gamma(1-\alpha)},  \\
\displaystyle \phantom{m > \Big\{}\frac{1-\alpha}{6(\gamma-1)\alpha-1}, \frac{1+\alpha+\gamma(1-\alpha)}{3(\gamma-1)}\Big\}.
\end{array}
\end{equation}
Then there exists a variational entropy solution to our system. Moreover, the pair $(\vr,\vu)$ is a renormalized solution to the continuity equation.

If additionally
\begin{equation} \label{9.1.7}
\begin{array}{l}
\displaystyle \gamma > \max\Big\{\frac{5}{3}, \frac{2+\alpha}{3\alpha}\Big\} \\
\displaystyle m > \max\Big\{1, \frac{(3\gamma-1)(1-\alpha)}{3\gamma-5}, \frac{(3\gamma-1)(1-\alpha)+2}{3(\gamma-1)}, \frac{(1-\alpha)(\gamma(2-3\alpha)+\alpha)}{\alpha(6\gamma^2-9\gamma+5)-2\gamma}\Big\},
\end{array}
\end{equation}
then this solution is a weak solution.
\end{theorem}

\begin{remark}
For special values of $\alpha$ formulas (\ref{9.1.6}) and (\ref{9.1.7}) yield the following restrictions.

For $\alpha = 1$:
\begin{equation} \label{9.1.8a}
\gamma > \frac 32 \quad \text{ and }  \quad m > \max\Big\{\frac{2}{3}, \frac{2}{3(\gamma-1)}\Big\}
\end{equation}
for the variational entropy solutions, and additionally
\begin{equation} \label{9.1.8b}
\gamma > \frac 53 \quad \text{ and }  \quad m > 1
\end{equation}
for the weak solutions.

For $\alpha = \frac 12$ (physically more relevant):
\begin{equation} \label{9.1.9a}
\gamma > \frac{7+\sqrt{13}}{6}  \quad \text{ and }  \quad m > \max\Big\{\frac{1}{2}, \frac{\gamma}{4\gamma-6}, \frac{\gamma}{6\gamma^2-14\gamma+6}\Big\}
\end{equation}
for the variational entropy solutions, and additionally
\begin{equation} \label{9.1.9b}
m > \max\Big\{1, \frac{\gamma+1}{2(\gamma-1)}, \frac{3\gamma-1}{6\gamma-10}\Big\}
\end{equation}
for the weak solutions.
\end{remark}

The proof is similar to the case without radiation with two additional difficulties. One is connected with radiation, especially with compactness properties of the transport equation and we are not going to comment on this issue here, the other one is connected with the fact that for $\alpha <1$ we loose the nice structure of the a priori estimates coming from the entropy inequality and the situation becomes more complex. More precisely,  dropping the $\delta$-dependent terms (they can be treated as above), the entropy inequality \eqref{3.3} provides us only
\begin{equation} \label{9.1.10}
\|\vud\|_{1,p}^p \leq C \|\vtd\|_{3m}^{\frac{3m(2-p)}{2}},
\end{equation}
where $\vtd$ fulfills \eqref{3.4} and $p= \frac{6m}{3m+1-\alpha}$ (i.e. $p=2$ if $\alpha =1$). This complicates technically the situation, on the other hand, the values of $\alpha$ below $1$ are physically more realistic. More details can be found in \cite{KrNePo_ZAMP_2013}.

\subsection{Steady flow of chemically reacting mixtures} 

We finally review the results of the paper \cite{GiPoZa_An}, see also \cite{Zat} for the study of the isothermal case. We consider the following system of equations in $\Omega \subset \R^3$ 
\begin{equation}\label{10.1.1}
\begin{array}{c}
\Div (\vr \vu) = 0,\\
\Div (\vr \vu \otimes \vu) - \Div \tn{S} + \Grad \pi =\vr \vc{f},\\
\Div (\vr E\vu )+\Div(\pi\vu) +\Div\vc{Q}- \Div (\tn{S}\vu)=\vr\vc{f}\cdot\vu,\\
\Div (\vr Y_k \vu)+ \Div \vc{F}_{k}  =  m_k\omega_{k},\quad k\in \{1,\ldots,n\}.
\end{array}
\end{equation}
The above system describes the flow of a chemically reacting gaseous mixture of $n$-components. It is assumed that the molar masses of the components are comparable, which is assumed e.g. by a mixture of {\it isomers}. We denote by
$Y_k = \vr_k/\vr$ the mass fraction, $\vr_k$ is the density of the $k$-th constituent.

The system is completed by the boundary conditions at $\partial \Omega$
\begin{equation} \label{10.1.6}
\vu = \vc{0},
\end{equation}
\begin{equation} \label{10.1.7}
 \vc{F}_{k}\cdot\vc{n}=0,
\end{equation}
\begin{equation}\label{10.1.8}
-\vc{Q}\cdot\vc{n}+L(\vt-\Theta_{0})=0,
\end{equation} 
the given total mass
\begin{equation}\label{conserva}
\intO{\vr}=M>0,
\end{equation}
and the following assumptions on the form of: 
\begin{itemize}
\item the pressure law
\begin{equation}\label{defp}
\pi(\vr,\vt)=\pi_{c}(\vr)+\pi_{m}(\vr,\vt),
\end{equation}
with $\pi_{m}$ obeying the Boyle law 
\begin{equation}\label{mol}
\pi_{m}=\sum_{k=1}^n \vr Y_k\vt=\vr\vt
\end{equation}
and the so-called "cold" pressure
\begin{equation*}
\pi_{c}=\vr^{\gamma}, \quad \gamma>1;
\end{equation*}
the corresponding form of the specific total energy is
\begin{equation*}
E(\vr,\vu,\vt,\vr_1,\ldots,\vr_n)=\frac{1}{2}|\vu|^{2}+e(\vr,\vt,Y_1,\ldots,Y_n),
\end{equation*}
where the internal energy takes the form
\begin{equation*}
e=e_{c}(\vr)+e_{m}(\vt,Y_1,\ldots,Y_n) 
\end{equation*}
with
\begin{equation*}
e_{c}=\frac{1}{\gamma-1}\vr^{\gamma-1},\qquad\qquad e_m= \sum_{k=1}^n Y_ke_k=\vt\sum_{k=1}^n c_{vk}Y_k ,
\end{equation*}
where $c_{vk}$ is the mass constant-volume specific heat. The constant-pressure specific heat, denoted by $c_{pk}$, is related (under assumption on the equality of molar masses) to $c_{vk}$ in the following way
\begin{equation}\label{cpcv}
c_{pk}=c_{vk}+1,
\end{equation} 
and both $c_{vk}$ and $c_{pk}$ are assumed to be constant (but possibly different for each constituent).

\item the specific entropy
\begin{equation}\label{cotos}
s=\sum_{k=1}^n Y_k s_{k}
\end{equation}
with $s_k$ the specific entropy of the $k$-th constituent. The Gibbs formula has the form
\begin{equation}\label{Gibbs}
\vt \vc{D} s=\vc{D} e+\pi\vc{D}\left({\frac {1}{\vr}}\right)-\sum_{k=1}^n g_{k}\vc{D} Y_k,
\end{equation}
with the Gibbs functions
\begin{equation}\label{defg}
g_{k}=h_{k}-\vt s_{k},
\end{equation}
where $h_k=h_k(\vt)$, $s_{k}=s_{k}(\vr,\vt, Y_{k})$ denote the specific enthalpy and the specific entropy of the $k$-th species, respectively, with the following exact forms   
\begin{equation*}
h_k= c_{pk}\vt,\quad s_k= c_{vk}\log\vt-\log\vr-\log{Y_k}.
\end{equation*}
The cold pressure and the cold energy correspond to  isentropic processes, therefore using (\ref{Gibbs}) one can derive an equation for the specific entropy $s$
\begin{equation}\label{entropy}
\Div(\vr s\vu)+\Div\left( \frac{\vc{Q}}{\vt}-\sum_{k=1}^n \frac{g_{k}}{\vt}\vc{F}_{k}\right)=\sigma,
\end{equation}
where $\sigma$ is the entropy production rate
\begin{equation} \label{sigma}
\sigma=\frac{\tn{ S}(\vt, \nabla\vu):\nabla\vu}{\vt}-{\frac{\vc{Q}\cdot\Grad\vt}{\vt^{2}}}-\sum_{k=1}^n\vc{F}_{k}\cdot\nabla\left({\frac{g_{k}} {\vt}}\right)-\frac{\sum_{k=1}^n g_{k}\omega_{k}}{\vt}.
\end{equation}

\item the viscous stress tensor
\begin{equation} \label{vis_ten}
\tn {S} = \tn{S}(\vt, \nabla\vu)= \mu(\vt)\left[\nabla \vu + \nabla^T\vu -\frac{2}{3}\Div \vu \tn{I}\right]+\xi(\vt)(\Div \vu)\tn{I},
\end{equation}
with
$$\mu(\vt)\sim (1+\vt), \qquad 0\leq \xi(\vt) \leq (1+\vt)$$

\item the heat flux

\begin{equation}\label{eq:heatd}
\vc{Q}=\sum_{k=1}^n h_k \vc{F}_{k}+\vc{q},\quad \quad \vc{q}=-\kappa(\vt)\nabla\vt,
\end{equation}
where $\kappa=\kappa(\vt) \sim (1+\vt^m)$ is the thermal conductivity coefficient 

\item the diffusion flux
\begin{equation}
\vc{F}_{k}=-Y_k\sum_{l=1}^n D_{kl}\Grad Y_l,
\label{eq:diff1}
\end{equation}
where $D_{kl}=D_{kl}(\vt,Y_1,\ldots,Y_n)$, $k,l=1,\ldots,n$ are the multicomponent diffusion coefficients;
%
we consider
	\begin{equation}\label{10.prop}
	\begin{gathered}
		D=D^t,\quad
		N(D)=\R  \vec{Y},\quad
		R(D)={\vec{Y}}^{\bot},\\
		D \quad\text{ is positive semidefinite over } \R^n,
	\end{gathered}
	\end{equation}
where we assumed that $\vec{Y}=(Y_1,\ldots,Y_n)^t>0$ and
$N(D)$ denotes the nullspace of matrix $D$, $R(D)$ its range, ${ U}=(1,\ldots,1)^{t}$ and ${U}^{\bot}$ denotes the orthogonal complement of $\R{U}$.
Furthermore, we assume that   the matrix $D$ is homogeneous of a non-negative order with respect to $Y_1,\ldots,Y_n$ and that $D_{ij}$ are differentiable functions of $\vt,Y_1,\ldots,Y_n$ for any $i,j\in\{1,\ldots,n\}$ such that 
$$
|D_{ij}(\vt,\vec{Y})| \leq C(\vec{Y}) (1+\vt^a)
$$
for some $a\geq 0$. 

\item the species production rates
$$\omega_k=\omega_{k}(\vr,\vt,Y_1,\ldots,Y_n)$$
are smooth bounded functions of their variables such that 
\begin{equation}\label{wform}
		\omega_{k}(\vr,\vt,Y_1,\ldots,Y_n)\geq 0\quad{\rm whenever}\ \ Y_{k}=0.
	\end{equation}
Next, in accordance with the second law of thermodynamics we assume that
	\begin{equation}\label{admiss}
	-{\sum_{k=1}^n{g_k\omega_k}}\geq 0,
	\end{equation}
where $g_k$ are specified in \eqref{defg}. Note that thanks to this inequality and properties of $D_{kl}$, together with 
\eqref{vis_ten} and \eqref{eq:heatd} yield that the entropy production rate defined in \eqref{sigma} is non-negative.

\end{itemize}

We consider weak only weak solutions defined in the standard way. We have the following result 

\begin{theorem}[Giovangigli, Pokorn\'y, Zatorska, 2016] \label{T1}
Let $\gamma > \frac 53$, $M>0$, $m > 1$, $a < \frac{3m-2}{2}$. Let $\Omega \in C^2$ be a bounded domain in $\R^3$. Then there exists at least one  weak solution to our problem above.  Moreover, $(\vr,\vu)$ is the renormalized solution to the continuity equation. 
\end{theorem}

The proof is based on a complicated approximation procedure, where the most difficult part is to deduce the correct form of the approximate entropy inequality and to estimate  all additional terms that appear there due to approximation.  The reason for the bounds $\gamma>5/3$ and $m>1$ is, roughly speaking, the convective term in the total energy balance. To reduce the assumptions on $\gamma $ and $m$ (both using improved estimates of the pressure and consider variation entropy solutions) is the work in the progress.

\section{Conclusions}

The known existence results for the steady compressible Navier--Stokes--Fourier equations for large data were reviewed. It is well known that strong solutions may not exist. Therefore two different notions of a solution are proposed: the weak and the variational entropy one, where the former includes the weak formulation of the total energy balance while in the latter, the total energy balance is replaced by the weak formulation of the entropy inequality and the global total energy balance. More details in the existence proof for the three-dimensional flows were presented,  subject to either the homogeneous Dirichlet or the Navier boundary conditions for the velocity. The main ideas behind the proof of existence of more regular solution in the case of the Navier boundary conditions and $\gamma >3$ were explained. In this case even the internal energy balance is fulfilled. The two-dimensional flows for $\gamma$ almost one were also studied. Finally, few results for more complex models were presented, where the Navier--Stokes-Fourier system is combined with other equations.

\bigskip 

{\bf Acknowledgments.} The work of P.B. Mucha has been partly supported by Polish NCN grant No  2014/13/B/ST1/03094.

%

\begin{thebibliography}{8.}
%
%
%

\bibitem{AxPo}
Axmann, \v{S}.,  Pokorn\'y, M.: Time-periodic solutions to the full Navier-Stokes-Fourier system with radiation on the boundary. 
J. Math. Anal. Appl. 428(1), 414--444 (2015)


\bibitem{BrNo_CMUC} B\v{r}ezina, J., Novotn\'y, A.: 
On weak solutions of steady Navier--Stokes equations for monoatomoc gas. Comment. Math. Univ. Carol. 49, 611--632 (2008)  

\bibitem{DaMu}
Danchin, R., Mucha, P.B.: The divergence equation in rough spaces. J. Math. Anal. Appl. 386(1), 9--31 (2012) 


\bibitem{EGH_Book} Elizier, S., Ghatak, A., Hora, H.:  An introduction to equations of states, theory and applications. Cambridge University Press, Cambridge (1996)

\bibitem{Er_M2AS} Erban, R.: On the existence of solutions to the Navier--Styokes equations of a two-dimensional compressible flow. Math. Methods Appl. Sci. 26, 489--517 (2003) 

\bibitem{Ev_Book}  Evans, L.C.: (1998) Partial Differential Equations. Graduate Studies in Math. 19, Amer. Math. Soc., Providence

\bibitem{Fe_Book} Feireisl, E.:   Dynamics of viscous compressible fluids. Oxford University Press, Oxford (2004)

\bibitem{FMNP}
Feireisl,  E., Mucha, P.B.,  Novotn\'y, A., Pokorn\'y, M.: 
Time-periodic solutions to the full Navier-Stokes-Fourier system. Arch. Ration. Mech. Anal. 204(3), 745--786 (2012) 


\bibitem{FeNo_Book} Feireisl, E., Novotn\'y, A.:  A Singular Limits in Thermodynamics of Viscous Fluids.
Advances in Mathematical Fluid Mechanics, Birkh\"auser, Basel (2009)

\bibitem{FeNo_Stab}
Feireisl, E., Pra\v{z}\'ak, D.: Asymptotic behavior of dynamical systems in fluid mechanics. AIMS, Springfield (2010) 

\bibitem{FrStWe_JMPA} Frehse, J., Steinhauer, M., Weigant, W.: The Dirichlet problem for steady viscous compressible flow in 3-D. J. Math. Pures Appl. 97, 85--97 (2009)

\bibitem{GiPoZa_An} Giovangigli, V., Pokorn\'y, M., Zatorska, E.:  On the steady flow of reactive gaseous mixture.  Analysis (to appear) (2016)

\bibitem{JeNo_JMPA} Jessl\'e, D., Novotn\'y, A.:  Existence of renormalized weak solutions to the steady equations describing compressible fluids in barotropic regime. J. Math. Pures Appl. 99, 280--296 (2013)

\bibitem{JeNoPo_M3AS} Jessl\'e, D., Novotn\'y, A., Pokorn\'y, M.:  Steady Navier--Stokes--Fourier system with slip boundary conditions. Math. Models Methods Appl. Sci 24, 751--781 (2013)

\bibitem{KrPo_Prep} Kreml, O., Pokorn\'y, M.:  In preparation (2016)

\bibitem {KrNePo_ZAMP_2013} Kreml, O., Ne\v{c}asov\'a, \v{S}., Pokorn\'y, M.:  On the steady equations for compressible radiative gas, Zeitschrift f\"ur angewandte Mathematik und Physik  64, 539--571 (2013)

\bibitem{KuJoFu_Book} Kufner, A., John, O., Fu\v{c}\'\i k, S.: Function spaces. Academia, Praha (1977) 

\bibitem{Li_Book2} Lions, P.-L.: Mathematical Topics in Fluid Dynamics, Vol.2: Compressible Models. Oxford Science Publication, Oxford (1998) 

\bibitem{Ma_Book} Maligranda, L.: Orlicz Spaces and Interpolation. Campinas SP, Brasil (1989) 

\bibitem {Mu_JDE} Mucha, P.B.: On cylindrical symmetric flows through pipe-like domains. J. Diff. Eq. 201, 304--323 (2004) 

\bibitem{MuPo_Nonlinearity} Mucha, P.B., Pokorn\'y, M.: On a new approach to the issue of
existence and regularity for the steady compressible Navier--Stokes
equations. Nonlinearity 19, 1747--1768  (2006)  

\bibitem  {MuPo_CMP}  Mucha, P.B.,  Pokorn\'y, M.:  On the steady compressible Navier--Stokes--Fourier system.
Comm. Math. Phys. 288: 349--377 (2009)

\bibitem{MuPo_M3AS} Mucha, P.B.,  Pokorn\'y, M.:  Weak solutions to equations of steady compressible heat conducting fluids. Math. Models Methods Appl. Sci. 20, 1--29 (2010)

\bibitem{MuPoJMFM}
Mucha, P.B., Pokorn\'y, M.:  The rot-div system in exterior domains. J. Math. Fluid Mech. 16(4), 701--720 (2010)

\bibitem{MuRa}
Mucha, P.B., Rautmann, R.: Convergence of Rothe's scheme for the Navier-Stokes equations with slip conditions in 2D domains. 
ZAMM Z. Angew. Math. Mech. 86(9), 691--701 (2006) 

\bibitem {NoPo_JDE} Novotn\'y, A., Pokorn\'y, M.:  Steady compressible Navier--Stokes--Fourier system for monoatomic gas and its generalizations. J. Differential Equations 251, 270--315 (2011)

\bibitem {NoPo_SIMA} Novotn\'y, A., Pokorn\'y, M.:  Weak and variational solutions to steady equations for compressible heat conducting fluids. SIAM J. Math. Anal. 43, 270--315 (2011) 

\bibitem {NoPo_AppMa} Novotn\'y, A., Pokorn\'y, M.:  Weak solutions for steady compressible Navier-Stokes-Fourier system in two space dimensions. Appl. Math. 56, 137--160 (2011)

\bibitem {NoSt_Book} Novotn\'y, A., Stra\v{s}kraba, I.: Introduction to the Mathematical Theory of Compressible Flow, Oxford University Press, Oxford (2004) 

\bibitem{PaPil}
Panasenko, G.,  Pileckas, K.: Divergence equation in thin-tube structures. Appl. Anal. 94(7), 1450--1459 (2015) 

\bibitem{PePo_CMUC} Pecharov\'a, P., Pokorn\'y, M.:  Steady compressible Navier-Stokes-Fourier system in two space dimensions. Comment. Math. Univ. Carolin. 51, 653--679 (2010)

\bibitem{PlSo_JMFM} Plotnikov, P.I., Sokolowski, J.: On compactness, domain dependence and existence of steady state solutions to compressible isothermal Navier--Stokes equations. J. Math. Fluid Mech. 7, 529--573 (2005) 

\bibitem{PlWe_2015} Plotnikov, P.I., Weigant, W.:  Steady 3D viscous compressible flows with adiabatic exponent $\gamma \in (1,\infty)$. J. Math. Pures Appl. 104, 58--82 (2015) 

\bibitem{Po_JPDE} Pokorn\'y, M.: On the steady solutions to a model of compressible heat conducting fluid in two space dimensions. J. Part. Diff. Eq. 24 (4), 334--350 (2011) 

\bibitem {MuPo_DCDS}  Pokorn\'y, M., Mucha, P.B.: 3D steady compressible Navier--Stokes equations.  Cont. Discr. Dyn. Systems S 1: 151--163 (2008) 

\bibitem{So_LOMI} Solonnikov, V.A.: Overdetermined elliptic boundary value problems. Zap. Nauch. Sem. LOMI 21, 112--158 (1971) 

\bibitem{Vo_AcPa} Vod\'ak, R.:  The problem $\Div \vc{v} = f$ and singular integrals in Orlicz spaces. Acta Univ. Olomuc. Fac. Rerum Natur. Math. 41, 161--173 (2002)

\bibitem{Zat} Zatorska, E.:  On the steady flow of a multicomponent, compressible, chemically reacting gas. Nonlinearity 24, 3267--3278 (2011)

\bibitem{Zat1} Zatorska, E.:  Analysis of semidiscretization of the compressible Navier-Stokes equations. J. Math. Anal. Appl. 386 (2), 559--580
 (2012)

\bibitem{Zhong_15} Zhong, X.: Weak solutions to the three-dimensional steady full compressible Navier-Stokes system. 
Nonlinear Anal. 127, 71--93 (2015) 

\bibitem {Zi_1989} Ziemer, W.P.:  Weakly Differentiable Functions. Springer Verlag, New York (1989)
\end{thebibliography}
%


\end{document}